\documentclass[12pt,a4paper]{article}
\usepackage[paper=a4paper,left=25mm,right=25mm,top=25mm,bottom=20mm]{geometry}
\usepackage{setspace}
\usepackage[utf8]{inputenc}
\usepackage{graphicx}
\usepackage{hyperref}
\usepackage{booktabs}
\usepackage[english,onelanguage]{algorithm2e}
\usepackage{listings}
\usepackage{physics}
\usepackage{amsmath}
\usepackage{amssymb}
\usepackage{bm}
\usepackage[noadjust]{cite}
\usepackage{mathtools}
\usepackage{afterpage}
\usepackage{amsthm}
\usepackage{mathrsfs}
\usepackage{tikz-cd}
\usepackage[title]{appendix}
\usepackage{dsfont}
\usepackage{oldgerm}
\usepackage{comment}

\usepackage{thm-restate}

\theoremstyle{plain}
\newtheorem{theorem}{Theorem}[subsection]
\newtheorem{corollary}[theorem]{Corollary}
\newtheorem{lemma}[theorem]{Lemma}
\newtheorem{proposition}[theorem]{Proposition}

\theoremstyle{definition}
\newtheorem{definition}[theorem]{Definition}
\newtheorem{example}[theorem]{Example}
\newtheorem{remark}[theorem]{Remark}
\newtheorem{remarks}[theorem]{Remarks}

\newtheorem{convention}[theorem]{Convention}
\newtheorem*{theorem*}{Theorem}

\newenvironment{tothereader}
{\par\bigskip
\begin{center}
\bfseries To the reader
\end{center}
\begin{quote}}
{\end{quote}\par\bigskip}

\renewenvironment{tothereader}
  {\small\noindent\textbf{To the reader.}\ }
  {}

\newcommand{\tn}{\textnormal}

\newcommand{\R}{\mathbb{R}}
\newcommand{\N}{\mathbb{N}}
\newcommand{\K}{\mathbb{K}}
\newcommand{\nnorm}[2]{\norm{#1}_{#2}}
\newcommand{\tens}{\otimes}
\newcommand{\del}{\partial}
\newcommand{\scal}[2]{\left\langle #1,#2\right\rangle}

\newcommand{\om}{\omega}
\newcommand{\VFs}{\mathfrak{X}(M)}
\newcommand{\Forms}[1]{\Omega^{#1}(M)}
\newcommand{\ClForms}[1]{\Forms{#1}_{\mathrm{cl}}}
\newcommand{\ExForms}[1]{\Forms{#1}_{\mathrm{ex}}}
\newcommand{\LieD}{\mathcal{L}}
\newcommand{\Diff}{\mathrm{Diff}(M)}
\newcommand{\SDiff}{\mathrm{Diff}(M)_{[H]}}
\newcommand{\Auteq}{\mathrm{Auteq}(M,H)}
\newcommand{\auteq}{\mathfrak{auteq}(M,H)}
\newcommand{\eh}[1]{\mathrm{exp}^h(#1)}
\newcommand{\exh}{\mathrm{exp}^h}
\newcommand{\Id}{\mathrm{Id}}

\newcommand{\Severa}{\v{S}evera}
\newcommand{\GenTM}{\mathbb TM}
\newcommand{\Ev}{\mathrm{Ev}}

\newcommand{\OP}{\mathrm{OP}}

\newcommand{\Hcal}{\mathcal{H}}
\newcommand{\Mcal}{\mathcal{M}}
\newcommand{\Ncal}{\mathcal{N}}
\newcommand{\Scal}{\mathcal{S}}
\newcommand{\Gcal}{\mathcal{G}}
\newcommand{\TMcal}{\mathcal{TM}}
\newcommand{\Ucal}{\mathcal{U}}
\newcommand{\Vcal}{\mathcal{V}}
\newcommand{\Pcal}{\mathcal{P}}
\newcommand{\gfrak}{\mathfrak{g}}
\newcommand{\tensf}[2]{\mathcal{T}^{#1}_{#2}(M)}
\newcommand{\Jcal}{\mathcal{J}}

\newcommand{\RNum}[1]{\uppercase\expandafter{\romannumeral #1\relax}}

\renewcommand{\d}{\mathrm{d}}

\renewenvironment{abstract}
  {\small\noindent\textbf{Abstract.}\ }
  {}

\title{The group of autoequivalences of an exact Courant algebroid as a tame Fréchet Lie group}
\author{Jan Niklas Heck}
\date{\today}

\usepackage{graphics}

\hypersetup{
  pdftitle={AutoEqTameFrechet2026},
  pdfauthor={Jan Niklas Heck}
}

\begin{document}

\maketitle

\begin{abstract}
The aim of this manuscript is to show that the group of autoequivalences of an exact Courant algebroid over a compact base manifold is a tame Fréchet Lie group.
Moreover, we compute its Lie algebra.
Furthermore, we show that the space of generalized almost complex structures is a tame Fréchet manifold and that the canonical action of the group of autoequivalences on the space of generalized almost complex structures is smooth tame.
\end{abstract}
\newline

\begin{tothereader}
The results of this manuscript were developed as part of my Master's thesis \cite{MyMastThesis} under the supervision of Prof.~Dr.~Cortés at the University of Hamburg in 2025.
Some minor corrections and adjustments have been added since its completion. 
\end{tothereader}

\tableofcontents

\section{Introduction and outline}

The field of generalized complex geometry was introduced by Hitchin in \cite{Hitchin_2003} and developed further by Gualtieri in \cite{gualtieri} and \cite{gualtieriAnn}. Instead of studying (almost) complex structures on the tangent bundle, one considers generalized (almost) complex structures. These are almost complex structures on the generalized tangent bundle $\GenTM\coloneqq TM\oplus T^*M$. The vector bundle $\GenTM$ is equipped with a bracket on its sections determined by a closed three-form $H$, a standard scalar product, and the canonical projection $\pi:\GenTM\to TM$ as an anchor map. Together, these structures make $\GenTM$ an exact Courant algebroid. Generalized almost complex structures are required to be compatible with the scalar product, while generalized complex structures additionally satisfy an integrability condition involving the bracket.
In particular, the fields of complex and symplectic geometry are encompassed by generalized complex geometry.

In \cite{gualtieri} Gualtieri studied deformations of generalized complex structures of exact Courant algebroids over compact manifolds. He shows that the moduli space of generalized almost complex structures is locally controlled by the cohomology of a certain differential graded Lie algebra. Furthermore, he showed the existence of an analytic obstruction map determining the integrable deformed structures. His proof is based on earlier work by Kuranishi, who showed an analogous result for the deformations of compact complex manifolds in \cite{Kuranishi}. Kuranishi's work itself was based on earlier work by Kodaira and Spencer of which a comprehensive discussion can be found in \cite{kodaira1986complex}.
Both Kuranishi and Gualtieri work on the Sobolev scale and use the implicit function theorem on Banach spaces to prove their results. However, certain details about its applicability are omitted.
In light of further research aiming to generalize Gualtieri's results it will be helpful to revisit his proof in Hamilton's Nash-Moser framework developed in \cite{hamilton}. In his paper Hamilton develops the Nash-Moser implicit function theorem which seems to be the correct tool for possibly generalizing Gualtieri's results. Indeed, Cortés and Pilatus proved Kuranishi's deformation theorem for complex structures using Hamilton's Nash-Moser framework in \cite{cortesPilatusCxStr}.

While Kuranishi deems complex structures equivalent if they can be related by a small diffeomorphism, Gualtieri considers two generalized complex structures equivalent if they can be related by an autoequivalence of the exact Courant algebroid, that is, a structure preserving vector bundle automorphism of the exact Courant algebroid that covers a diffeomorphism of the base manifold to itself. 
It is known that the diffeomorphism group of a compact manifold is a tame Fréchet Lie group, that is, a Lie group in Hamilton's Nash-Moser framework; see \cite{hamilton}. This fact is used by Cortés and Pilatus in \cite{cortesPilatusCxStr} to prove Kuranishi's deformation theorem for complex structures. This naturally raises the question whether the group of autoequivalences of an exact Courant algebroid over a compact base manifold is also a tame Fréchet Lie group. Establishing this will be the main goal of this thesis.

We note that it was shown by Rubio and Tipler in \cite{RubioTipler} that the autoequivalences of an exact Courant algebroid over a compact base manifold form a so-called ILH Lie group.

Fréchet spaces arise naturally when studying geometry as they model spaces of smooth maps.  Hamilton discovered in \cite{hamilton} that, although in general there is neither an inverse nor an implicit function theorem on Fréchet spaces, the Nash-Moser framework of tame Fréchet spaces and smooth tame maps provides enough structure for the Nash-Moser inverse and implicit function theorem. In his paper Hamilton also showed that many mapping spaces that arise naturally in geometry, such as spaces of sections of fiber bundles over a compact base or the diffeomorphism group of a compact manifold, are tame Fréchet manifolds. The existence of an inverse function theorem and an implicit function theorem in the Nash-Moser category, hence, makes thinking of such mapping spaces as tame Fréchet manifolds a useful tool in geometry. 
More broadly, there exists a theory of locally convex manifolds, and Lie groups and their associated Lie algebras, a comprehensive overview of which can be found in \cite{InfDimLieneeb} and \cite{schmeding_2022}.

Many analytic tools when working with spaces of sections of vector bundles are obtained by employing the theory of Sobolev spaces, an introduction to which can be found for example in \cite{Wells07}. This provides a particularly elegant framework for the Hodge theory of elliptic complexes and gives rise to a rich theory of analytic tools that will prove very useful throughout the thesis. We will show that the necessary analytic tools and operators translate naturally to the Nash-Moser framework.

The main result of this thesis is that the group of autoequivalences of an exact Courant algebroid with a compact base manifold is a tame Fréchet Lie group. We note that compactness of the base manifold is necessary in order to work in the Nash-Moser framework and apply Hodge theory. We will denote the autoequivalences of an exact Courant algebroid over a compact base manifold $M$ determined by a closed three-form $H\in \ClForms{3}$ by $\Auteq$ and prove the following theorem.

\begin{theorem}
\label{MainThm:AuteqMfdAndLieGr}
Let $M$ be a compact manifold and let $H\in \ClForms{3}$. Then the group $\Auteq$ can be equipped with the structure of a tame Fréchet manifold turning it into a tame Fréchet Lie group.
The tame Fréchet Lie group structure depends only on the de Rham class of $H$ up to isomorphism of tame Fréchet Lie groups.
\end{theorem}

Having shown that the group of autoequivalences is a tame Fréchet Lie group it is natural to compute its Lie algebra. Denoting the Lie derivative by $\LieD$ and the space of closed two-forms by $\ClForms{2}$ we prove the following result.

\begin{theorem}
\label{MainThm:LieAlg}
The Lie algebra of $\Auteq$ expressed in a standard chart centered at $\mathrm{Id}_{\GenTM}\in \Auteq$ is
$(\VFs\times \ClForms{2},[-,-])$ with the Lie bracket
\begin{align*}
    [(X,B),(Y,C)]=(\LieD_YX,\LieD_YB-\LieD_XC-(\LieD_XQ\LieD_Y-\LieD_YQ\LieD_X)H+Q(\LieD_{\LieD_XY}H))
\end{align*}
for $X,Y\in \VFs$ and $B,C\in \ClForms{2}$. 
\end{theorem}

Lastly, the group of autoequivalences of a Courant algebroid naturally acts on the space of its generalized almost complex structures. This action comes up in the deformation theory of generalized almost complex structures and so becomes worthwhile studying in the Nash-Moser context. Denoting the space of fiberwise linear generalized almost complex structures on $\GenTM$ over a compact manifold $M$ equipped with the standard scalar product by $\Jcal_M$ we prove the following additional result.

\begin{theorem}
\label{MainThm:Jcal_M}
Let $M$ be a compact manifold and $H\in \ClForms{3}$. The space $\Jcal_M$ is a fiber bundle and the space $\Gamma(\Jcal_M)$ of generalized almost complex structures on $\GenTM$ equipped with the standard scalar product is a tame Fréchet manifold.
Furthermore, the canonical left group action 
\begin{align*}
    \Auteq\times \Gamma(\Jcal_M)&\longrightarrow \Gamma(\Jcal_M)\\
    (F,J)&\longmapsto F\circ J\circ F^{-1}
\end{align*}
is a smooth tame map.
\end{theorem}

The rest of the thesis will be structured as follows.

In Section \ref{Section:ExactCourantAlgebroids} we review some of the basic theory of Courant algebroids, in particular, exact Courant algebroids. We recall the notion of autoequivalences of Courant algebroids and their explicit description in the case of exact Courant algebroids.

In Section \ref{Section:Nash-Moser Category} we review some of the theory of locally convex vector spaces and the methods of differential calculus on them. We introduce and focus on tame Fréchet spaces and smooth tame maps which together make up the Nash-Moser category.
A special focus is put on the tame Fréchet space of sections of a finite-dimensional smooth vector bundle over a compact manifold. We review the relevant theory of Sobolev spaces of smooth sections of vector bundles over compact manifolds. In particular, we transport the tools developed in Hodge theory to the context of tame Fréchet spaces. 

In Section \ref{Section:TameFrechetManifolds} we review the relevant theory of infinite-dimensional manifolds with a particular focus on infinite-dimensional Lie groups and their Lie algebras. Despite the focus of the thesis being on tame Fréchet Lie groups we formulate most of the results more generally for various types of infinite-dimensional manifolds.
We particularly focus on manifolds consisting of smooth maps between finite-dimensional manifolds. In particular we study the diffeomorphism group of a compact manifold, families of smooth maps, and show that the pullback of tensor fields along diffeomorphisms is a smooth tame operation.

In Section \ref{Section:TheGroupOfautoequivalencesOfAnExCAAsATameLieGroup} we equip the group of autoequivalences with the structure of a tame Fréchet manifold and show that, as such, it is actually a tame Fréchet Lie group.
We compute the Lie algebra of the group of autoequivalences of an exact Courant algebroid over a compact manifold.
Lastly, we introduce the notion of generalized almost complex structures and equip them with the structure of a tame Fréchet manifold in the case of a compact base manifold. Moreover, we show that in the case of exact Courant algebroids over compact base manifolds the canonical left action of the group of autoequivalences on the space of generalized almost complex structures is smooth tame.

\section*{Acknowledgements}
I want to thank Professor Cortés for supervising my Master's thesis. Moreover, I want to thank Professor Cortés and the Geomatikum team for providing an office that I could share with other students for a semester, leading to many interesting discussions.
I thank Doctor Röser for being the second examiner of this thesis and for helpful conversations.
I thank Paula Pilatus for her guidance and many very helpful discussions.

\section*{Preliminaries}
\begin{convention}
Throughout this thesis, when referring to manifolds, fiber bundles, and vector bundles, we generally mean smooth finite-dimensional manifolds and vector bundles, respectively, unless we explicitly use qualifiers such as "Banach manifold" or "tame Fréchet vector bundle". Similarly, when we call a map $B\to M$ a surjective submersion we imply that $B$ and $M$ are finite-dimensional smooth manifolds and that the map $B\to M$ is a smooth surjective submersion.
\end{convention}

\begin{convention}
Let $M$ be a compact manifold and $p:B\to M$ a surjective submersion, e.g.~a fiber bundle or a vector bundle. Then, we will denote by $\Gamma(B)$ the space of smooth sections $\{s\in C^\infty(M,B)|~p\circ s=\Id\}$.
\end{convention}

\begin{convention}
    For any manifold $M$ we denote by $\Diff$ the group of diffeomorphisms of $M$.
\end{convention}

\begin{convention}
Throughout this thesis we will denote the natural numbers including $0$ by $\N=\{0,1,2,\ldots \}$.
\end{convention}

\begin{convention}
Let $X$ be a topological space and $A\subseteq X$ a subset. Then by $\Bar{A}$ we denote the closure of $A$ in $X$.
\end{convention}

\begin{convention}
Let $M$ be a manifold, and $\tau\in \Gamma((TM)^n\tens (T^*M)^m)$ be a smooth $(n,m)$-tensor field with $n,m\in \N$. Let $X\in \VFs$ be a smooth vector field on $M$ then we denote by $\LieD_X\tau$ the Lie derivative of $\tau$ in the direction of $X$. Furthermore, if $m>0$ we denote by $\iota_X\tau$ the interior product of $\tau$ and $X$.
\end{convention}

\begin{convention}
    Let $M$ be a manifold. We fix the notation $\Forms{k}\coloneqq \Gamma(\bigwedge^k T^*M)$ as well as $\ClForms{k}$ for closed differential $k$-forms and $\ExForms{k}$ for exact differential $k$-forms.
\end{convention}

\begin{convention}
    Throughout the thesis we will sometimes perform estimates involving positive constants. These constants are not necessarily fixed and may change from one occurrence to the next. So, in an estimate of the form $d\leq C e\leq C f$ the two occurrences of $C$ may refer to two different positive constants.
\end{convention}

\section{Exact Courant algebroids}
\label{Section:ExactCourantAlgebroids}

We want to describe the group of autoequivalences of an exact Courant algebroid. In this section we briefly review some concepts relating to (exact) Courant algebroids. These form the basis of generalized (complex) geometry. We recall the group of autoequivalences of a Courant algebroid. Furthermore, we reproduce the definition of the (twisted) generalized tangent bundle, its relation to exact Courant algebroids, and discuss how exact Courant algebroids can be classified up to isomorphism by their \Severa\ class. Lastly, we discuss how one can explicitly describe the group of autoequivalences of a twisted generalized tangent bundle by diffeomorphisms and differential $2$-forms related by a certain compatibility condition.

\begin{definition}
Let $M$ be a manifold, $E\to M$ be a vector bundle, $\scal{-}{-}\in \mathrm{Sym}^2(E^*)$ be a fiberwise nondegenerate symmetric bilinear form, $[-,-]:\Gamma(E)\times \Gamma(E)\to \Gamma(E)$ be a bilinear map, and $\pi:E\to TM$ be a smooth vector bundle homomorphism for which the following equations hold:
\begin{alignat*}{3}
&\text{1.)} \quad && [u,[v,w]]=[[u,v],w]+[v,[u,w]] \quad && \text{(Jacobi identity),} \\
&\text{2.)} \quad && \pi(u)(\scal{v}{w})=\scal{[u,v]}{w}+\scal{v}{[u,w]} \text{,} \quad && \\
&\text{3.)} \quad && \scal{[u,v]+[v,u]}{w}=\pi(w)(\scal{u}{v}) \text{,} \quad &&
\end{alignat*}
for all $u,v,w\in \Gamma(E)$. Then we call $(E,[-,-],\scal{-}{-},\pi)$ a \textit{Courant algebroid}. We call $\scal{-}{-}$ the \textit{scalar product}, $[-,-]$ is called the \textit{Dorfman bracket}, and we call $\pi$ the \textit{anchor map}.
\end{definition}

\begin{definition}
Let $(E,[-,-],\scal{-}{-},\pi)$ and $(E',[-,-]',\scal{-}{-}',\pi')$ be Courant algebroids. A homomorphism of vector bundles $\phi:E\to E'$ covering $\Id$ is called a \textit{homomorphism of Courant algebroids} if $\pi'\circ \phi =\pi$, $\phi^*\scal{-}{-}'=\scal{-}{-}$ and $[\phi u,\phi v]'=\phi [u,v]$ for every $u,v\in \Gamma(\mathbb TM)$. It is called an \textit{isomorphism of Courant algebroids} if $\phi:E\to E'$ is additionally an isomorphism of vector bundles. In that case we call $E$ and $E'$ \textit{isomorphic as Courant algebroids}.
\end{definition}

\begin{definition}
Let $(E,[-,-],\scal{-}{-},\pi)$ be a Courant algebroid then it is called
\begin{enumerate}
    \item \textit{transitive} if $\pi$ has constant rank,
    \item \textit{regular} if $\pi$ is surjective, and
    \item \textit{exact} if the sequence $$0\to T^*M\xrightarrow{\pi^*}E\xrightarrow{\pi} TM\to 0$$ is exact.
\end{enumerate}
Here $\pi^*:T^*M\to E$ denotes the dual of $\pi$ as a map $T^*M\to E^*$ composed with the natural identification $E^*\cong E$ via the scalar product $\scal{-}{-}$.
\end{definition}

\begin{definition}
\label{Def:GenTanBund}
Let $M$ be a manifold and define $\mathbb{T}M\coloneqq TM\oplus T^*M$. Then, $\mathbb{T}M$ is called the \textit{generalized tangent bundle}. 

On $\mathbb{T}M$ the \textit{canonical scalar product} $\scal{-}{-}_{\mathrm{can}}\in \Gamma(\mathrm{Sym}^2(\mathbb{T}M)^*)$ is defined by
\begin{align*}
    \scal{X+\xi}{Y+\eta}_{\mathrm{can}}\coloneqq \frac{1}{2}(\xi(Y)+\eta(X))
\end{align*}
for $X,Y\in T_pM$ and $\xi,\eta\in T_p^*M$, $p\in M$.

The \textit{canonical anchor} $\pi_{\mathrm{can}}:\mathbb TM= TM\oplus T^*M\to TM$ is defined by the canonical projection.

Lastly, for every $H\in \Forms{3}$ a bracket $[-,-]_H:\Gamma(\mathbb TM)\times\Gamma(\mathbb TM)\to \Gamma(\mathbb TM)$ is defined by
\begin{align*}
    [X+\xi,Y+\eta]_H\coloneqq \LieD_X(Y+\eta)-\iota_Y d\xi + \iota_X\iota_YH
\end{align*}
where $X,Y\in \VFs$, $\xi,\eta \in \Forms{1}$. This is called the \textit{(canonical) $H$-twisted Dorfman bracket}. In the case that $H=0$ we also call it the \textit{(canonical) untwisted Dorfman bracket}.
\end{definition}

\begin{remark}
\label{Rem:DifferentH-twisted}
Note that there are varying definitions of the $H$-twisted Dorfman bracket for $H\in \Forms{3}$ in literature. The contesting definition is 
\begin{align*}
    [X+\xi,Y+\eta]'_H\coloneqq \LieD_X(Y+\eta)-\iota_Y d\xi + H(X,Y,-).
\end{align*}
It can be found, for example, in \cite{severa2001poissongeometry3formbackground}. Note that $[-,-]'_H=[-,-]_{(-H)}$. We stick with $[-,-]_H$, which is the definition that is used, for example, in \cite{gualtieriAnn} and in \cite{RubioTipler}.
\end{remark}

\begin{convention}
\label{Conv:ExCAshorthand}
We denote $(\mathbb TM,[-,-]_H,\scal{-}{-}_\mathrm{can},\pi_\mathrm{can})$ by $(\mathbb TM,H)$.
\end{convention}

\begin{theorem}[{\cite[p.~35]{gualtieri}}]
Let $M$ be a manifold and $H\in \Forms{3}$. Then $(\mathbb TM,H)$ is a Courant algebroid if and only if $H$ is closed.
\end{theorem}

\begin{definition}
An \textit{automorphism} of a Courant algebroid is simply an isomorphism of the Courant algebroid to itself.

More generally an \textit{autoequivalence} of a Courant algebroid $(E\to M,[-,-],\scal{-}{-},\pi)$ is a vector bundle isomorphism $F:E\to E$ covering a diffeomorphism $f\in \Diff$, such that $\pi \circ F = \d f \circ \pi$, $\scal{F u}{F v}=\scal{u}{v}$ for any $u,v \in E_p$, $p \in M$, and $[F s,F s']=F [s,s']$ for any $s,s'\in \Gamma(E)$. Here we define $F s\coloneqq F \circ s \circ f^{-1}$, $s\in \Gamma(E)$, in order to obtain a section of $E$.
\end{definition}

\begin{proposition}[{\cite[p.~82]{gualtieriAnn}}]
\label{B-fieldTrafoProp}
Let $M$ be a manifold, $B\in \Forms{2}$, and $H\in \ClForms{3}$ then $e^B:\GenTM\to \GenTM$ defined by $e^B(X+\xi)\coloneqq X+\xi+\iota_XB$ for $X\in \VFs$ and $\xi\in \Forms{1}$ is an isomorphism  $e^B:(\mathbb TM,H)\stackrel{\cong}{\to} (\mathbb TM,H+\d B)$.

In particular $e^B$ is an automorphism of $(\mathbb TM, H)$ if and only if $B\in \ClForms{2}$.
Moreover, every automorphism of $(\mathbb TM, H)$ is of the form $e^B$ for some $B\in \ClForms{2}$.
\end{proposition}

\begin{remark}
Throughout the literature some of the signs in the previous proposition might vary as pointed out in Remark \ref{Rem:DifferentH-twisted}. To avoid confusion, we give a more explicit proof of the first part of the previous proposition using the convention for the $H$-twisted Dorfman bracket from Definition \ref{Def:GenTanBund} which is used throughout this thesis. We find that
\begin{align*}
    &[e^B(X+\xi),e^B(Y+\eta)]_H-e^B([X+\xi,Y+\eta]_H)\\
    &=[\iota_XB,Y]+[X,\iota_YB]-\iota_{[X,Y]}B\\
    &=-\iota_Y\d\iota_XB+\LieD_X\iota_YB-\iota_{[X,Y]}B\\
    &=(\LieD_X\iota_Y-\iota_Y\LieD_X)B+\iota_Y\iota_X\d B-\iota_{[X,Y]}B\\
    &=\iota_Y\iota_X\d B=-\iota_X\iota_Y\d B
\end{align*}
where we have used the standard formula $\LieD_X\iota_Y-\iota_Y\LieD_X=\iota_{[X,Y]}$.
So we indeed see that
\begin{align*}
    &[e^B(X+\xi),e^B(Y+\eta)]_{H+\d B}-e^B([X+\xi,Y+\eta]_H)\\
    &=[e^B(X+\xi),e^B(Y+\eta)]_H+\iota_X\iota_Y\d B-e^B([X+\xi,Y+\eta]_H)=0.
\end{align*}
So, $e^B$ is indeed an isomorphism from $(\mathbb TM,H)$ to $(\mathbb TM,H+\d B)$. Hitchin used a similar proof in \cite[Proposition 1]{hitchin2010lecturesgeneralizedgeometry}.
\end{remark}

\begin{proposition}[{\cite[p.~5]{severa2001poissongeometry3formbackground}, \cite[p.~39]{gualtieri}}]
\label{Prop:SeveraClass}
Any exact Courant algebroid
$(E\to M,[-,-],\scal{-}{-},\pi)$ is isomorphic as a Courant algebroid to $(\mathbb TM, H)$ for some $H\in \ClForms{3}$. 
Let $H,H'\in \ClForms{3}$ then $(\mathbb TM,H)\cong (\mathbb TM,H')$ as Courant algebroids if and only if $[H]=[H']\in H_{\mathrm{dR}}^3(M)$ (cf.~Proposition \ref{B-fieldTrafoProp}), where $H_{\mathrm{dR}}^k(M)$ denotes the $k$-th de Rham cohomology of $M$ for $k\in \N$. 
Hence, to any exact Courant algebroid $(E\to M,[-,-],\scal{-}{-},\pi)$ we can assign a unique class $[H]\in H_\mathrm{dR}^3(M)$, the so called \textit{\Severa\ class}, such that $(E\to M,[-,-],\scal{-}{-},\pi)\cong (\mathbb TM, H)$ as Courant algebroids. This induces a bijection between the set of isomorphism classes of exact Courant algebroids and $H_\mathrm{dR}^3(M)$.
\end{proposition}

\begin{theorem}[{\cite[Proposition 2.2]{gualtieriAnn}}]
\label{AuteqCharacterizationTheorem}
Let $M$ be a smooth manifold and $H\in \ClForms{3}$. Then, the autoequivalences of the exact Courant algebroid $(\mathbb{T}M, H)$ are of the form $f_* e^B:\mathbb{T}M\to \mathbb{T}M$ with $f\in \Diff$ and $B\in \Forms{2}$, such that the equation
\begin{align*}
    f^*H-H=\d B
\end{align*}
is fulfilled. 
Here $f_*$ is the usual map
\begin{alignat*}{3}
    f_*(X+\xi)=\d f_p(X)+\xi\circ (\d f_p)^{-1}\quad && \text{and} &&\quad e^B(X+\xi)=X+\xi+\iota_XB
\end{alignat*}
with $X\in T_pM$ and $\xi\in T_p^*M$, $p\in M$. 
\end{theorem}

\section{The Nash-Moser Category}
\label{Section:Nash-Moser Category}

It is our goal to describe the group of autoequivalences of an exact Courant algebroid over a compact base manifold as a tame Fréchet Lie group. The model spaces of such Lie groups are tame Fréchet spaces and the chart transition maps as well as the group multiplication and inversion map are smooth tame maps. In this section, we review some of the theory regarding tame Fréchet spaces and smooth tame maps which together form the Nash-Moser Category. This makes these types of maps particularly useful and allows for instance the formulation of an implicit function theorem, see \cite{hamilton}.
We first review tame Fréchet spaces, and more generally, some basic results regarding locally convex topological vector spaces. Then, we introduce smooth tame maps. As the prime example of a tame Fréchet space, we discuss the space of smooth sections of a vector bundle over a compact base manifold. These types of spaces form the model space for all tame Fréchet manifolds we shall consider in this thesis. Moreover, we discuss the most important classes of smooth tame maps between such spaces of smooth sections, namely nonlinear vector bundle operators and more generally nonlinear partial differential operators. These are fundamental ingredients of most of the smooth tame maps we will consider in this thesis. Finally, we discuss how the theory of Sobolev spaces of smooth sections of vector bundles over compact base manifolds translates to the setting of tame Fréchet spaces putting a particular focus on Hodge theory. We will show that the space of closed differential forms on a compact manifold is a tame Fréchet space and that certain operators from Hodge theory are linear (smooth) tame maps. This will allow us to find a useful trivialization of the underlying manifold of the group of autoequivalences of exact Courant algebroids over a compact base manifold in Section \ref{Section:TheGroupOfautoequivalencesOfAnExCAAsATameLieGroup}.

\subsection{Tame Fréchet spaces}

Certain tame Fréchet spaces will be the model spaces of the group of autoequivalences of an exact Courant algebroid over a compact base manifold understood as a tame Fréchet Lie group. So we need to first discuss the relevant theory of tame Fréchet spaces. In this section, we introduce the notion of locally convex spaces and Fréchet spaces. We discuss some topological results for these types of spaces, and present the Riemann integral for locally convex spaces. Then we introduce the notion of graded and, finally, tame Fréchet spaces.

\begin{definition}
\label{Def:LocConvSpace}
Let $\K=\R,\mathbb{C}$.
\begin{enumerate}
\item We say that a vector space $F$ over $\K$ with a topology is a topological vector space if addition and scalar multiplication are continuous.

\item Let $F$ be a vector space over $\K$ and let $(\nnorm{-}{\alpha}:F\to \R_{\geq 0})_{\alpha\in A}$ be a family of seminorms where $A$ is some index set.
For $\epsilon>0$ we denote by 
$\epsilon B_\alpha \coloneqq \epsilon B_{\nnorm{-}{\alpha}} \coloneqq \left\{x\in F|~\nnorm{x}{\alpha}<\epsilon \right\}$ the ball of radius $\epsilon$ with respect to the seminorm $\nnorm{-}{\alpha}$.
The topology $\tau$ induced by the family $(\nnorm{-}{\alpha})_{\alpha\in A}$ is defined such that the sets $x+\epsilon B_\alpha$, with $\alpha \in A$, $\epsilon \in (0,\infty)$ form a neighbourhood subbase for every $x\in F$.
\end{enumerate}
\end{definition}

\begin{remark}
It is easy to see from the triangle inequality that $\tau$ from the previous definition is indeed a well defined topology and that it turns $F$ into a topological vector space. If additionally $\tau$ is Hausdorff, then we call $(F,\tau)$ a \textit{locally convex} (topological vector) space.
The topology being Hausdorff is equivalent to the condition that $\nnorm{x}{\alpha}=0$ for all $\alpha\in A$ implies that $x=0$ (cf.~\cite[Proposition 4.3.3]{TVSc}). We call such a family of seminorms \textit{separating}.
\end{remark}

\begin{definition}
\label{Def:Completeness}
\begin{enumerate}
\item We say that a filter $\mathcal F$ on a topological space $X$ converges to a point $x\in X$ if it is finer than the neighbourhood filter of $x$, that is, if every neighbourhood $U$ of $x$ is an element of $\mathcal F$.

\item We call a filter $\mathcal F$ on a topological vector space $F$ a Cauchy filter if for every neighbourhood $U$ of $0\in F$ there is some $V\in \mathcal F$ such that $V-V\coloneqq \{x-y|~x,y\in V\}$ is a subset of $U$.

\item A locally convex topological vector space is called \textit{complete} if every Cauchy filter converges. 

\item A Cauchy sequence on a topological vector space $F$ is a sequence $(x_i)_{i\in \N}$ on $F$ such that for every neighbourhood $U\subseteq F$ of $0\in F$ there exists an $N\in \N$ such that $x_i-x_j\in U$ for every $i,j>N$.

\item A topological vector space is called sequentially complete if every Cauchy sequence converges.
\end{enumerate}
\end{definition}

\begin{definition}
A locally convex topological vector space is called \textit{metrizable}, if the family of seminorms $(\nnorm{-}{\alpha})_{\alpha\in A}$, for some index set $A$, defining its topology is countable. We may then simply assume that $A=\N$.
\end{definition}

\begin{remark}
It is easy to see that if $F$ is a vector space, then the topology induced by some family of seminorms $(\nnorm{-}{\alpha}:F\to \R_{\geq 0})_{\alpha\in A}$, where $A$ is some index set, is the initial topology with respect to the maps $F\ni x\mapsto \nnorm{v-x}{\alpha}\in [0,\infty)$, $\alpha \in A$, $v\in F$ or equivalently the coarsest topology on $F$ such that addition $+:F\times F\to F$ is continuous and each $\nnorm{-}{\alpha}$, $\alpha\in A$, is continuous.
\end{remark}

\begin{remark}
If a vector space $F$ equipped with some family of seminorms $(\nnorm{-}{k})_{k\in \N}$ is a metrizable locally convex space then we can define the translation invariant metric 
\begin{align*}
    d(x,y)\coloneqq \sum_{k=0}^\infty \frac{1}{2^k}\frac{\nnorm{x-y}{k}}{1+\nnorm{x-y}{k}}
\end{align*}
for $x,y\in F$. It is easy to see that the topology induced by $d$ and the topology induced by $(\nnorm{-}{k})_{k\in \N}$ is the same. The idea is that while $d$ does depend on all seminorms $\nnorm{-}{k}$ we can find for every $\epsilon>0$ an $N\in \N$ such that $\sum_{k=N}^\infty \frac{1}{2^k}\frac{\nnorm{x-y}{k}}{1+\nnorm{x-y}{k}}<\epsilon$ for all $x,y\in F$.

As the topologies induced by $(\nnorm{-}{k})_{k\in \N}$ and $d$ agree, the notion of a Cauchy sequence presented in Definition \ref{Def:Completeness} is equivalent to the notion of a Cauchy sequence with respect to the metric $d$. 

Another equivalent way to define a Cauchy sequence on $F$ is as a sequence $(x_i\in F)_{i\in \N}$ such that for each $k\in \N$ and each $\epsilon>0$ there exists an $N\in \N$ such that $\nnorm{x_i-x_j}{k}<\epsilon$ for all $i,j>N$. The equivalence of these notions follows directly from the way we defined the topology on a vector space from a family of seminorms.
\end{remark}

\begin{remark}
\label{Rem:CompImpliesSeqComp}
We note that to any sequence $(x_n)_{n\in \N}$ on a topological space $X$ we can associate the filter $\mathcal F_S\coloneqq \{A\subseteq X|~\abs{S\setminus A}<\infty \}$, where $S=\{x_0,x_1,...\}$. It is easy to see that $(x_n)_{n\in \N}$ converges to some $x\in X$ precisely if $\mathcal F_S$ converges to $x$. 
If, moreover, $X$ is a topological vector space and $(x_n)_{n\in \N}$ a Cauchy sequence, then it is easy to see that $\mathcal F_S$ is a Cauchy filter. So we see that any complete topological vector space is, in particular, sequentially complete. In the following lemma we show the standard result that the inverse is also true if $X$ is a topological vector space whose topology can be induced by a translation invariant metric.
\end{remark}

\begin{lemma}
\label{Lemma:MetrizableSeqCompImpliesComp}
Let $F$ be a topological vector space and $d:F\times F\to \R$ a translation-invariant metric such that the topology of $F$ is induced by $d$. Then, if $F$ is sequentially complete it is also complete.
\end{lemma}
\begin{proof}
Although this is a standard result we include a short proof for the reader's convenience.
The neighbourhood filter of $0\in F$ has a countable basis which we will denote by $(B_n)_{n\in \N}$ such that $B_n=\{x\in F|~d(0,x)<1/n\}$ for $n\in \N_{>0}$ and $B_0=F$. Let $\mathcal F$ be some Cauchy filter on $F$ then by Definition \ref{Def:Completeness} for any $n\in \N$ we can find some $\Tilde U_n\in \mathcal F$ such that $\Tilde{U}_n-\Tilde{U}_n=\{x-y|~x,y\in \Tilde{U}_n\}\subseteq B_n$ we can now iteratively define sets $U_n\coloneqq \cap_{k\leq n} \Tilde{U}_k\in \mathcal F$. These satisfy again $U_n-U_n\subseteq B_n$ but also $U_{n+1}\subseteq U_n$ for any $n\in \N$. We now choose some sequence $(x_n)_{n\in \N}$ such that $x_n\in U_n$. Then by definition we find that $x_i-x_j\in (U_i-U_j)\subseteq (U_n-U_n)\subseteq  B_n$ for $i,j\geq n$. So $(x_n)_{n\in \N}$ is a Cauchy sequence and converges to some $x\in F$ by sequential completeness. Moreover, by the triangle inequality and convergence of $(x_n)_{n\in \N}$ we find that for any $n<i$, with $i,n\in \N$, that $x_i-x\in B_n$. Now, let $(y_n\in U_n)_{n\in \N}$ be another such choice of sequence. Then $(y_n)_{n\in \N}$ also converges to some $y\in F$ and we claim that $y=x$. This is because $x_i-y_i\in U_i-U_i\subseteq B_i$ for every $i\in \N$. So, in particular, $y_i-x=y_i-x_i+x_i-x$ converges to $0\in F$. Since the choice of $x_n\in U_n$ was arbitrary we have $x_n\in x+B_n$ for every choice of point $x_n\in U_n$ and so $U_n\subseteq x+B_n$ which implies that $\mathcal F$ is finer than the neighborhood filter of $x\in F$, so $\mathcal F$ converges to $x$.    
\end{proof}

\begin{definition}
We call a sequentially complete metrizable locally convex topological vector space a \textit{Fréchet space}.

We may equivalently define a Fréchet space as a complete metrizable locally convex topological vector space (cf.~Remark \ref{Rem:CompImpliesSeqComp} and Lemma \ref{Lemma:MetrizableSeqCompImpliesComp}).

We will frequently denote a Fréchet space $F$ with the topology induced by a family of seminorms $(\nnorm{-}{k})_{k\in \N}$ by $(F,(\nnorm{-}{k})_{k\in \N})$.
\end{definition}

\begin{example}
Clearly, Banach spaces are Fréchet spaces.
\end{example}

Of course there are Fréchet spaces that are not Banach spaces.
\begin{example}
\label{Ex:SmoothSectionsFrechetSpace}
Let $M$ be a compact manifold and $V\to M$ be a vector bundle. Then by choosing a connection on $T^*M$ and a connection on $V$ we get an induced connection on $(T^*M)^{\tens k}\tens V$ for any $k\in \N$. This allows us to make sense of the $k$-th covariant derivative $\nabla^k\sigma\in \Gamma((T^*M)^{\tens k}\tens V)$ for any smooth section $\sigma\in \Gamma(V)$, and any $k\in \N$. Similarly, by choosing a positive definite bundle metric on $T^*M$ and $V$, respectively, we obtain a positive definite bundle metric on $(T^*M)^{\tens k}\tens V$ for any $k\in \N$. For any $k\in \N$, we denote the induced fiberwise norm by $|-|$.
Then, the vector space of smooth sections $\Gamma(V)$ equipped with the topology induced by the family of norms
\begin{align*}
    \left( \nnorm{\sigma}{n}\coloneqq \sum_{k=0}^n \sup_{x\in M}|\nabla^k\sigma|\right)_{n\in \N}
\end{align*}
is a Fréchet space, see \cite[Part I, Example 1.1.5]{hamilton}. We will in fact later see that this is an example of a tame Fréchet space.

In particular, the smooth functions $C^\infty(M,\R)$ form a Fréchet space by identifying them with the smooth sections of the trivial vector bundle $M\times \R\to M$.
\end{example}

\begin{remark}
\label{Rem:ClosedSubspAndQuotientFrechet}
Let $(F,(\nnorm{-}{n})_{n\in \N})$ be a Fréchet space. Since clearly any linear subspace of a metrizable locally convex topological vector space is again a metrizable locally convex topological vector space, and since closed subsets of complete metric spaces are complete, we find that any closed linear subspace of a Fréchet space with the subspace topology is a Fréchet space. 
Also, every quotient of $F$ by a closed subspace $G\subseteq F$ is a Fréchet space, see \cite[Chapter I, 2.3 and Chapter II, §4, p.~49]{SchaeferTVS}.
\end{remark}

\begin{theorem}[Open mapping theorem, {\cite[Corollaries 2.12]{rudin1991functional}}]
\label{Thm:OpenMapping}
Let $F$ and $G$ be Fréchet spaces and $L:F\to G$ a continuous linear surjection. Then $L:F\to G$ is an open map.
\end{theorem}

\begin{remark}
\label{Rem:FreFinerTop}
If a vector space $F$ can be equipped with two topologies $\tau_1$ and $\tau_2$ such that $(F,\tau_1)$ and $(F,\tau_2)$ are Fréchet spaces and $\tau_1$ is finer than $\tau_2$ then $(F,\tau_1)\xrightarrow{\mathrm{Id_F}}(F,\tau_2)$ is a continuous map and a vector space isomorphism and hence by the open mapping theorem $\tau_1=\tau_2$.

One might be tempted to define a common locally convex Hausdorff and metrizable refinement of the topologies $\tau_1$ and $\tau_2$. However, that refinement might not be complete so we cannot infer that all topologies on $F$ that make it into a Fréchet space are equivalent.
\end{remark}

\begin{theorem}[Hahn-Banach Theorem, {\cite[Theorem 5.1.1]{TVSc}}]
\label{Thm:GenHahnBanach}
Let $F$ be a vector space over $\K=\R,\mathbb C$, and let $\nnorm{-}{}:F\to \R_{\geq 0}$ be a seminorm. Let $G\subseteq F$ be a linear subspace and $l:G \to \K$ be a linear functional such that 
$|l(x)|\leq \nnorm{x}{}$
holds for all $x\in G$. 
Then, we can extend $l$ to a linear functional $\Tilde{l}:F\to \K$ such that
$|\Tilde l(x)|\leq \nnorm{x}{}$
holds for all $x\in F$.
\end{theorem}

\begin{corollary}
\label{Cor:HahnBanachForLocConv}
Let $A$ be some index set and $(F,(\nnorm{-}{\alpha})_{\alpha\in A})$ be a locally convex space over the field $\K=\R,\mathbb C$. Let $G\subseteq F$ be a closed linear subspace. Then, any continuous linear functional $l:G\to \K$ can be extended to a continuous linear functional $\Tilde{l}:F\to \K$.
\end{corollary}
\begin{proof}
Let $l:G\to \K$ be continuous. Then it is clear by the linearity and continuity of $l$, and by $G$ being equipped with the subspace topology, that there is some continuous seminorm $\nnorm{-}{}$ on $F$ which may be written as a sum of finitely many seminorms of the form $\nnorm{-}{\alpha}$, $\alpha\in A$, such that $\abs{l(x)}\leq \nnorm{x}{}$ for all $x\in G$. This holds in much greater generality, see \cite[Proposition 4.6.1]{TVSc}.
So we may extend $l$ to a linear functional $\Tilde{l}:F\to \K$ such that $|\Tilde{l}(x)|\leq \nnorm{x}{}$ for all $x\in F$. Hence, $\Tilde{l}:F\to \K$ is continuous.
\end{proof} 

\begin{lemma}[{\cite[Theorem 3.1.1]{TVSc}}]
\label{Lemma:FinDimHausdTVS}
Let $F$ be a Hausdorff topological vector space of dimension $n<\infty$ over $\K=\R,\mathbb C$. Then $F$ and $\K^n$ are isomorphic as topological vector spaces. $\K^n$ is of course equipped with the usual Euclidean topology.
\end{lemma}
\begin{corollary}
\label{Cor:l(v)=l(w)Thenv=w}
Let $F$ be a locally convex space over the field $\K=\R,\mathbb C$, and let $v,w\in F$. If for every continuous linear functional $l:F\to \K$ we have $l(v)=l(w)$, then $v=w$.
\end{corollary}
\begin{proof}
Suppose $v\neq w\in F$ and $l(v-w)=0$ for every continuous linear functional $l:F\to \K$ then $\mathrm{span}_\K(v-w)\subseteq F$ is isomorphic to $\K$ (cf.~Lemma \ref{Lemma:FinDimHausdTVS}). Thus, the linear functional $l_{v-w}:\mathrm{span}_\K(v-w)\to \K$ mapping $v-w\mapsto 1$ is continuous and can be extended to a continuous linear functional $\Tilde{l}_{v-w}:F\to \K$ (cf.~Corollary \ref{Cor:HahnBanachForLocConv}). By construction we have $\Tilde{l}_{v-w}(v-w)=1$ contradicting the assumption.
\end{proof}

\begin{lemma}[{\cite[\S4, p.~209]{TVS1Book}} and {\cite[Theorem 2.5.14]{TVSc}}]
\label{Lemma:ExtOfMapsIntoCompHausdTVS}
Let $F$ and $G$ be Hausdorff topological vector spaces over the field $\K=\R,\mathbb C$, $A\subseteq F$ be a dense linear subspace, and let $G$ be complete. Then
\begin{enumerate}
    \item any continuous seminorm $\nnorm{-}{}:A \to \R_{\geq 0}$ can be extended uniquely to a continuous seminorm $\nnorm{-}{}:F\to \R_{\geq 0}$, and
    \item any continuous linear map $L:A\to G$ can be extended uniquely to a continuous linear map $L:F\to G$. 
\end{enumerate}
\end{lemma}

\begin{lemma}[{\cite[Theorem 2.5.15]{TVSc}} and {\cite [\S4, p.~208f]{TVS1Book}}]
\label{Lemma:CompletionAndExtension}
Let $F$ be a Hausdorff topological vector space. Then there exists a complete Hausdorff topological vector space $\hat{F}$ and a continuous linear embedding $i:F\to \hat{F}$ such that $i(F)$ is a dense subset of $\hat{F}$. The space $\hat{F}$ is unique up to isomorphism of topological vector spaces.
In particular, if $F$ is complete then $\hat{F}\cong F$.

If $F$ is locally convex with a topology induced by seminorms $(\nnorm{-}{\alpha})_{\alpha\in A}$, for some index set $A$, then $\hat{F}$ is also a locally convex topological vector space and its topology is induced by the unique continuous seminorms obtained as extensions of the seminorms $(\nnorm{-}{\alpha})_{\alpha\in A}$ (cf.~Lemma \ref{Lemma:ExtOfMapsIntoCompHausdTVS}). 
\end{lemma}

\begin{definition}
Let $F$ be a Hausdorff topological vector space. Then we call the unique (up to isomorphism of topological vector spaces) complete Hausdorff topological vector space  $\hat{F}$ as in the previous lemma the \textit{completion of $F$}.
\end{definition}

\begin{remark}
\label{Rem:TopOnContPaths}
Let $F$ be a locally convex space, and let the topology on $F$ be induced by a family of seminorms $(\nnorm{-}{\alpha})_{\alpha\in A}$ for some index set $A$.
We denote by $C([a,b],F)$ the continuous paths in $F$ over the interval $[a,b]$. It immediately follows from the fact that $(\nnorm{-}{\alpha})_{\alpha\in A}$ is a separating family of seminorms, that if we equip $C([a,b],F)$ with the topology induced by the family of seminorms 
\begin{align}
\label{Eq:SupNormOnC([a,b],F)}
    \left(\nnorm{f}{\alpha,\mathrm{sup}}\coloneqq \sup_{t\in [a,b]}\nnorm{f(t)}{\alpha}\right)_{\alpha\in A}
\end{align}
it is Hausdorff and so $C([a,b],F)$ has the structure of a locally convex space.
\end{remark}

On the level of Fréchet spaces the following proposition is well-known (cf.~\cite[Part I, Theorem 2.1.1]{hamilton}). In the general form as we formulate it, it has already been used by Neeb in \cite{InfDimLieneeb} and by Glöckner in \cite{GlockInfDimLieGr}. However, we could not find a proof in this generality, and so found it worthwhile to do some simple generalizations of the proof for Fréchet spaces provided by Hamilton in \cite[Part I, Theorem 2.1.1]{hamilton}.

\begin{proposition}
\label{Prop:RiemInt}
Let $F$ be a locally convex topological vector space over the field $\K=\R,\mathbb C$, let $\hat{F}$ be its completion, and $f:[a,b]\to F$, $a,b\in \R$ and $a<b$, be a continuous path. Then there is an element $\int_a^bf(t)dt\in \hat{F}$ such that for any continuous linear functional $l:\hat F\to \K$ the equation $$l\left(\int_a^bf(t)dt\right)=\int_a^b(l\circ f)(t)dt$$ holds. It is unique by Corollary \ref{Cor:l(v)=l(w)Thenv=w} and has the following additional properties:
\begin{enumerate}
    \item \label{Item:IntSeminorm} For every continuous seminorm $\nnorm{-}{}:\hat F\to \R_{\geq0}$ the inequality
    \begin{align*}
        \nnorm{\int_a^bf(t)dt}{}\leq \int_a^b\nnorm{f(t)}{}dt
    \end{align*}
    holds.
    \item \label{Item:intac+intcb=intab} The equation
    \begin{align*}
        \int_a^cf(t)dt+\int_c^bf(t)dt=\int_a^bf(t)dt
    \end{align*}
    holds for $c\in [a,b]$.
    \item \label{Item:IntContLin} The induced map $\int_a^b:C([a,b],F)\to \hat F$ is continuous and linear with respect to the natural topology on $C([a,b],F)$ discussed in Remark \ref{Rem:TopOnContPaths}. 
\end{enumerate}
\end{proposition}
\begin{proof}
Throughout this proof, we view $F$ as a subspace of its completion $\hat F$.

We consider piecewise linear continuous paths in $F$ over the interval $[a,b]$, that is, paths $f\in C([a,b],F)$ such that there is a partition $a=t_0\leq t_1\leq ... \leq t_n=b$, $n\in \N$ of the interval $[a,b]$ such that $\left.f\right|_{[t_{i},t_{i+1}]}(t)=f_i+tf'_i$ with $f_i,f_i'\in F$, $i=0,\ldots,n-1$, and $t\in [t_i,t_{i+1}]$. We denote the linear subspace of piecewise linear paths in $F$ over the interval $[a,b]$ by $\mathrm{PL}([a,b],F)\subseteq C([a,b],F)$. As a subspace of a locally convex space, it is also a locally convex space, and it is easy to see, e.g.~by a Lebesgue number argument, that $\mathrm{PL}([a,b],F)$ is a dense subspace of $C([a,b],F)$. 

We define the Riemann integral on $\mathrm{PL}([a,b],F)$ as the linear map 
$\int_a^b: \mathrm{PL}([a,b],F)\to F$ that maps a path that is piecewise linear with respect to a partition $a=t_0\leq t_1\leq \ldots \leq t_n=b$, $n\in \N$ as above to 
\begin{align}
\label{Eq:IntOnPL}
    \int_a^bf(t)dt\coloneqq \sum_{i=0}^{n-1}\frac{1}{2}(f(t_{i+1})+f(t_i))(t_{i+1}-t_i).
\end{align}
It is easily observed that it is well defined, and continuous.
We note that we can, of course, view this map as a map with codomain $\hat F$. So, by Lemma \ref{Lemma:ExtOfMapsIntoCompHausdTVS} we may extend $\int_a^b:\mathrm{PL}([a,b],F)\to \hat F$ uniquely to a continuous linear map, which we will denote again by $\int_a^b:C([a,b],F)\to \hat F$. Hence, in particular, property \ref{Item:IntContLin} is satisfied.

We denote by $\int_{\K,a}^b:C([a,b],\K)\to \K$ the continuous linear map with $\int_{\K,a}^bf=\int_a^bf(t)dt$ for $f\in C([a,b],\K)$. It is clear that if $f\in \mathrm{PL}([a,b],F)$ and $l:\hat F\to \K$ is a continuous linear functional then, by applying $l$ on both sides to equation \eqref{Eq:IntOnPL} we obtain
$l\left(\int_a^bf(t)dt\right)=\int_a^bl(f(t))dt$.
We have already seen that $\int_a^b$ can be uniquely continuously extended to paths in $C([a,b],F)$ so $l\circ \int_a^b:C([a,b],F)\to \K$ is the unique continuous linear extension of $l\circ\int_a^b=\int_{\K,a}^b\circ l:\mathrm{PL}([a,b],F)\to \K$. On the other hand
$\int_{\K,a}^b\circ l:C([a,b],F)\to \K$ is a well defined continuous linear map. 
So we find, that
$l\circ \int_a^b=\int_a^b\circ l:C([a,b],F)\to \K$. Moreover, $\int_a^bg(t)dt$ is unique for every $g\in C([a,b],F)$ by Corollary \ref{Cor:l(v)=l(w)Thenv=w}.

Since the completion of a locally convex space is again locally convex (cf.~Lemma \ref{Lemma:CompletionAndExtension}) we get by Corollary \ref{Cor:l(v)=l(w)Thenv=w} and the previous property that the properties of the usual Riemann integral with $\K$-valued paths translate to the Riemann integral on locally convex spaces. In particular property \ref{Item:intac+intcb=intab} is satisfied.

Lastly, we need to show that property \ref{Item:IntSeminorm} is satisfied. Again, it is clear that if $f\in \mathrm{PL}([a,b],F)$ the property is satisfied by applying a continuous seminorm $\nnorm{-}{}:\hat F\to \R$ to equation \eqref{Eq:IntOnPL}. We denote the supremum norm on $C([a,b],F)$ induced by the restriction of $\nnorm{-}{}$ to $F$ by $\nnorm{-}{\mathrm{sup}}$ (cf.~equation \eqref{Eq:SupNormOnC([a,b],F)}). Now, by using the linearity of the integral, the triangle inequality, and that property \ref{Item:IntSeminorm} is satisfied on $\mathrm{PL}([a,b],F)$ we see that for $f\in C([a,b],F)$ and $g\in \mathrm{PL}([a,b],F)$ the following inequalities hold: 
\begin{align*}
    \nnorm{\int_a^bf(t)dt}{}&\leq \int_a^b\nnorm{g(t)}{}dt+\nnorm{\int_a^b(f(t)-g(t))dt}{} \\ 
    &\leq \int_a^b\nnorm{f(t)}{}dt+\int_a^b\nnorm{g(t)-f(t)}{}dt+\nnorm{\int_a^b(f(t)-g(t))dt}{} \\
    &\leq \int_a^b\nnorm{f(t)}{}dt+(b-a)\nnorm{g-f}{\mathrm{sup}}+\nnorm{\int_a^b(f(t)-g(t))dt}{}.
\end{align*}
Now, by the density of $\mathrm{PL}([a,b],F)\subset C([a,b],F)$, and the linearity and continuity of the integral $\int_a^b:C([a,b],F)\to \hat F$ we can find a $g_\epsilon \in \mathrm{PL}([a,b],F)$ for any $\epsilon \in \R_{>0}$ such that $(b-a)\nnorm{g_\epsilon-f}{\mathrm{sup}}\leq \epsilon/2$ and $\nnorm{\int_a^b(f(t)-g_\epsilon(t))dt}{}\leq \epsilon /2$. Hence, $\nnorm{\int_a^bf(t)dt}{}\leq \int_a^b\nnorm{f(t)}{}dt+\epsilon$ for every $\epsilon\in \R_{>0}$. It follows that property \ref{Item:IntSeminorm} is satisfied for all $f\in C([a,b],F)$.
\end{proof}

\begin{definition}
For $f\in C([a,b],F)$ we call $\int_a^bf(t)dt\in \hat F$ as defined in the previous proposition the \textit{Riemann integral of $f$}.
\end{definition}

\begin{remark}
If in the previous proposition $F$ is assumed to be a complete locally convex vector space, e.g.~a Fréchet space, then we can identify $F$ with $\hat F$ and hence do not need to consider completions for the Riemann integral.
\end{remark}

\begin{remark}
\label{Rem:IntegralCommWithLinMap}
Let $F$ and $G$ be locally convex vector spaces over $\K=\R,\mathbb C$, $\hat F$ and $\hat G$ be their respective completions, $f\in C([a,b],F)$, $a\leq b\in \R$, and let $L: F\to G$ be a continuous linear map. Then by Lemma \ref{Lemma:ExtOfMapsIntoCompHausdTVS} the map $F\xrightarrow{L} G\hookrightarrow \hat G$ has a unique extension $\hat L:\hat F\to \hat G$. Viewing $F$ and $G$ again as subspaces of $\hat F$ and $\hat G$ respectively, we obtain the following familiar equation,
\begin{align*}
    \int_a^b (L\circ f)(t)dt=\int_a^b (\hat L\circ f)(t)dt=\hat L\left(\int_a^bf(t)dt\right).
\end{align*}
This follows immediately from the previous proposition since if $l:\hat G\to \K$ is a linear continuous functional, then so is $l\circ \hat L:\hat F\to \K$. Hence, we find that
\begin{align*}
    l\left(\int_a^b(L\circ f)(t)dt\right)
    =\int_a^b (l\circ L)(f(t))dt=\int_a^b (l\circ \hat L)(f(t))dt
    =l\left(\hat L\left(\int_a^bf(t)dt\right)\right)
\end{align*}
which uniquely determines the Integral by the Hahn-Banach theorem (cf.~Corollary \ref{Cor:l(v)=l(w)Thenv=w}).

In particular, if $F$ and $G$ are complete and $F\subseteq G$ is a closed linear subspace then the Riemann integral on $F$ and the Riemann integral on $G$ agree on continuous paths with values in $F$.
\end{remark}

\begin{lemma}
\label{Lemma:ProductOfLCAndFreSpaces}
Let $I$ be some index set and let $(F_i,(\nnorm{-}{i,\alpha})_{\alpha \in A_i})$ for all $i\in I$ be locally convex spaces, the $A_i$ again being index sets. Then the product $\Pi_{i\in I}F_i$ with the product topology and the natural vector space structure forms a locally convex vector space. Moreover, if $I$ is countable and $F_i$ is a Fréchet space for every $i\in I$,  then $\Pi_{i\in I}F_i$ is a Fréchet space.
\end{lemma}
\begin{proof}
Since the proof of this standard result is quite short we present it here.

It is easy to see that the product topology on $\Pi_{i\in I}F_i$ is precisely the topology induced by seminorms of the form $\nnorm{(x_i)_{i\in I}}{\alpha_{k}\in A_{k}}=\nnorm{x_{k}}{k,\alpha_{k}}$, with $k\in I$, and $(x_i)_{i\in I}\in \Pi_{i\in I}F_i$. This is a separating family of seminorms and hence proves the first part of the lemma.

From this we deduce that if $I$ is countable and the vector spaces $F_i$, $i\in I$, are Fréchet spaces, then the topology on $\Pi_{i\in I}F_i$ is induced by a countable separating family of seminorms. It is clear that if $(x_{i,j})_{i\in I, j\in \N}$ is a Cauchy sequence in $\Pi_{i\in I}F_i$ then $(x_{i,j})_{j\in \N}$ is a Cauchy sequence in $F_i$ for every $i\in I$. Since it converges with respect to all seminorms $\nnorm{-}{i,k}$ with $i\in I$, and $k\in A_i=\N$ it converges in $\Pi_{i\in I}F_i$.
\end{proof}

\begin{definition}[{\cite[Part II, Definition 1.1.1]{hamilton}}]
\label{Def:GradFreSpac}
Let $(F,(\nnorm{-}{n})_{n\in \N})$ be a Fréchet space. Then we call it \textit{graded} if the seminorms $(\nnorm{-}{n})_{n\in \N}$ are increasing in strength, i.e.
\begin{align*}
    \nnorm{-}{0}\leq \nnorm{-}{1}\leq \nnorm{-}{2}\leq \ldots
\end{align*}
holds, and we call $(\nnorm{-}{n})_{n\in \N}$ a \textit{grading} on $F$.
\end{definition}

The following convention for the grading on a product of graded Fréchet spaces was introduced in \cite[p.~134]{hamilton}
\begin{remark}
\label{Rem:ProdOfFreSpa}
We have seen in the previous lemma that the countable product of Fréchet spaces is again a Fréchet space. Let now $(F,(\nnorm{-}{F,i})_{i\in \N})$ and $(G,(\nnorm{-}{G,j})_{j\in \N})$ be graded Fréchet spaces. Then $(F\oplus G, (\nnorm{-}{n,n}^\oplus)_{n\in \N})$, with $\nnorm{(f,g)}{n,n}^\oplus=\nnorm{f}{F,n}+\nnorm{g}{G,n}$ for every $n\in \N$, $f\in F$, and $g\in G$, is linearly isomorphic to $F\times G$ and the topologies agree. Hence $(\nnorm{-}{n,n}^\oplus)_{n\in \N}$ is a grading on $F\times G$.
One could also easily show that $(F\oplus G, (\nnorm{-}{n,n}^\oplus)_{n\in \N})$ is a Fréchet space and employ Remark \ref{Rem:FreFinerTop}.
\end{remark}

\begin{convention}
\label{Conv:ProdOfFreSpa}
Throughout this thesis when talking about products of finitely many graded Fréchet spaces we mean the product equipped with the natural grading introduced in Remark \ref{Rem:ProdOfFreSpa}.
\end{convention}

\begin{example}
\label{Example:Sigma(B)}
Let $(B,\nnorm{-}{B})$ be a Banach space. By $\Sigma(B)$ we denote the space of sequences $x=(x_i\in B)_{i\in \N}$ such that for every $n\in \N$ 
\begin{align*}
    \nnorm{x}{n}\coloneqq \sum_{k=0}^\infty e^{kn}\nnorm{x_k}{B}<\infty.
\end{align*}
Then, the space $\Sigma(B)$ with the topology induced by $(\nnorm{-}{n})_{n\in \N}$ is a graded Fréchet space, see \cite[Part II, Examples 1.1.2 (2)]{hamilton}.
\end{example}

\begin{definition}
\label{Def:TameLinMap}
Let $(F,(\nnorm{-}{k}^F)_{k\in\N})$ and $(G,(\nnorm{-}{k}^G)_{k\in\N})$ be graded Fréchet spaces. A linear map $L:F\to G$ is called tame linear of degree $r$ and base $b$ if there are constants $C_n>0$ such that 
\begin{align*}
    \nnorm{Lx}{n}^G\leq C_n \nnorm{x}{n+r}^F
\end{align*}
for all $x\in F$ and $n\geq b$. Naturally, every tame linear map is also continuous.

We say that a linear map between graded Fréchet spaces is a tame isomorphism if it is a vector space isomorphism, tame linear, and its inverse is tame linear as well.
\end{definition}

\begin{definition}[{\cite[Part II, Definition 1.1.3]{hamilton}}]
Let $F$ be a Fréchet space, and let $(\nnorm{-}{i})_{i\in \N}$ and $(\nnorm{-}{j}')_{j\in \N}$ be two gradings of $F$. We say that these gradings are \textit{tamely equivalent of degree $r$ and base $b$} if there are positive constants $C_n,C_n'>0$ for every $n\in \N$ such that
\begin{alignat*}{1}
    \nnorm{x}{n}\leq C_n \nnorm{x}{n+r}'\quad \textnormal{and} \quad \nnorm{x}{n}'\leq C_n' \nnorm{x}{n+r}
\end{alignat*}
for each $x\in F$ and $n\geq b$, or equivalently if the map $(F,(\nnorm{-}{i})_{i\in \N})\xrightarrow{\mathrm{Id}} (F,(\nnorm{-}{j}')_{j\in \N})$ and its inverse are tame linear of degree $r$ and base $b$.
\end{definition}

\begin{definition}[{\cite[Part II, Definition 1.3.1]{hamilton}}]
\label{Def:TameDirSummand}
Let $F$ and $G$ be graded Fréchet spaces. We call $F$ a \textit{tame direct summand} of $G$ if there is a graded Fréchet space $H$ and a tame isomorphism $G\cong F\times H$ (cf.~Remark \ref{Rem:ProdOfFreSpa} for the grading on $F\times H$).    
\end{definition}

\begin{proposition}
\label{Prop:TameDirSumBySequence}
Let $F$ and $G$ be graded Fréchet spaces. Then, $F$ is a tame direct summand of $G$ if and only if there are tame linear maps $L:F\to G$ and $M:G\to F$ such that $M\circ L=\mathrm{Id}_F:F\to F$. 
\end{proposition}
\begin{proof}
Let $F$ be a tame direct summand of $G$ then there is a graded Fréchet space $H$ and a tame isomorphism $J:G\xrightarrow{\cong} F\times H$. Let $i:F\to F\times H$ be the inclusion map $i(x)=(x,0)$ for $x\in F$, and $\pi:F\times H\to F$ be the obvious projection. These maps are clearly tame linear. Hence, the maps $L\coloneqq J^{-1}\circ i$ and $M\coloneqq\pi\circ J$ are tame linear and clearly $M\circ L=\mathrm{Id}_F$. 

On the other hand if $L:F\to G$ and $M:G\to F$ are tame linear maps and satisfy $M\circ L=\mathrm{Id}_F$ then $L$ is injective and $M$ is surjective. The map $M$ is continuous, linear and $\{0\}\subseteq F$ is a closed subset since $F$ is Hausdorff. Hence, $\mathrm{ker}(M)$ is a closed subspace of $G$ and thus a graded Fréchet (sub-)space. We can now construct a tame isomorphism $N:F\times \operatorname{ker}(M)\to G$ given by $N(f,k)=Lf+k$ for $f\in F$ and $k\in \mathrm{ker}(M)$. It is tame by the triangle inequality and $L$ being tame. It is also obviously linear and has the inverse $N^{-1}g=(Mg,(\mathrm{Id_G}-LM)g)$, $g\in G$, which is again tame linear by the triangle inequality and $M$ and $L$ being tame. Thus $N$ is a tame linear isomorphism.    
\end{proof}

\begin{definition}
Let $F$ be a graded Fréchet space. We call $F$ a tame Fréchet space or simply \textit{tame space} if it is a tame direct summand of $(\Sigma(B),(\nnorm{-}{n})_{n\in \N})$ for some Banach space $B$ (cf.~Example \ref{Example:Sigma(B)}).
\end{definition}

\begin{example}
\label{Ex:BanachSpacesAreTameFre}
A trivial example of a tame Fréchet space is again a Banach space $B$. It is clear that the map $B\ni x\mapsto (x,0,\ldots)\in \Sigma(B)$ and $\Sigma(B)\ni (x_i)_{i\in \N}\mapsto x_0\in B$ are tame linear maps that compose to $\mathrm{Id}_B$. 
\end{example}

\begin{remark}
\label{Rem:DirSumm&ProdAreTame}
Naturally, any tame direct summand of a tame Fréchet space is again a tame Fréchet space. Note that this is a stronger condition than simply being a closed subspace.

Also clearly the product of two tame Fréchet spaces is again a tame Fréchet space as $\Sigma(B_1\times B_2)\cong \Sigma(B_1)\times \Sigma(B_2)$ for two Banach spaces $B_1$ and $B_2$. In fact we find that 
\begin{align*}
    \nnorm{(x_k,y_k)_{k\in\N}}{\Sigma(B_1\times B_2),n}&=\sum_{k=0}^\infty e^{kn}\nnorm{(x_k,y_k)}{B_1\times B_2}=\sum_{k=0}^\infty e^{kn}(\nnorm{x_k}{B_1}+\nnorm{y_k}{B_2})\\
    &=\nnorm{(x_k)_{k\in \N}}{\Sigma(B_1),n}+\nnorm{(y_k)_{k\in \N}}{\Sigma(B_2),n}
\end{align*}
holds for every $(x_k,y_k)_{k\in \N}\in \Sigma(B_1\times B_2)$.

If there is a graded Fréchet space $F$, a tame Fréchet space $G$ and vector space isomorphism $J:F\to G$ such that $J$ and its inverse are tame linear maps, then $F$ is a tame Fréchet space. Indeed, by the sequence 
\begin{align*}
    F\xrightarrow{J}G\xrightarrow{J^{-1}}F
\end{align*}
$F$ is a tame direct summand of $G$. In particular if a graded Fréchet space is a tame Fréchet space then it is a tame Fréchet space with respect to every equivalent grading.
\end{remark}

\subsection{Smooth tame maps}

We want to show that the autoequivalences of an exact Courant algebroid with a compact base manifold are a tame Fréchet Lie group. To be able to formulate the theory of such Lie groups we need not only understand tame Fréchet spaces but also smooth and in particular smooth tame maps of such spaces.
The advantage of tame Fréchet spaces and smooth tame maps is the existence Hamilton's Nash-Moser implicit function theorem (\cite{hamilton}). 
Here, we introduce the notion of smooth tame maps. To do so, we first introduce the notion of smooth maps of locally convex spaces and discuss some fundamental results from calculus on locally convex spaces. We then introduce the notion of smooth tame maps and zero-tame maps. Furthermore we discuss some direct results for maps of Banach spaces.

\begin{definition}
\label{Def:GateauxDifferentiability}
Let $F$ and $G$ be locally convex topological vector spaces and $U\subseteq F$ open. Let $P:U\to G$ be a continuous map. We define the \textit{derivative of $P$ at the point $f\in U$ in the direction $h\in F$} as
\begin{align*}
    DP(f)h\coloneqq \lim_{t\to 0}\frac{P(f+th)-P(f)}{t}.
\end{align*}
We call $P$ \textit{differentiable at $f$ in the direction $h$} if this limit exists. We call $P$ \textit{continuously differentiable} or $C^1$ if it is differentiable at every $f\in U$ and in every direction $h\in F$ and the resulting map $DP:U\times F\to G$ is continuous.
\end{definition}

\begin{remark}
We may sometimes specifically refer to differentiability in the sense of Definition \ref{Def:GateauxDifferentiability} as (continuous) differentiability with respect to the Gâteaux differential when discussing Banach spaces or manifolds modeled on Banach spaces. 
\end{remark}

\begin{lemma}[{\cite[Lemma II.2.3]{InfDimLieneeb}}]
\label{Lemma:DP(f)Lin}
Let $F$ and $G$ be locally convex topological vector spaces and $U\subseteq F$ open. Let $P:U\to G$ be a $C^1$-map. Then $DP(f)h$ is $\R$-linear in $h\in F$ for every $f\in U$.
\end{lemma}

\begin{definition}
Let $F_1,\ldots,F_k$ and $G$ be locally convex spaces, $U\subseteq \prod_{i=1}^k F_i$ open and let $P:U \to G$ be a continuous map. Then for $(f_1,\ldots,f_k)\in U$ and $h_i\in F_i$ with $i\in \{1,\ldots,k\}$ we define
\begin{equation*}
    D_iP(f_1,\ldots,f_i,\ldots,f_k)h_i=\lim_{t\to 0}\frac{P(f_1,\ldots,f_i+th_i,\ldots,f_k)-P(f_1,\ldots,f_i,\ldots,f_k)}{t}.
\end{equation*}
We call these the \textit{partial derivatives with respect to $f_i$} or \textit{$i$-th partial derivative}. We think of $D_iP$ again as a map $D_iP:U\times F_i \to G$. We say P is $C^1$ in $f_i$ if $D_iP$ exists and is continuous. We will sometimes also denote $D_iP$ by $D_{f_i}P$ for easier readability. 
\end{definition}

\begin{theorem}[{\cite[Proposition II.2.6]{InfDimLieneeb}}]
\label{Thm:PartialDer}
Let $F_1,\ldots,F_k$ and $G$ be locally convex spaces, $U\subseteq \prod_{i=1}^k F_i$ open, and let $P:U \to G$ be a continuous map.
Then $D_iP$ exists and is continuous for all $i\in\{1,\ldots,k\}$ if and only if $P$ is continuously differentiable. Additionally 
\begin{equation*}
    DP(f_1,\ldots,f_k)(h_1,\ldots,h_k)=\sum_{i=1}^k D_iP(f_1,\ldots,f_k)h_i
\end{equation*}
with $(f_1,\ldots,f_k)\in U$ and $(h_1,\ldots,h_k)\in \prod_{i=1}^k F_i$
\end{theorem}

\begin{remark}
Let $F$, $G_1,\ldots,G_l$, and $V$ be locally convex vector spaces, $U\subseteq F$ open and let $L$ be a map $L:U\times \prod_{i=1}^l G_i\to V$. Let $L(f,g_1,\ldots,g_l)$ be continuous, $C^n$ in $f$ for $n\in \N \cup \{\infty\}$ and linear in each $g_i$, $i\in \{1,\ldots,l\}$, then it is $C^n$. In particular its first derivative is
\begin{equation*}
    DL(f,g_1,\ldots,g_l)(h,k_1,\ldots,k_l)=D_fL(f,g_1,\ldots,g_l)h+\sum_{i=1}^nL(f,g_1,\ldots,g_{i-1},k_i,g_{i+1},\ldots,g_l)
\end{equation*}
since $D_{g_i}L(f,g_1,\ldots,g_l)k_i=L(f,g_1,\ldots,g_{i-1},k_i,g_{i+1},\ldots,g_l)$.
\end{remark}

\begin{convention}
\label{Conv:DiffOfFamOfLins}
Such a map $L(f,g_1,\ldots,g_l)$ as in the previous remark we denote by $L(f)\{g_1,\ldots,g_l\}$ and call it a $C^n$ family of multilinear maps.

If $L(f)\{g_1,\ldots,g_l\}$ is $C^1$ then we know $D_fL(f,g_1,\ldots,g_l)h$ is linear in $h$ and each $g_i$, $i\in \{1,\ldots,l\}$. This yields the following notation
\begin{equation*}
    DL(f)\{g_1,\ldots,g_l,h\}\coloneqq D_fL(f,g_1,\ldots,g_l)h
\end{equation*}
Note that this is not the total derivative but rather 
\begin{equation*}
    DL(f)\{g_1,\ldots,g_l,h\}= DL(f,g_1,\ldots,g_l)(h,0,\ldots,0).
\end{equation*}
Nevertheless for such families of (multi-) linear maps it is reasonable to call the derivative with respect to the family parameter $f$ simply the derivative of the family of (multi-) linear maps. 
\end{convention}

An easy argument shows that a linear map $L:X\to Y$ of topological vector spaces $X$ and $Y$ is continuous if and only if $L$ is continuous in $0\in X$. In the next lemma we wish to generalize this statement. This is somewhat unrelated to present discussions but we will need it later on.

\begin{lemma}
\label{Lemma:ContFamOfMultLin}
Let $X$, $V_1,\ldots,V_n$ and $Y$ be topological vector spaces, $U\subseteq X$ open and $L:U\times \prod_{i=1}^nV_i\to Y$ be a family of multilinear maps, i.e.~$L(x)\{v_1,\ldots,v_n\}$ be linear in $v_1,\ldots,v_n$ separately with $x\in U$, $v_i\in V_i$, $i=1,\ldots,n$.
We define the map $\iota_{v_I}$ as the map that inserts $v_i\in V_i$ into $L$ for every $i\in I\subseteq \{1,\ldots,n\}$, e.g.~for $n\geq 3$ we have $\iota_{v_{\{2,3\}}}L=L(-)\{-,v_2,v_3,-,\ldots,-\}$.  
The map $L$ is continuous if and only if for every subset $I\subseteq \{1,\ldots,n\}$ the map $\iota_{v_I}L$ is continuous in $x$ for all $x\in U$, and $0\in V_i$ for $i\in \{1,\ldots,n\}\setminus I$.
\end{lemma}
\begin{proof}
Let $W\subseteq Y$ be a neighbourhood of $y\coloneqq L(x)\{v_1,\ldots,v_n\}$. We define 
$N_n\coloneqq \{1,\ldots,n\}$. 
Let now $B$ be an open neighbourhood of $x$ and $U_i\subseteq V_i$ be open neighbourhoods of $0\in V_i$, $i=1,\ldots,n$. Then, we find that
\begin{align}
\label{eq:L(B)(v_1+U_1,...)}
\begin{split}
    &L(B)\{v_1+U_1,\ldots,v_n+U_n\}-L(x)\{v_1,\ldots,v_n\}\\
    &\subseteq (L(B)\{v_1,\ldots,v_n\}-L(x)\{v_1,\ldots,v_n\})
    +\sum_{ I\subsetneq N_n}(\iota_{v_I}L)(B)\{U_{j_1},\ldots,U_{j_{n-\abs{I}}}\}
\end{split}
\end{align}
with $j_1<j_2<\ldots<j_{n-\abs{I}}$ and $\{j_1,\ldots,j_{n-\abs{I}}\}=N_n\setminus I$. Because addition on $Y$ is continuous there is an open neighbourhood $V_t\subseteq Y$ of $0\in Y$ for every $t\in \N$ such that $\sum_{l=1}^t V_t\subset W$. By the conditions placed on $L$ we can now choose for any $t$ the open sets $B$ and $U_1,\ldots,U_n$ so small that the individual terms in equation
\eqref{eq:L(B)(v_1+U_1,...)}
lie in $V_t$ and by choosing $t$ large enough we ensure that $L(B)\{v_1+U_1,\ldots,v_n+U_n\}\subseteq W$. This proves continuity in an arbitrary point $(x,v_1,\ldots,v_n)$ and hence continuity of $L$.
\end{proof}

\begin{definition}
\label{Def:HigherDiffs}
Let $F$ and $G$ be locally convex spaces, $U\subseteq F$ open and $P:U\to G$ continuous. We define higher differentials recursively. Let $D^kP:U\times F^k\to G$ denote the $k$-th differential of $P$ for some $k\in \N$ then we define the differential of order $k+1$ as $D_1(D^kP):U\times F^{k+1}\to G$. Note that we already know $D^0P=P$ and $D^1P$. If $D^nP:U\times F^n\to G$, $n\in \N$, exists and is continuous then $P$ is called $C^n$. If $P$ is $C^n$ for every $n\in \N$ then we call $P$ \textit{smooth} or $C^\infty$.
\end{definition}

\begin{remark}
In the context of Definition \ref{Def:HigherDiffs} we define $D^{k+1}P\coloneqq D_1(D^kP)$. Note that we already know that $DP(f):F\to G$ is linear for each $f\in U$ (cf.~Lemma \ref{Lemma:DP(f)Lin}) and that we are using $D_1$ here in accordance with Convention \ref{Conv:DiffOfFamOfLins}.
\end{remark}

We will now collect some of the properties of the directional derivative and Riemann integral some of which we have already established.
\begin{proposition}[{\cite[Lemma II.2.3, Proposition II.2.5]{InfDimLieneeb}}]
\label{Prop:PropsOfDiff}
Let $F$ and $G$ be locally convex topological vector spaces, e.g.~Fréchet spaces, $U\subseteq F$ open and $P:U\to G$ a $C^1$-map. Then, if we consider $G$ as a subspace of its completion the following properties are fulfilled:
\begin{enumerate}
    \item \label{Item:DiffLin} For every $f\in U$ the map $DP(f):F\to G$ is $\R$-linear.
    \item $P$ is continuous.
\end{enumerate}
If $P:U\to G$ is additionally $C^n$ for $n\geq 1$ then the following hold:
\begin{enumerate}
\setcounter{enumi}{3}
    \item \label{Item:HigherDiffMultLin} For all $f\in U$ the map $D^nP(f):F^n\to G$ is fully symmetric and $n$-linear in $F^n$ over $\R$.
    \item If $f+[0,1]h\subseteq U$ for $f\in U$ and $h\in F$ then 
    \begin{align*}
        P(f+h)=&P(f)+DP(f)h+\ldots+\frac{1}{(n-1)!}D^{n-1}P(f)\{h,\ldots,h\}+\\
        &\frac{1}{(n-1)!}\int_0^1(1-t)^{n-1}D^nP(f+th)\{h,\ldots,h\}.
    \end{align*}
\end{enumerate}
If $P:U\to G$ is $C^n$, $n\in \N\cup \{\infty\}$, $K$ is a locally convex space, $V\subseteq G$ is open, $P(U)\subseteq V$ and $Q:V\to K$ is $C^n$, then we have the usual chain rule:
\begin{enumerate}
\setcounter{enumi}{6}    
    \item The composition $Q\circ P$ is $C^n$ and 
    \begin{align*}
        D[Q\circ P](f)h=DQ(P(f))DP(f)h.
    \end{align*}
\end{enumerate}
\end{proposition}

\begin{definition}
Let $F$ and $G$ be locally convex spaces and $U\subseteq F$ open, and let $P:U\to G$ be a map. Then we define $TP:U\times F\to G$ by $TP(u,f)=(P(u),DP(u)f)$ for $u\in U$ and $f\in F$. We call $TP$ the \textit{tangent map} of $P$.
We define the higher tangent maps $T^nP$ simply by $T^nP=T(T^{n-1}P)$ for $n>1$.
\end{definition}

\begin{remark}
Clearly a map $P$ as in the previous definition is $C^k$, $k\in \N$, if and only if $T^kP$ exists and is continuous, and $P$ is $C^\infty$ if and only if $T^nP$ exists and is continuous for every $n\in \N$. Furthermore, the chain rule simply becomes $T(Q\circ P)=TQ\circ TP$ for appropriate maps $Q$ and $P$.
\end{remark}

\begin{remark}
\label{Rem:ConvenientCalculus}
The setting of convenient calculus was developed by Kriegl and Michor in \cite{ConvSetting}. We take a quick look at this setting in the context of Fréchet spaces.
It turns out that if $F$ and $G$ are Fréchet spaces, $U\in F$ is open, and $P:U\to G$ is a map, then $P$ is smooth if and only if for every $c\in C^\infty(\R,U)$ we have $P\circ c\in C^\infty(\R,G)$. In other words for $P$ to be smooth it suffices that it maps smooth curves to smooth curves. Such a map we will call $c^\infty$-smooth. 
Since we could only find remarks on this precise statement but no explicit proofs we give a short outline of a proof and will refer to \cite{ConvSetting} for the technical details.
Let $F$ be a locally convex vector space then we denote by $c^\infty F$ the underlying vector space of $F$ equipped with the final topology with respect to all smooth curves $C^\infty(\R,F)$. Since Fréchet spaces are bornological spaces we see that if $F$ is a Fréchet space then $c^\infty F=F$ by \cite[Theorem 4.11 (1)]{ConvSetting}. Now, let $F$ and $G$ be Fréchet spaces, $U\subseteq F$ be open, and $P:U\to G$ a $c^\infty$-smooth map. Then, $P$ is continuous since for any open subset $V\subseteq G$ we find that $(P\circ c)^{-1}(V)=c^{-1}(P^{-1}(V))$ is open for every $c\in C^\infty(\R,U)$ and hence open in $c^\infty F=F$. Note that it suffices to check this on smooth curves $c\in C^\infty(\R,U)$ since $U$ is open in $c^\infty F=F$, for more details see \cite[Lemma 24.6 (1)]{ConvSetting}.
Now by \cite[Theorem 3.12, Definition 3.17, and Theorem 3.18]{ConvSetting} we see that $DP:U\times F\to G$ exists and is $c^\infty$-smooth and hence continuous. By iteratively applying this argument we find that $P$ is smooth. The reverse statement that $P$ maps smooth curves to smooth curves if $P$ is smooth is obvious. 

More generally, one can similarly show that for any Fréchet manifolds (cf.~Section \ref{Section:TameFrechetManifolds}) $\Mcal$ and $\Ncal$ a map $P:\Mcal\to \Ncal$ is smooth if any only if for any $c\in C^\infty(\R,\Mcal)$ we have $P\circ c\in C^\infty(\R,\Ncal)$.
\end{remark}

\begin{definition}[{\cite[Part II, Definition 2.2.1]{hamilton}}]
Let $F$ and $G$ be graded Fréchet spaces, and $U\subseteq F$ be open. We say a map $P:U\to G$ satisfies a \textit{tame estimate of degree $k$ and base $b$} if there are constants $C_n>0$ such that
\begin{align*}
    \nnorm{P(x)}{n}\leq C_n (1+\nnorm{x}{n+k})
\end{align*}
for all $x\in U$, and all $n\geq b$.
\end{definition}

\begin{definition}[{\cite[Part II, Definition 2.1.1, and p.~143]{hamilton}, \cite[Definition 2.2.11]{RWang}}]
Let $F$ and $G$ be graded Fréchet spaces and $U\subseteq F$ be open. We call a map $P:U\to G$ 
\begin{enumerate}
    \item \textit{tame} if it satisfies a tame estimate on a neighbourhood of each $x\in U$, where the degree and base may vary from neighbourhood to neighbourhood,
    \item \textit{smooth tame} if it is smooth and $D^nP$ is tame for every $n\in \N$, and
    \item \textit{zero-tame} if it is smooth and for each $x\in U\times F^n$ there is a neighbourhood on which $D^nP$ satisfies a tame estimate of degree $0$ and base $0$ for each $n\in \N$.
\end{enumerate}
Here we use the convention $D^0P\coloneqq P$.
\end{definition}

\begin{remarks}
Clearly, compositions of tame, smooth tame or zero-tame maps are again tame, smooth tame or zero-tame, respectively. Zero-tame maps are, by definition, smooth tame. They are particularly useful as the composition of a tame map with a zero-tame map does not change the degree or base of the tame estimates.

It is easy to see that the notion of tame linear maps introduced in Definition \ref{Def:TameLinMap} agrees with the notion of being a tame map which is also linear (cf.~\cite[Part II, Theorem 2.1.5]{hamilton}). In particular, tame linear maps are smooth tame.
\end{remarks}

\begin{proposition}[{\cite[Part I, Theorem 5.1.3]{hamilton}}]
\label{Prop:LipschitzCont}
Let $(F,(\nnorm{-}{F,k})_{k\in \N})$ be a graded Fréchet space, let $G$ be a Fréchet space, $U\subseteq F$ be open, and $P:U\to G$ be $C^1$. Then for every continuous seminorm $\nnorm{-}{G}$ on $G$ and every $f_0\in U$ there exists a $k\in \N$ and an $\epsilon>0$ such that if $\nnorm{f_i-f_0}{F,k}<\epsilon$, $i=1,2$, then
\begin{align*}
    \nnorm{P(f_1)-P(f_2)}{G}\leq C \nnorm{f_1-f_2}{F,k}
\end{align*}
for some constant $C>0$.
\end{proposition}

\begin{remark}
We note for Proposition \ref{Prop:LipschitzCont} that we have made minor adjustments to Hamilton's result \cite[Part I, Theorem 5.1.3]{hamilton}. The proof, however, stays practically the same which is why we will not repeat it here.
\end{remark}

\begin{corollary}[of Proposition \ref{Prop:LipschitzCont}]
\label{Cor:SmToBanSmTame}
Let $(B,\nnorm{-}{B})$ be a Banach space, $F$ be a graded Fréchet space, and $U\subseteq F$ be open. Then any smooth map $P:U\to B$ is smooth tame.
\end{corollary}
\begin{proof}
Let $f_0\in U$ then by Proposition \ref{Prop:LipschitzCont} we can find an $n\in \N$ and an $\epsilon>0$ such that if $f\in U$ and $\nnorm{f-f_0}{n}<\epsilon$ then there are constants $C>0$ such that
\begin{align*}
    \nnorm{P(f)}{B}&\leq \nnorm{P(f)-P(f_0)}{B}+\nnorm{P(f_0)}{B}\\
    &\leq C(\nnorm{f}{n}+\nnorm{f_0}{n})+\nnorm{P(f_0)}{B}\leq C(1+\nnorm{f}{n})
\end{align*}
where we additionally used the triangle inequality and absorbed the terms $\nnorm{f_0}{n}$ and $\nnorm{P(f_0)}{B}$ into the constant $C$.
\end{proof}

\begin{corollary}
\label{Cor:MapsOfBanachSpacesZeroTame}
Any smooth map of Banach spaces (in the sense of Gâteaux differentiability) is zero-tame.
\end{corollary}

\subsection{Sections of Vector Bundles and Differential Operators}
\label{Section:SectionsOfVBsAndDiffOps}

Let $V\to M$ be a vector bundle over a compact manifold $M$. Its smooth sections $\Gamma(V)$ are fundamental objects in differential geometry. It turns out that this space can be equipped with a topology making it a tame Fréchet space. These types of spaces will be the model spaces for all tame Fréchet Lie groups discussed in this thesis, in particular, the group of autoequivalences for an exact Courant algebroid over a compact base manifold. We introduce the tame Fréchet space structure of $\Gamma(V)$ and show some fundamental results.
Moreover, we discuss nonlinear partial differential operators and show that they are smooth tame maps as maps of spaces of smooth section. Most of the smooth tame maps considered in this thesis can either fully or partially be built from such operators.

\begin{theorem}[{\cite[Part II, Corollary 1.3.9]{hamilton} and Appendix \ref{Appendix:EquivNormsOnGamma(V)}, Proposition \ref{Prop:equiv of norms on Gamma(V)}}]
\label{VBsectTameFre}
Let $M$ be a compact manifold and $V\to M$ a vector bundle. Choose covariant derivatives and positive definite bundle metrics on $V$ and $T^*M$. Let $\nabla^j\sigma$ denote the induced $j$-th covariant derivative of $\sigma\in \Gamma(V)$, and let $|-|$ denote the induced fiberwise norms on $(T^*M)^{\tens k}\tens V$, $k\in \N$, as in Example \ref{Ex:SmoothSectionsFrechetSpace}.

Then there is the grading on $\Gamma(V)$ obtained by the norms
\begin{align}
\label{VBsectGlobNorm}
    \nnorm{\sigma}{n}\coloneqq \sup_{x\in M}\sum_{j=0}^n|\nabla^j\sigma|
\end{align}
for $n\in \N$.

Let $\{\Tilde{U}_k\}_{k=1,\ldots,N}$, $N\in \N$, be a finite open cover of $M$ by coordinate charts $(\Tilde{U}_k,\varphi_k)$, let $(U_k)_{k=1,\ldots,N}$ be a refinement such that $\Bar{U}_k\subset \Tilde{U}_k$, and let $(\theta_k)_{k=1,\ldots,N}$ be a partition of unity subordinate to $(U_k)_{k=1,\ldots,N}$ on $M$. Over each of the $\Tilde{U}_k$, $k=1,\ldots,N$, we assume that a local frame $(e_{k,i})_{i=1,\ldots,\mathrm{rk}(V)}$ of $V$ exists. We define a grading on $\Gamma(V)$ by 
\begin{equation}
\label{VBsectLocNorm}
    \nnorm{\sigma}{n}'=\sum_{k=1}^N\sum_{i=1}^r\sup_{|\alpha|\leq n}\sup_{x\in M}|D^{\alpha}(\theta_k\sigma^i)|,
\end{equation}
where we abbreviate $\theta_k\sigma^i\coloneqq (\theta_k\sigma^{k,i})\circ\varphi_k^{-1}$ where $\sigma^{k,i}$ are the coefficient functions of $\sigma$ with respect to $e_{k,i}$.
Then the norms $\nnorm{-}{n}$ and $\nnorm{-}{n}'$ as defined above are equivalent for each $n\in \N$. It follows that $\Gamma(V)$ equipped with either grading is a tame Fréchet space and the gradings are equivalent of degree $0$ and base $0$.
\end{theorem}

\begin{remark}
\label{Rem:NormsOnSectOfVBIndepOfChoices}
Note that the equivalence of the norms $\nnorm{-}{n}$ (cf.~equation \eqref{VBsectGlobNorm}) and $\nnorm{-}{n}'$ (cf.~equation \eqref{VBsectLocNorm}) for each $n\in \N$ shows that both norms are independent of the choices that were made in their definitions up to the equivalence of norms. In particular, the induced topology and grading making $\Gamma(V)$ into a tame Fréchet space is independent of these choices up to the equivalence of gradings of degree $0$ and base $0$.
\end{remark}

\begin{lemma}
\label{Lemma:Gamma(UsubE)Open}
Let $E\xrightarrow{\pi} M$ be a vector bundle, $M$ be a compact manifold, and $U\subseteq E$ be open. Then $\Gamma(U\subseteq E)\coloneqq \{s\in \Gamma(E)|~s(M)\subset U\}$ is an open subset of $\Gamma(E)$.
\end{lemma}
\begin{proof}
It can happen that $\Gamma(U\subseteq E)=\emptyset$ and thus it is trivially open.
So, let us assume that $\Gamma(U\subseteq E)\neq\emptyset$. Then by the compactness of $M$ we can choose $\epsilon >0$ small enough such that $s+\epsilon B_{\nnorm{-}{0}}\subseteq \Gamma(U\subseteq E)$ where $B_{\nnorm{-}{0}}$ denotes the open Ball centered at $0$ of radius $1$ with respect to $\nnorm{-}{0}$ (cf.~equation \eqref{VBsectGlobNorm}).
\end{proof}

\begin{proposition}
\label{Prop:EplusFsect}
Let $E\to M$ and $F\to M$ be vector bundles and $M$ a compact manifold. Then $\Gamma(E\oplus F)\cong\Gamma(E)\times\Gamma(F)$ is a zero-tame isomorphism.
\end{proposition}
\begin{proof}
We equip $E$ and $F$ with connections $\nabla^E$ and $\nabla^F$ respectively. We also choose bundle metrics $\scal{-}{-}_E$ and $\scal{-}{-}_F$ on $E$ and $F$, respectively. We call the induced norms on $(\nnorm{-}{n}^E)_{n\in \N}$ on $\Gamma(E)$, and $(\nnorm{-}{n}^F)_{n\in \N}$ on $\Gamma(F)$ (cf.~equation \eqref{VBsectGlobNorm}). Similarly, we denote by $(\nnorm{-}{n}^{E\oplus F})_{n\in \N}$ the norms induced by $\nabla^E\oplus \nabla^F$ and $\scal{-}{-}_E\oplus\scal{-}{-}_F$ (cf.~equation \eqref{VBsectGlobNorm}).
Let $\alpha\in \Gamma(E)$, and $\beta\in\Gamma(F)$ then a simple calculation shows that
\begin{align*}
    \nnorm{\alpha\oplus\beta}{n}^{E\oplus F}\leq \nnorm{\alpha}{n}^E+\nnorm{\beta}{n}^F\tn{,}\\
    \nnorm{\alpha}{n}^E\leq\nnorm{\alpha\oplus\beta}{n}^{E\oplus F}\tn{, and}\\
    \nnorm{\beta}{n}^F\leq \nnorm{\alpha\oplus\beta}{n}^{E\oplus F}.
\end{align*}
This shows that the map $\Gamma(E\oplus F)\ni\alpha\oplus \beta\mapsto (\alpha,\beta)\in \Gamma(E)\times\Gamma(F)$ is a linear zero-tame isomorphism.
\end{proof}

\begin{corollary}
\label{Cor:FsubbundleE->FtameDirectSummand}
Let $M$ be a compact manifold, $V\to M$ be a vector bundle and $F\to M$ be a vector subbundle. Then $\Gamma(F)$ is a tame direct summand of $\Gamma(E)$.
\end{corollary}
\begin{proof}
Using a bundle metric on $V$ we have the isomorphism of vector bundles  
$V\cong F\oplus F^\bot$. By Proposition \ref{Prop:EplusFsect} we find that $\Gamma(V)\cong\Gamma(F)\times \Gamma(F^\bot)$ is a zero-tame isomorphism. Note that clearly $\Gamma(V)\cong \Gamma(F\oplus F^\perp)$. More generally, if $W\to M$ is a vector bundle, and $V\cong W$ as vector bundles, then there is a zero-tame isomorphism $\Gamma(V)\cong \Gamma(W)$. This can be achieved by simply transporting the defining data on $V$ of the grading on $\Gamma(V)$ to $W$ and use the induced grading on $\Gamma(W)$ (cf.~Theorem \ref{VBsectTameFre}).
\end{proof}

\begin{definition}[{\cite[Part I, Example 3.1.7]{hamilton}}]
\label{Def:VBop}
Let $M$ be a compact manifold and let $E\to M$ and $F\to M$ be vector bundles. We call an operator $P:\Gamma(U\subseteq E)\to \Gamma(F)$ a \textit{nonlinear vector bundle operator} if there is a smooth fiber-preserving map
$p:U\to F$ such that $P$ is the induced map $P(s)=p\circ s$ for all $s\in \Gamma(U\subseteq E)$. 
\end{definition}

\begin{lemma}
\label{VBopSmoothTame}
Nonlinear vector bundle operators $P:\Gamma(U\subseteq E)\to \Gamma(F)$ as defined above are smooth zero-tame maps.
\end{lemma}
\begin{proof}
By \cite[Part I, Example 3.3.3]{hamilton} $P$ is $C^1$.
Hamilton uses this result in \cite[Part I, Example 3.6.5]{hamilton} to show inductively that $P$ is smooth. We note that in his proof Hamilton shows that up to the isomorphism $\Gamma(E\oplus E)\cong \Gamma(E)\times\Gamma(E)$, which we had shown in Proposition \ref{Prop:EplusFsect} is a zero-tame isomorphism, $TP$ is again a vector bundle operator and hence $C^1$.
Lastly, by \cite[Part II, Theorem 2.2.6]{hamilton} we see that $P$ is tame and locally fulfills tame estimates of degree and base $0$. Since $TP$ is a nonlinear vector bundle operator up to a zero-tame isomorphism, we see that the same holds for $TP$. Iteratively, we find that $P$ is zero-tame.
\end{proof}

\begin{remark}
In \cite[Part I, Example 3.6.5]{hamilton} Hamilton additionally shows that the tangent map $TP$ of a nonlinear vector bundle operator $P:\Gamma(U\subseteq E)\to \Gamma(F)$ defined by a smooth fiber-preserving map $p:(U\subseteq E)\to F$ is the nonlinear vector bundle operator defined by the vertical tangent map $T_vp:U\oplus E\to F\oplus F$ after the zero-tame identification of $\Gamma((U\oplus E)\subseteq (E\oplus E))\cong \Gamma(U\subseteq E)\times \Gamma(E)$ (cf.~Proposition \ref{Prop:EplusFsect}), where $U\oplus E\coloneqq \{v\oplus w\in E\oplus E|~v\in U,~w\in E\}$. If $p:U\to F$ is locally given by $(x,v)\mapsto (x,f(x,v))$ then $T_vp$ is locally given by $(x,v,w)\mapsto (x,f(x,v),D_2f(x,v)w)$.
\end{remark}

\begin{lemma}
Let $E\to M$ be a vector bundle over a compact manifold $M$. 
For every $k\in\N$ the $k$-jet extension map $j^k:\Gamma(E)\to \Gamma(J^kE)$ is a tame linear operator of order $k$ and base $0$. 
Here, $J^kE$ denotes the $k$-jet vector bundle with respect to sections of $E$.
\end{lemma}
\begin{proof}
We define the norms $(\nnorm{-}{E,n}')_{n\in \N}$ on $E$ as in equation \eqref{VBsectLocNorm} with respect to charts, a partition of unity on $M$, and local trivializations of $E$. We use the same charts and partition of unity on $M$, and induced local trivializations on $J^kE$ to define $\nnorm{-}{J^k,n}'$ as in equation \eqref{VBsectLocNorm} for every $n\in \N$. Hence, we trivially see that $\nnorm{j^ks}{J^kE,n}'\leq C \nnorm{s}{E,n+k}'$ as the components of $j^ks$ are precisely the partial derivatives of the components of $s$ of order up to $k$.   
\end{proof}

The following definition aims to make Hamilton's definition given in \cite[p.~146]{hamilton} more precise.
\begin{definition}
\label{Def:NonlinPartDiffOp}
Let $E\to M$ and $F\to M$ be vector bundles over a compact manifold $M$. Let $U\subseteq \Gamma(E)$ be open. We call a map $P:U\to \Gamma(F)$ a \textit{nonlinear partial differential operator of order $k\in \N$} if for every $s\in U$ there is an open neighbourhood $s\in U_s\subseteq U$, an open subset $\Tilde{U}_s\subseteq J^kE$ with $j^k(U_s)\subseteq \Gamma(\Tilde{U}_s\subseteq J^kE)$, and a nonlinear vector bundle operator $\Tilde{P}_s:\Gamma(\Tilde{U}_s\subseteq J^kE)\to \Gamma(F)$ such that $\left.P\right|_{U_s}=\Tilde{P}_s\circ \left.j^k\right|_{U_s}$, and $k$ is the smallest natural number for which this holds. 
\end{definition}

\begin{proposition}[{\cite[Part II, Corollary 2.2.7]{hamilton}}]
\label{Prop:NonlinPartDiffOpsAreSmoothTame}
A nonlinear partial differential operator $P:(U\subseteq \Gamma(E))\to \Gamma(F)$ is a smooth tame map. In particular if $P$ is of order $k\in \N$ then for every $n\in \N$ the maps $D^nP$ and $T^nP$ locally fulfill tame estimates of base $0$ and degree $k$.     
\end{proposition}
\begin{proof}
In the previous lemma we have seen that the jet extension map $j^k$ is indeed a tame linear map of base $0$ and degree $k$. In Lemma \ref{VBopSmoothTame} we have seen that (nonlinear) vector bundle operators are zero-tame maps. Hence, by definition nonlinear partial differential operators are smooth tame as they can be locally written as compositions of smooth tame maps. The fact that $D^nP$ and $T^nP$ locally fulfill tame estimates of base $0$ and degree $k$ for every $n\in \N$ follows immediately from the chain rule, tame linearity of $j^k$ of base $0$ and degree $k$, and the zero-tameness of nonlinear vector bundle operators.     
\end{proof}

\subsection{The space of smooth sections of vector bundles as the inverse limit of Sobolev spaces}

In this section $V\to M$ and $E\to M$ denote vector bundles over a compact manifold $M$. Here $V$ and $E$ may be complex or real vector bundles and linear maps of their sections are understood to be real or complex linear in the respective cases. If $F\to M$ is a real vector bundle we denote by $F^{\mathbb C}$ its complexification with fibers $F^{\mathbb C}_p=F_p\tens_\R\mathbb C$, $p\in M$.

There is a well-understood theory of the Sobolev spaces of sections of $V$. We want to compare that description to the description of $\Gamma(V)$ as a tame Fréchet space.

This chapter is mostly based on \cite[Chapter IV]{Wells07}. Wells develops the theory of Sobolev spaces of smooth sections of smooth complex vector bundles over real compact manifolds. However, the statements generally also hold for real vector bundles. This is usually quite clear. For the sake of completeness we will prove or sketch how one can deduce the respective statements for real vector bundles from their complex counterpart.
After recalling the relevant theory of Sobolev spaces we show that the Fréchet space $\Gamma(V)$ is the inverse limit of the associated Sobolev spaces and that certain useful linear operators of Sobolev spaces correspond to tame linear operators of the corresponding tame Fréchet spaces.

\begin{definition}[{\cite[p.~110f]{Wells07}}]
Let $M$ be a compact $n$-dimensional manifold, and let $\pi:V\to M$ be a real or complex vector bundle of rank $\mathrm{rk}(V)=r$. Let $(U_\alpha,\phi_\alpha)_{\alpha =1,\ldots,l}$, $l\in \N$, be a finite cover of $M$ by trivializations $\phi_\alpha:\pi^{-1}(U_\alpha)\to \Tilde{U}_\alpha \times \K^r$, with $\K=\R,\mathbb C$,  and $r=\mathrm{rk}(V)$ such that $\phi_\alpha$ covers a chart $\varphi_\alpha:U_\alpha\to\Tilde{U}_\alpha\subseteq \R^n$, with $n=\mathrm{dim}(M)$. Finally, let $\theta_\alpha$ be a partition of unity subordinate to the open cover $(U_\alpha)_{\alpha=1,\ldots,l}$ of $M$. Then, we define the $k$-th Sobolev norm as
\begin{equation}
\label{SobolevNorm}
    \nnorm{\sigma}{k}^S=\sum_\alpha \nnorm{(\phi_\alpha)_*(\theta_\alpha \sigma)}{k,\R^n,\K^r}^S
\end{equation}
for $\sigma\in\Gamma(V)$ where 
\begin{equation*}
    \nnorm{f}{k,\R^n,\K^r}^S=\int_{\R^n}\abs{\hat f(y)}^2(1+\abs{y}^2)^kdy
\end{equation*}
for $f:\R^n\to \K^r$ a compactly supported smooth function. Here,  
\begin{equation*}
    \hat f(y)=\frac{1}{(2\pi)^n}\int_{\R^n}e^{ixy}f(x)\d x
\end{equation*}
is the component-wise Fourier transform.    
\end{definition}

\begin{remark}
\label{Rem:SobNormsCompSupp}
Note that the Sobolev norms are well defined since $\theta_\alpha$ and so $(\phi_\alpha)_*(\theta_\alpha\sigma)$ have compact support, as $\mathrm{supp}(\theta_\alpha)$ is closed and $M$ is compact.
\end{remark}

\begin{definition}
For any $k\in \mathbb Z$ we denote by $W^k(V)$ the completion of $\Gamma(V)$ with respect to the $k$-th Sobolev norm and call it the \textit{$k$-th Sobolev space of (smooth sections of) $V$}. 
\end{definition}

\begin{convention}
When writing $\Gamma(V)\subset W^k(V)$ for any $k\in \mathbb Z$ we explicitly mean the vector space $\Gamma(V)$ equipped with the topology induced by $\nnorm{-}{k}^S$.
\end{convention}

\begin{proposition}[{\cite[§4, p.~110 and §4 Proposition 1.2]{Wells07}}]
\label{Prop:Rellich}
The Sobolev spaces $W^k(V)$, $k\in \mathbb Z$, are Hilbert spaces and the natural vector space inclusions
\begin{align*}
    \ldots\supset W^{-1}(V)\supset W^0(V) \supset W^1 (V)\supset\ldots
\end{align*}
are continuous and, in fact, compact as maps of the Banach spaces $W^k(V)\hookrightarrow W^{k-1}(V)$, $k\in\mathbb Z$. 
\end{proposition}

\begin{remark}
Wells proves these results for complex vector bundles in \cite[§4, Proposition 1.2]{Wells07}. However, one can easily see that using the complexification $V^{\mathbb C}$ of $V$ we find that $W^k(V)\cong \overline{\Gamma(V)}\subset W^k(V^\mathbb C)$. Then, the vector space inclusion as a map of Banach spaces $W^k(V)\hookrightarrow W^{k-1}(V)$ is a compact operator since it is simply the restriction of the inclusion map $W^k(V^\mathbb{C})\to W^{k-1}(V^\mathbb C)$ to $\overline{\Gamma(V)}\subset W^k(V^{\mathbb C})$. 
\end{remark}

\begin{definition}
We denote by $\Gamma_k(V)$ for $k\in \N$ the Banach space of $k$-times continuously differentiable sections of $V$ equipped with the topology induced by the norm $\nnorm{-}{k}$ (cf.~equation \eqref{VBsectGlobNorm}) or equivalently the norm $\nnorm{-}{k}'$ (cf.~equation \eqref{VBsectLocNorm}). 
\end{definition}

\begin{proposition}[{\cite[§4, Proposition 1.1]{Wells07} and \cite[Chapter X, §4, Theorem 4]{IndexTheorem}}]
\label{Prop:SobolevEmbThm}
For $s>n/2+k+1$ we find that on the level of vector spaces $W^s(V)\subset \Gamma_k(V)$, that is, every Cauchy sequence in $\Gamma(V)$ with respect to $\nnorm{-}{s}^S$ converges in $\Gamma_k(V)$. Additionally, the inclusion map $W^s(V)\hookrightarrow \Gamma_k(V)$ is continuous.
\end{proposition}

\begin{remark}
Although the statements in \cite[§4, Proposition 1.1]{Wells07} and \cite[Chapter X, §4, Theorem 4]{IndexTheorem} are formulated for complex vector bundles the real case follows immediately since
$W^s(V)\subset W^s(V^{\mathbb C})\hookrightarrow \Gamma_k(V^{\mathbb C})$ has image in $\Gamma_k(V)$.
\end{remark}

\begin{lemma}
\label{Lemma:SobolevVsC^k}
Let $\nnorm{-}{k}'$ be as in equation \eqref{VBsectLocNorm}, and let $\nnorm{-}{k}^S$ be the $k$-th Sobolev norm (cf.~equation \eqref{SobolevNorm}), $k\in \N$. Let $n=\mathrm{dim}(M)$. Then for every $k\in \N$ and $d>n$ we find a constant $C_k>0$ such that 
\begin{align}
\label{C^k<Sob}
    \nnorm{\sigma}{k}^S\leq C_k\nnorm{\sigma}{k+d}' 
\end{align}
and
\begin{align}
\label{Sob<C^k}
    \nnorm{\sigma}{k}' \leq C_k\nnorm{\sigma}{k+s}^S
\end{align}
with $s>[n/2]+1$, and $\sigma \in \Gamma(V)$.
\end{lemma}
\begin{proof}
We write $\K=\R,\mathbb C$ depending on whether $V\to M$ is a real or complex vector bundle.
We first show the inequality \eqref{C^k<Sob}. Let $f$ be a function $f:\R^n\to \K^r$, for $r=\mathrm{rk}(V)$, with support in some fixed compact subset $K\subset \R^n$. Note that by Remark \ref{Rem:SobNormsCompSupp} this is a valid assumption when considering sections of vector bundles. We have
\begin{align*}
    \nnorm{f}{k,\R^n,\K^r}^S&=\int_{\R^n}\abs{\hat f(y)}^2\frac{(1+\abs{y}^2)^{k+d}}{(1+\abs{y}^2)^d}dy\stackrel{i)}{\leq} C\sum_{\abs{\alpha}\leq (k+d)}\int_{\R^n}\abs{\widehat{D^\alpha f}(y)}^2\frac{1}{(1+\abs{y}^2)^d}dy\\
    &\leq C\sum_{\abs{\alpha}\leq (k+d)}\int_{\R^n}\abs{\int_{\R^n}e^{ixy}D^\alpha f(x)\d x}\frac{1}{(1+\abs{y}^2)^d}dy\\
    &\leq C\sum_{\abs{\alpha}\leq (k+d)}\int_{\R^n}\left(\int_{\R^n}\abs{e^{ixy}}\abs{D^\alpha f(x)}\d x\right)\frac{1}{(1+\abs{y}^2)^d}dy\\
    &\stackrel{ii)}{\leq} C\sum_{\abs{\alpha}\leq (k+d)}\sup_{x\in \R^n}\abs{D^\alpha f(x)}\int_{\R^n}\frac{1}{(1+\abs{y}^2)^d}dy \stackrel{iii)}{\leq} C\sup_{\abs{\alpha}\leq (k+d)}\sup_{x\in \R^n}\abs{D^\alpha f(x)}\\
    &\leq C \sum_{i=1}^r\sup_{\abs{\alpha}\leq (k+d)}\sup_{x\in \R^n}\abs{D^\alpha f_i(x)}.
\end{align*}
Here $f=(f_1,\ldots,f_r)$, and $\alpha=\alpha_1\ldots \alpha_n$ denotes a multi-index. For $i)$ we have used that 
\begin{equation*}
    (1+\abs{y})^l\leq \left(1+\sum_{j=1}^n \abs{y_j}\right)^l\leq C\sum_{\abs{\alpha}\leq l}|y_1|^{\alpha_1}\ldots|y_d|^{\alpha_n}=C\sum_{\abs{\alpha}\leq l}|y^\alpha|,
\end{equation*}
and that since $f$ has compact support the Fourier transform takes derivatives of the form $D^\alpha f_j$, $j=1,\ldots,r$, to terms of the form $i^{\abs{\alpha}}y^\alpha f_j$. By $y^\alpha$ we, of course, mean $y_1^{\alpha_1}\ldots y_n^{\alpha_n}$.
For $ii)$ we have used that $f$ has support $K\subset \R^n$ which is compact and hence has finite volume. 
For $iii)$ we have used that the integral over $(1+\abs{y})^{-d}$ is finite for $d>n$ as can easily be checked in spherical coordinates.
The general case for sections $\sigma \in \Gamma(V)$ now follows as a direct consequence.

The inequality \eqref{Sob<C^k} follows from the Sobolev embedding Theorem (Proposition \ref{Prop:SobolevEmbThm}).
\end{proof}

\begin{theorem}
\label{Theorem:C^k-Norm=Sob-Norm}
The space $(\Gamma(V),(\nnorm{-}{k}^S)_{k\in \mathbb{Z}})$ is a tame Fréchet space and the grading is tamely equivalent to the grading $(\nnorm{-}{k}')_{k\in \N}$ (cf.~equation \eqref{VBsectLocNorm}) and the grading $(\nnorm{-}{k})_{k\in \N}$ (cf.~equation \eqref{VBsectGlobNorm}).   
\end{theorem}
\begin{proof}
This is a direct consequence of Lemma \ref{Lemma:SobolevVsC^k}. In the notation of that lemma we saw that for $l\coloneqq \mathrm{max}\{d,s\}$ we have $\nnorm{\sigma}{k}^S\leq C_k\nnorm{\sigma}{k+l}'$ and $\nnorm{\sigma}{k}'\leq C_k\nnorm{\sigma}{k+l}'$ for all $\sigma \in \Gamma(V)$, with some constants $C_k>0$. This shows that $(\Gamma(V),(\nnorm{-}{k}^S)_{k\in \mathbb{N}})$ is a graded Fréchet space and, in fact, that it is tame. Now adding the seminorms $\nnorm{-}{k}^S$ for $k\leq 0$ does not change the induced topology. In fact, the space $(\Gamma(V),(\nnorm{-}{k}^S)_{k\in \mathbb{Z}})$ is still a tame Fréchet space with a grading tamely equivalent of base $0$ and degree $l$ to $(\nnorm{-}{k}')_{k\in \N}$.

We already know that $\nnorm{-}{n}'$ and $\nnorm{-}{n}$ (cf.~equation \eqref{VBsectGlobNorm}) are equivalent for every $n\in \N$ (cf. Theorem \ref{VBsectTameFre}, or Appendix \ref{Appendix:EquivNormsOnGamma(V)} Proposition \ref{Prop:equiv of norms on Gamma(V)} for a detailed proof). Hence, the grading $(\nnorm{-}{k}^S)_{k\in \mathbb{Z}}$ is also tamely equivalent, of base $0$ and degree $l$, to $(\nnorm{-}{k})_{k\in \N}$.
\end{proof}

\begin{remark}
\label{Rem:InvLim}
Recall that if $(I,<)$ is some directed partially ordered set, $(X_i)_{i\in I}$ a family of topological spaces, and $f_{ij}:X_j\to X_i$ for $i\leq j\in I$ are continuous maps such that $f_{ii}=\mathrm{Id}_{X_i}$ and $f_{ik}=f_{ij}\circ f_{jk}$ for $i\leq j\leq k\in I$ then we call $((X_i)_{i\in I},(f_{jk})_{j\leq k\in I})$ an inverse system of topological spaces.

Recall that the inverse limit of an inverse system $((X_i)_{i\in I},(f_{jk})_{j\leq k\in I})$ of topological spaces is defined as the topological space
\begin{align*}
    \varprojlim_{i\in I}X_i\coloneqq \left\{\left.(x_i)_{i\in I}\in\prod_{i\in I}X_i\right|~ x_i=f_{ij}(x_j)~~\forall i\leq j\in I\right\}
\end{align*}
equipped with the subspace topology relative to $\prod_{i\in I}X_i$. It is easy to see that if $X_i$ is Hausdorff for each $i\in I$, then $\varprojlim_{i\in I}X_i\subseteq \prod_{i\in I}X_i$ is a closed subspace.

If $X_i$ is a topological vector space for each $i\in I$ and the maps $f_{ij}$ are continuous linear maps for all $i\leq j\in I$, then we call $((X_i)_{i\in I},(f_{jk})_{j\leq k\in I})$ an inverse system of topological vector spaces.
For such an inverse system of topological vector spaces $\varprojlim_{i\in I}X_i$ is a linear subspace of $\prod_{i\in I}X_i$. Hence, it is a topological vector space.
If additionally all $X_i$, $i\in I$, are locally convex vector spaces, then $\varprojlim_{i\in I}X_i$ is a locally convex vector space (cf.~Lemma \ref{Lemma:ProductOfLCAndFreSpaces}).
Furthermore, if additionally $I$ is countable and $X_i$ is a Fréchet space for each $i\in I$, then $\varprojlim_{i\in I}X_i$ is a Fréchet space. Indeed, we know that by Hausdorffness $\varprojlim_{i\in I}X_i\subseteq \prod_{i\in I}X_i$ is a closed subspace. Furthermore, the countable product of Fréchet spaces is a Fréchet space (cf.~Lemma \ref{Lemma:ProductOfLCAndFreSpaces}) and closed subspaces of Fréchet spaces are Fréchet spaces (cf. Remark \ref{Rem:ClosedSubspAndQuotientFrechet}). 
\end{remark}

In the following theorem the inverse system of sobolev spaces is given by $(W^i(V))_{i\in \mathbb Z}$ with the maps $f_{ij}=\iota_{ij}:W^j(V)\hookrightarrow W^i(V)$ for $i\leq j$ being the inclusion maps.
\begin{theorem}
\label{VBsectInvLimW^k}
The tame Fréchet space $(\Gamma(V),(\nnorm{-}{k}^S)_{k\in \mathbb{Z}})$ is isomorphic to the (inverse) limit $\varprojlim_{k\in \mathbb Z} W^k(V)$, $k\in \mathbb Z$, in the category of Fréchet spaces.   
\end{theorem}
\begin{proof}
The inclusion maps $\iota_{ij}:W^j(V)\hookrightarrow W^i(V)$ for $i\leq j\in \mathbb Z$ are continuous and even compact (cf.~Proposition \ref{Prop:Rellich}). Also, clearly the vector space inclusion maps $\pi_i:(\Gamma(V),(\nnorm{-}{k}^S)_{k\in \mathbb{Z}})\hookrightarrow W^i(V)$ are continuous for every $i\in \mathbb Z$, and $\pi_i=\iota_{ij}\circ \pi_j$ for $i\leq j\in \mathbb Z$.

For each $s\in \Gamma(V)$ it is clear that $(\pi_i(s))_{i\in \mathbb Z}\in \varprojlim_{k\in \mathbb Z} W^k(V)\subseteq \prod_{k\in \mathbb Z}W^k(V)$.
Furthermore, if $(s_i)_{i\in \mathbb Z}\in \varprojlim_{k\in \mathbb Z} W^k(V)\subseteq \prod_{k\in \mathbb Z}W^k(V)$, then we have $s_i=\iota_{ij}(s_j)$ for $i\leq j\in \mathbb Z$. Hence, by the Sobolev embedding theorem (Proposition \ref{Prop:SobolevEmbThm}) we see that there is some section $s\in\Gamma(V)$ such that $s_i=\pi_i(s)$ for every $i\in \mathbb Z$. So, the map 
$\pi:\Gamma(V)\to \varprojlim_{k\in \mathbb Z} W^k(V)$ defined by $\pi(s)=(\pi_i(s))_{i\in \mathbb Z}$ for $s\in \Gamma(V)$ is a vector space isomorphism.
By the universal property of the cartesian product the map
$\pi:(\Gamma(V),(\nnorm{-}{k}^S)_{k\in \mathbb{Z}})\to \varprojlim_{k\in \mathbb Z} W^k(V)$ is continuous and by the open mapping theorem (cf.~Theorem \ref{Thm:OpenMapping}) it is a linear homeomorphism. Note, that as explained in Remark \ref{Rem:InvLim} the space $\varprojlim_{k\in \mathbb Z} W^k(V)$ is a Fréchet space as the inverse limit of Fréchet (more specifically Hilbert) spaces.
\end{proof}

In the following definition, the inclusion of a lower bound $b$ is usually unnecessary for the purposes of this thesis. However, this definition more closely reflects the definition of tame linear operators, as we will see in Theorem \ref{Theorem:OP_k=Gr_k}.
\begin{definition}[{\cite[p.~114]{Wells07}}]
By $\OP_k(V,E)$, $k\in \mathbb Z$, we denote the space of linear maps $L:\Gamma(V)\to \Gamma(E)$ such that there is a continuous linear extension
\begin{align*}
    \hat{L}_n:W^n(V)\to W^{n-k}(E)
\end{align*}
for every $n\geq b$ for some $b\in \mathbb Z$.
\end{definition}

\begin{theorem}
\label{Theorem:OP_k=Gr_k}
The operators in $\OP_k(V,E)$ are precisely the tame linear operators $\Gamma(V)\to \Gamma(E)$ of degree $k\in \mathbb Z$ and arbitrary base with respect to the Sobolev gradings on $\Gamma(V)$ and $\Gamma(E)$.
Moreover, they are tame linear as maps of Fréchet spaces $L:\Gamma(V)\to \Gamma(E)$, where $\Gamma(V)$ and $\Gamma(E)$ are equipped with either of the gradings $(\nnorm{-}{n})_{n\in \N}$ (cf.~equation \eqref{VBsectGlobNorm}), $(\nnorm{-}{n}')_{n\in \N}$ (cf.~equation \eqref{VBsectLocNorm}) or $(\nnorm{-}{n}^S)_{n\in \mathbb Z}$ (cf.~equation \eqref{SobolevNorm}).
\end{theorem}
\begin{proof}
Let $L\in \OP_k(V,E)$. This means precisely that $L$ extends to bounded operators of Banach spaces
$\hat L_n:W^n(V)\to W^{n-k}(E)$ for $n\geq b$ with some $b\in \mathbb Z$. So in particular $L:(\Gamma(V),\nnorm{-}{n}^S)\to (\Gamma(E),\nnorm{-}{n-k}^S)$ is bounded for every $n\geq b$ which is precisely the condition for tame linearity of degree $k$ and base $b$.

On the other hand if $L:(\Gamma(V),\nnorm{-}{n}^S)\to (\Gamma(E),\nnorm{-}{n-k}^S)$ is bounded for every $n\geq b$ for some $b\in \mathbb Z$ we get that $\pi_{n-k}\circ L:(\Gamma(V),\nnorm{-}{n}^S)\to W^{n-k}(E)$ is continuous. Here,
$\pi_i:(\Gamma(E),(\nnorm{-}{j}^S)_{j\in \mathbb{Z}})\to W^i(E)$, $i\in \mathbb Z$, again denotes the continuous inclusion of the vector space. 
Since $W^k(V)$ and $W^{n-k}(E)$ are Banach spaces and $\Gamma(V)\subset W^k(V)$ is dense $\pi_{n-k}\circ L$ extends to a unique bounded linear map $\hat L_n:W^n(V)\to W^{n-k}(E)$ for every $n\in \mathbb Z$ by Lemma \ref{Lemma:ExtOfMapsIntoCompHausdTVS}.
Hence, $L\in \OP_k(V,E)$.

By Theorem \ref{Theorem:C^k-Norm=Sob-Norm} we get that they are also tame linear with respect to either of the other two gradings. In fact $\Gamma(V)$ and $\Gamma(E)$ do not even need to be equipped with the same type of grading.
\end{proof}

\begin{remark}
\label{Rem:RealOpsComplexifications}
Let $V\to M$ and $E\to M$ be real vector bundles. Then there is an isomorphism $\{L\in \mathrm{OP}_k(V^{\mathbb C},E^{\mathbb C})|~L(\Gamma(V))\subseteq \Gamma(E)\}\cong \mathrm{OP}_k(V,E)$ for every $k\in \N$. This is an immediate consequence of Theorem \ref{Theorem:OP_k=Gr_k}. One can also easily prove this directly. 
\end{remark}

\begin{definition}
\label{Def:L2InnerProd}
Let $d\mu$ be some density on $M$. Let $\scal{-}{-}_E$ be a Riemannian or Hermitian bundle metric on the real or complex vector bundle $E\to M$ respectively. Then we can induce an inner product $(-,-)_E:\Gamma(E)\times \Gamma(E)\to \K$, with $\K=\R,\mathbb C$ depending on whether $E$ is real or complex, by defining
\begin{align*}
    (e,s)_E=\int_M \scal{e}{s}_Ed\mu
\end{align*}    
for $e,s\in \Gamma(E)$.
Note that linearity in the first argument, (conjugate) symmetry, and positive definiteness are immediate consequences of the properties of $\scal{-}{-}_E$ and the integral. 
\end{definition}

\begin{definition}
\label{Def:AdjointOp}
Let $V\to M$ and $E\to M$ be $\K=\R,\mathbb C$ vector bundles. Let $L:\Gamma(V)\to \Gamma(E)$ be a $\K$-linear map. We call a $\K$-linear map $S:\Gamma(E)\to \Gamma(V)$ a \textit{formal adjoint} of $L$ if $(Lv,e)_E=(v,Se)_V$ for every $v\in \Gamma(V)$ and $e\in \Gamma(E)$.
\end{definition}

\subsection{Hodge Theory}

Throughout this section let $M$ be a compact manifold, and $E\to M$ and $F\to M$ be real or complex vector bundles.

We have seen that the autoequivalences of an exact Courant algebroid can be modeled by diffeomorphisms and closed differential $2$-forms (cf.~Theorem \ref{AuteqCharacterizationTheorem}). In this section our goal is to show, using methods from Hodge theory, that the space of closed differential ($2$-)forms on a compact manifold is a tame Fréchet space. Moreover, we know that the compatibility condition between a diffeomorphism and the differential $2$-form defining an autoequivalence involves the differential of that $2$-form. So, it would be very useful to have a smooth tame inverse of the de Rham differential, in a suitable sense. Indeed, we will show that the Greens operator from Hodge theory provides the desired map.

We start by briefly reviewing elliptic complexes of differential operators. In this context we discuss the de Rham complex of $M$. We introduce the usual operators from Hodge theory and show that they are tame linear. Furthermore, we show that the Hodge decomposition is an isomorphism of tame Fréchet spaces and, in particular, we show that the closed differential forms on $M$ are a tame direct summand of the tame Fréchet space of differential forms on $M$ for any degree.

\begin{definition}
A \textit{linear differential operator $P:\Gamma(E)\to \Gamma(F)$ of order $k\in \N$} is a partial differential operator of order $k$ (cf.~Definition \ref{Def:NonlinPartDiffOp}) such that it can be written as a composition $L\circ j^k$ with $L\in \Gamma(\mathrm{Hom}(J^kE,F))$, and $j^k:\Gamma(E)\to \Gamma(J^kE)$ being the $k$-jet extension map.

In other words, a \textit{linear differential operator of order $k\in \N$} is any linear operator $P:\Gamma(E)\to \Gamma(F)$ such that for any local frames $(e_i)_{i=1,\ldots,\mathrm{rk(E)}}$ of $E$ and $(s_i)_{i=1,\ldots,\mathrm{rk}(F)}$ of $F$ over some open subset $U\subseteq M$ and summing over repeating indices we have that 
\begin{align*}
    P(\sigma^ie_i)=\sum_{|\alpha|\leq k} a_{\alpha,j}^{l}D^\alpha \sigma^js_l
\end{align*}
where $a_{\alpha,j}^l\in C^\infty(U)$ for every $l=1,\ldots,\mathrm{rk}(F)$, $j=1,\ldots,\mathrm{rk}(E)$ and $\alpha=\alpha_1\ldots \alpha_n$ being a multi-indices for $n\leq k$ and $\alpha_i=1,...,\dim(M)$, and such that at least one component $a^l_{\alpha,j}\neq 0$ for $\abs{\alpha}=k$.
\end{definition}

\begin{lemma}[{\cite[§4, Proposition 2.1 and Proposition 2.2 and ]{Wells07}}]
\label{Lemma:DiffOpsInOP}
Let $\Gamma(E)$ and $\Gamma(V)$ be equipped with $L^2$-scalar products as in Definition \ref{Def:L2InnerProd}.
Any linear differential operator $L:\Gamma(E)\to \Gamma(F)$ of order $k\in \N$ is an operator in $\mathrm{OP}_k(E,F)$ and has a unique formal $L^2$-adjoint $L^*\in \OP_k(V,E)$.
\end{lemma}
\begin{proof}
In the case of complex vector bundles this follows immediately from \cite[§4, Proposition 2.2]{Wells07}. The case of real vector bundles follows then by Remark \ref{Rem:RealOpsComplexifications}.
\end{proof}

\begin{definition}
Let $M_0\subset T^*M$ denote the image of the zero section, and let $\pi:T^*M\setminus M_0\to M$ be the restriction of the usual bundle projection. Let $E\to M$ and $F\to M$ be complex vector bundles. Let $P:\Gamma(E)\to \Gamma(F)$ be a linear differential operator of order $k$. Then, we define the \textit{principal symbol} of $P$ as $\sigma(P)\in \mathrm{Hom}(\pi^*E,\pi^*F)$ given at $v\in T^*_xM\setminus \{0\}$, $x\in M$, and $e\in E_x\cong (\pi^*E)_v$ by 
\begin{align*}
    \sigma(P)(v)e\coloneqq P\left(\frac{i^k}{k!}(g-g(x))^kf\right)(x)
\end{align*}
with $g\in C^\infty(M)$ such that $dg_x=v$, and $f\in \Gamma(E)$ such that $f(x)=e$.
\end{definition}

For a proof that the principal symbol $\sigma(P)$ is well defined and indeed in $\mathrm{Hom}(\pi^*E,\pi^*F)$, see \cite[Chapter IV, Proposition 2.3]{Wells07}.

\begin{definition}
We say that a linear differential operator $P:\Gamma(E)\to \Gamma(F)$ of degree $k\in \N$ is elliptic if for every $x\in M$ and every $v\in T_x^*M\setminus\{0\}$ the map $\sigma(P)(v):E_x\to F_x$ is an isomorphism.
\end{definition}

\begin{definition}[{\cite[§4, p.~144-146]{Wells07}}]
Let $E_i\to M$, $i\in \N$, be vector bundles.
We say that a complex of differential operators 
\begin{align*}
    \Gamma(E_0)\xrightarrow{P_0}\Gamma(E_1)\xrightarrow{P_1}\Gamma(E_2)\xrightarrow{P_2}\ldots
\end{align*}
of fixed order $n\in \N$ is an elliptic complex if the sequence of principal symbols
\begin{align*}
    0\to \pi^*E_0\xrightarrow{\sigma(P_0)}\pi^*E_1\xrightarrow{\sigma(P_1)}\pi^*E_2\xrightarrow{\sigma(P_2)}\ldots
\end{align*}
is exact. 
We denote the complex by $(E,P)$.

Moreover, after equipping $\Gamma(E_k)$ with an $L^2$-inner product as in Definition \ref{Def:L2InnerProd} we define the \textit{(Hodge) Laplace operator} in degree $k\in \N$ as
\begin{align*}
    \Delta_k\coloneqq P_{k-1}P^*_{k-1}+P_k^* P_k:\Gamma(E_k)\to \Gamma(E_k)
\end{align*}
where $P^*_k:\Gamma(E_{k+1})\to \Gamma(E_k)$ is the formal $L^2$-adjoint of $P_k$.
The Laplacian operator is a self-adjoint elliptic operator of order $2n$.

Lastly we call sections $\alpha\in \Gamma(E_k)$ with $\Delta_k\alpha=0$ \textit{harmonic} and denote the space of harmonic sections by
\begin{align*}
    \Hcal^k(E,P)\coloneqq \ker \Delta_k \subset \Gamma(E_k).
\end{align*}
\end{definition}

\begin{proposition}[{\cite[§4, Theorem 4.12, Theorem 5.2 and Theorem 5.3]{Wells07}}]
\label{Prop:HodgeRelations}
Let $(E,P)$ be an elliptic complex of fixed order $n$.
For each $k\in \N$ there are operators $G_k\in \OP_{-2n}(E_k,E_k)$ and $\Hcal_k\in \OP_{-\infty}(E_k,E_k)$ such that 
\begin{enumerate}
    \item $\Hcal^k(E,P)=\Hcal_k(E_k)$,
    \item $\dim_{\mathbb C}\Hcal^k(E,P)<\infty$,
    \item $\Delta_k G_k+\Hcal_k=G_k \Delta_k+\Hcal_k=\mathrm{Id}_{E_k}$,
    \item $P_k \Delta_k=\Delta_{k+1} P_k$, $P^*_{k} \Delta_{k+1}=\Delta_{k} P^*_k$,
    \item $P_k G_k=G_{k+1} P_k$, $P^*_k G_{k+1}=G_{k} P^*_k$,
    \item $\Hcal_k G_k=G_k \Hcal_k=0$, and
    \item \label{Item:Hd=0} $\Hcal_{k+1} P_k=P_k \Hcal_k=0$, $\Hcal_{k} P^*_k=P^*_k \Hcal_{k+1}=0$.
\end{enumerate}
\end{proposition}

\begin{remarks}
    In \cite[§4, Theorem 4.12, Theorem 5.2 and Theorem 5.3]{Wells07} Wells only established that the operators $G_k$ and $\Hcal_k$ are of order $0$. However, since $\Hcal_k$ has finite-dimensional image all Sobolev norms are equivalent on its image and so it follows that it is a smoothing operator of infinite order. To get that $G_k$ is of order $-2n$ one can use the construction in \cite[Theorem 4.12]{Wells07} to see that $G_k$ extends continuosly to $W^0(E_k)\to W^{2n}(E_k)$ and inductively use an elliptic estimate for $\Delta_k$ which we obtain for instance from \cite[Theorem 4.4]{Wells07} to show that $G_k\in \OP_{-2n}(E_k,E_k)$.

\label{Rem:HodgeOpsSmoothTame}
By Theorem \ref{Theorem:OP_k=Gr_k} the operators $P_k$,
$P^*_k$, $\Delta_k$, $G_k$ and $\Hcal_k$ are all tame linear with respect to the Sobolev norms (cf.~equation \eqref{SobolevNorm}) and the $C^k$-norms (cf.~equation \eqref{VBsectGlobNorm} and equation \eqref{VBsectLocNorm}). So, they are in particular smooth tame.
\end{remarks}

The de Rham complex is an elliptic complex of differential operators of fixed order $n=1$, see \cite[§4, Example 2.5]{Wells07}. So, the previous proposition holds for the de Rham complex of complex differential forms. However, in this case, it is easy to verify that all the operators involved are real. 
So, by Remark \ref{Rem:RealOpsComplexifications} we get the following corollary.

\begin{corollary}
    Proposition \ref{Prop:HodgeRelations} holds for the real de Rham complex
    \begin{align*}
        C^\infty(M)\stackrel{\d}{\to}\Forms{1}\stackrel{\d}{\to} \Forms{2}\stackrel{\d}{\to}\ldots
    \end{align*}
    of any compact manifold $M$.
\end{corollary}

\begin{convention}
In line with the previous exposition the formal $L^2$-adjoint of $\d_k$ is denoted by $\d^*_k$. 
From now on $G_k$, $\Hcal_k$, and $\Delta_k$ will refer to the operators from Proposition \ref{Prop:HodgeRelations} in the context of explicitly the real de Rham complex.
Moreover, we will mostly suppress the degree of the operators $\d_k$, $\d^*_k$, $\Delta_k$, $G_k$, and $\Hcal_k$, $k\in \N$, and simply write $\d$, $\d^*$, $\Delta$, $G$ and $\Hcal$, respectively.
\end{convention}

\begin{lemma}
\label{Lemma:dQex=Id}
The tame linear operator $Q\coloneqq \d^*G$ satisfies $(\d Q)(\d\alpha)=\d \alpha$ for $\alpha\in \Forms{k}$. Similarly, the tame linear operator $Q'\coloneqq \d G$ satisfies $(\d^*Q')(\d^*\alpha)=\d^*\alpha$ for $\alpha\in \Forms{k}$.
\end{lemma}
\begin{proof}
From Proposition \ref{Prop:HodgeRelations} we know that $\Delta G+\Hcal=\mathrm{Id}_{\Forms{k}}$, $\d G=G\d$, $\d^*G=G\d^*$, and $\Hcal \d=0$. Hence, $\d\alpha=\Delta G\d\alpha=\d \d^*Gd\alpha+\d^*G\d^2\alpha=\d Q\d\alpha$ for every $\alpha\in \Forms{k}$. 
Similarly, we have
$\d^*\alpha=\Delta G\d^*\alpha=\d G(\d^*)^2\alpha+\d^*\d G\d^*\alpha=\d^*Q'\d^*\alpha$
for $\alpha \in \Forms{k+1}$. 
\end{proof}

\begin{convention}
We fix the notation $Q=\d^*G$ and $Q'=\d G$ for the tame linear operators introduced in the previous lemma.
\end{convention}

\begin{lemma}
\label{Lemma:ExFormsClosedSubsp}
In every degree $k\in \N$ the spaces of exact differential forms, $\d^*$-exact differential forms, closed differential forms, $\d^*$-closed differential forms, and harmonic differential forms are closed subspaces of the tame Fréchet space $\Forms{k}$ respectively.
\end{lemma}
\begin{proof}
Since $\d$, $\d^*$, and $\Delta$ are tame linear operators their kernels are closed subspaces. So, the spaces of closed, $\d^*$-closed, and harmonic $k$-forms are closed subspaces of $\Forms{k}$ for every $k\in \N$ respectively.

To show that the subspaces of exact and $\d^*$-exact forms are closed in $\Forms{k}$ we show that they are complete as subspaces of the metrizable space $\Forms{k}$. 
To see completeness, let $(\alpha_i)_{i\in \N}$ be a Cauchy sequence in $\ExForms{k}\subseteq \Forms{k}$ with limit
$\lim_{i\to\infty}\alpha_i\eqqcolon \alpha \in \Forms{k}$. Using Lemma \ref{Lemma:dQex=Id} and the continuity of $\d Q$ we have
\begin{align*}
    \d(Q\alpha)=(\d Q)\left(\lim_{i\to \infty}\alpha_i\right)=\lim_{i\to \infty}(\d Q)(\alpha_i)=\lim_{i\to \infty}\alpha_i=\alpha.
\end{align*}
Hence, $\alpha\in \ExForms{k}$ which proves completeness.
Analogously, one can show completeness of the subspace of $\d^*$-exact differential forms using instead the operator $Q'$ and $\d^*$ in place of $Q$ and $\d$.
\end{proof}

\begin{lemma}
For each $k\in \N$ the exact differential forms, the $\d^*$-exact differential forms, and the harmonic differential forms are tame Fréchet spaces. Specifically, they are tame direct summands of the tame Fréchet space $\Forms{k}$.
\end{lemma}
\begin{proof}
Since all of these subspaces are closed it is clear that they are graded Fréchet spaces. 
By Proposition \ref{Prop:HodgeRelations} and Lemma \ref{Lemma:dQex=Id} the following maps are tame linear maps that compose to the identity
\begin{align*}
    \ExForms{k}\hookrightarrow \Forms{k}\xrightarrow{\d Q}\ExForms{k}\\
    \d^*\Forms{k+1}\hookrightarrow\Forms{k}\xrightarrow{\d^*Q'}\d^*\Forms{k+1}\\
    \Hcal^k(M)\hookrightarrow\Forms{k}\xrightarrow{\Hcal}\Hcal^k(M).
\end{align*}
Hence, these spaces are tame direct summands of $\Forms{k}$ (cf.~Proposition \ref{Prop:TameDirSumBySequence}).
\end{proof}

The maps above naturally induce maps
\begin{align*}
    \Hcal^k(M)\times\ExForms{k}\times \d^*\Forms{k+1}&\to \Forms{k}\\
    (\alpha,\beta,\gamma)&\mapsto \alpha+\beta+\gamma
\end{align*}
and 
\begin{align*}
    \Forms{k}&\to \Hcal^k(M)\times\ExForms{k}\times \d^*\Forms{k+1}\\
    \om&\mapsto (\Hcal w,\d Q \om,\d^*Q'\om)
\end{align*}
using the previous lemma and Proposition \ref{Prop:HodgeRelations} and Lemma \ref{Lemma:dQex=Id} we see that these tame linear maps are inverse to each other implying the following proposition.

\begin{proposition}
\label{Prop:HodgeDecompTameIso}
The Hodge decomposition
$\Forms{k}\cong\Hcal^k(M)\times \ExForms{k}\times d^*\Forms{k+1}$ is a tame isomorphism of tame Fréchet spaces for each $k\in \N$.
\end{proposition}

\begin{corollary}
\label{Cor:ClFormsTameDirSumm}
The space of closed differential forms $\ClForms{k}\subseteq \Forms{k}$ is a tame Fréchet space and tamely isomorphic to the tame Fréchet space $\Hcal^k(M)\times \ExForms{k}$ for each $k\in \N$.
\end{corollary}
\begin{proof}
Let $\om\in\ClForms{k}$. Then there is some $\alpha\in \Hcal^k(M)$, $\beta\in \Forms{k-1}$, and
$\gamma\in \Forms{k+1}$ such that $\om=\alpha+\d\beta+\d^*\gamma$. Then $0=\d\om=\d \d^*\gamma$ and so 
$0=(\d \d^*\gamma,\gamma)_{L^2}=(\d^*\gamma,\d^*\gamma)_{L^2}$ where $(-,-)_{L^2}$ denotes the respective $L^2$-inner product (cf.~Definition \ref{Def:L2InnerProd}). So $\d^*\gamma=0$. Since clearly $\Hcal^k(M)\subset \ClForms{k}$ and 
$\ExForms{k}\subset \ClForms{k}$ we see that the tame isomorphism
$\Forms{k}\cong\Hcal^k(M)\times \ExForms{k}\times \d^*\Forms{k+1}$ from Proposition \ref{Prop:HodgeDecompTameIso} restricts to a tame isomorphism of the tame Fréchet space $\Hcal^k(M)\times \ExForms{k}$ and the graded Fréchet space $\ClForms{k}$. Hence, $\ClForms{k}$ is a tame Fréchet space by Remark \ref{Rem:DirSumm&ProdAreTame}.
\end{proof}

\section{Tame Fréchet manifolds}
\label{Section:TameFrechetManifolds}

It is our goal to show that the group of autoequivalences of an exact Courant algebroid over a compact base manifold is a tame Fréchet Lie group. That means that we need to understand infinite-dimensional manifolds and Lie groups, in particular, tame Fréchet manifolds and tame Fréchet Lie groups.

In this section we develop in a very general language some of the theory of infinite-dimensional manifolds, and Lie groups and their Lie algebras. We discuss certain spaces of smooth maps as quintessential examples of tame Fréchet manifolds and tame Fréchet Lie groups. We discuss in particular the diffeomorphism group of a compact manifold as a tame Fréchet Lie group. Furthermore, we compare two notions of smooth families of smooth maps. Lastly, we show that the pullback of tensor fields along diffeomorphism is a smooth tame map of tame Fréchet manifolds.

\subsection{Infinite-dimensional manifolds}

To understand the theory of tame Fréchet Lie groups we first need to understand some of the theory of infinite-dimensional manifolds. Here, we lay out some very general definitions of different types of infinite-dimensional manifolds and present some basic results regarding them. This section is largely based on \cite{InfDimLieneeb} and \cite{hamilton}.

\begin{definition}[{\cite[Definition II.3.1]{InfDimLieneeb}}]
Let $\Mcal$ be a Hausdorff topological space. Let $\Ucal\subseteq \Mcal$ be an open subset, let $F$ be some locally convex vector space, and let $V\subseteq F$ be open. As in the finite-dimensional case we call a homeomorphism $\psi:\Ucal\to V$ a \textit{chart of $\Mcal$} and denote it by $(\Ucal,\psi)$.  

We call two charts $(\Ucal,\psi)$ and $(\mathcal{V},\varphi)$ smoothly compatible if either $\Ucal\cap \Vcal=\emptyset$ or the coordinate transition $\psi\circ \varphi^{-1}:\varphi(\Ucal \cap \Vcal)\to \psi(\Ucal \cap \Vcal)$ is smooth.

If there is a family of pairwise smoothly compatible charts $((\Ucal_\alpha,\psi_\alpha))_{\alpha \in A}$, for some index set $A$, with values in locally convex vector spaces $F_\alpha$, and if $(\Ucal_\alpha)_{\alpha\in A}$ is an open covering of $\Mcal$, then we call $((\Ucal_\alpha,\psi_\alpha))_{\alpha \in A}$ a \textit{smooth atlas} of $\Mcal$. 

A \textit{smooth structure} on $\Mcal$ is a maximal smooth atlas of $\Mcal$, that is, an atlas such that any chart of $\Mcal$ that is smoothly compatible to any chart in the atlas is already in the atlas. 

The space $\Mcal$ together with a smooth structure is called a locally convex manifold.
\end{definition}

\begin{remark}
Clearly, any smooth atlas on $\Mcal$ induces a unique smooth structure containing it.    
\end{remark}

\begin{definition}{\cite[p.~146]{hamilton}}
Let $\Mcal$ be a locally convex manifold. Let there be an atlas of $\Mcal$ such that all charts have image in (tame) Fréchet spaces and have smooth (tame) coordinate transition maps, then we call that atlas a smooth (tame) Fréchet atlas.

By a smooth (tame) Fréchet structure we mean a maximal smooth (tame) Fréchet atlas. 

A Hausdorff topological space together with a smooth (tame) Fréchet structure is called a (tame) Fréchet manifold. 
\end{definition}

\begin{remark}
Clearly any smooth (tame) Fréchet atlas induces a unique smooth (tame) Fréchet structure containing it.
\end{remark}

\begin{remark}
Just as we can define a smooth tame Fréchet atlas, we can also define a zero-tame Fréchet atlas as a smooth tame Fréchet atlas whose chart transition functions are zero-tame. This then leads to the notions of a zero-tame Fréchet structures and zero-tame manifolds. It is clear that any zero-tame Fréchet atlas induces a unique zero-tame structure containing it.
\end{remark}

\begin{example}
\label{Example:ManifoldsModelledOnBanachSpaces}
If $\mathcal{M}$ is a locally convex manifold with an atlas for which its charts take values in Banach spaces then this is a zero-tame atlas (cf.~Corollary \ref{Cor:MapsOfBanachSpacesZeroTame} and Example \ref{Ex:BanachSpacesAreTameFre}). Moreover, smooth maps of such manifolds are zero-tame with respect to the induced zero-tame structure (cf.~Corollary \ref{Cor:MapsOfBanachSpacesZeroTame}).
\end{example}

\begin{definition}
Let $\Mcal$ and $\Ncal$ be locally convex or (zero-)tame Fréchet manifolds. A map $f:\Mcal\to \Ncal$ is called \textit{smooth} if for every $x\in \Mcal$ there is a chart around $x$ and a chart around $f(x)\in \Ncal$ such that the local representative of $f$ in these charts is smooth. If $\Mcal$ and $\Ncal$ are (zero-)tame Fréchet manifolds and around each $x\in \Mcal$ and $f(x)$ are respective charts such that the local representative of $f$ is smooth (zero-)tame then we call $f$ a \textit{smooth (zero-)tame map between (zero-)tame Fréchet manifolds}.
\end{definition}

\begin{convention}
\label{Conv:ScalManifolds}
Because the following definitions and results are sensible in the context of locally convex manifolds as well as (tame) Fréchet manifolds and zero-tame Fréchet manifolds, we want to formulate them as general as possible. To this end, we denote by $\Scal$ the category of either locally convex manifolds and smooth maps, or Fréchet manifolds and smooth maps, or tame Fréchet manifolds and smooth tame maps, or zero-tame Fréchet manifolds and zero-tame maps.

We will refer to the objects of the category $\Scal$ as $\Scal$-manifolds and to the morphisms as $\Scal$-smooth maps.

We will also refer to the respective types of vector spaces, i.e.~locally convex vector spaces, Fréchet spaces, tame Fréchet spaces, or zero-tame Fréchet spaces as $\Scal$-vector spaces.

If $\Mcal$ is an $\Scal$-manifold then by charts we always mean charts in the $\Scal$-smooth structure on $\Mcal$.
\end{convention}

\begin{definition}[{\cite[Part I, Definition 4.2.1]{hamilton}}]
Let $\Ncal \subseteq \Mcal$ be a subset of an $\Scal$-manifold $\Mcal$. We say that $\Ncal$ is an (embedded) \textit{$\Scal$-submanifold of $\Mcal$} if around each $p\in \Ncal$ there is an open subset $\Ucal\subseteq \Mcal$ and a chart $\varphi:\Ucal \to W\times V$ of $\Mcal$, where $W$ and $V$ are $\Scal$-vector spaces, such that $\varphi(\Ucal\cap \Ncal)=\varphi(\Ucal)\cap (W\times \{0\in V\})$.

We call such a chart an \textit{($\Scal$-)submanifold chart} and we call a collection of $\Scal$-submanifold charts whose domains cover $\Ncal$ an \textit{($\Scal$-)submanifold atlas}. 
\end{definition}

\begin{example}
Let $\Mcal$ and $\Ncal$ be $\Scal$-manifolds. Then, clearly, $\Mcal\cong \Mcal\times \{y\}$ and $\Ncal\cong \{x\}\times \Ncal$ are embedded $\Scal$-submanifolds of $\Mcal \times\Ncal$ for any $x \in \Mcal$, $y\in \Ncal$ respectively.
\end{example}

\begin{example}
\label{Ex:GraphSubmfd}
Let $F$ and $G$ be $\Scal$-vector spaces, $U\subseteq F$ be open, and
$P:U\to G$ be an $\Scal$-smooth map. Then the graph $\left\{(x,P(x))\in U \times G\right\}$ is an embedded $\Scal$-submanifold of $U\times G$. In fact, we can simply define the map 
$\varphi:U\times G\to U\times G$ by $\varphi(x,g)=(x,g-P(x))$ for $(x,v)\in U\times G$. This map is clearly $\Scal$-smooth and has the $\Scal$-smooth inverse $\varphi^{-1}(x,g)=(x,g+P(x))$. Hence, it is an $\Scal$-diffeomorphism. In particular, $\varphi(x,v)=(x,0)$ is equivalent to $v=P(x)$. So, $\varphi$ is a submanifold chart.

By locally identifying $\Scal$-manifolds with open subsets of $\Scal$-vector spaces it immediately follows that the graph of any $\Scal$-smooth map $P:\Mcal\to \Ncal$ between $\Scal$-manifolds $\Mcal$ and $\Ncal$ is again an embedded $\Scal$-submanifold of $\Mcal\times \Ncal$.
\end{example}

\begin{definition}
Let $\mathcal{B}$ and $\Mcal$ be $\Scal$-manifolds and let $\pi:\mathcal B\to \Mcal$ be a surjection. As in the finite-dimensional case we call $\mathcal{B}\xrightarrow{\pi}\Mcal$ an \textit{$\Scal$-fiber bundle} if
\begin{itemize}
    \item there is an open cover $(\Ucal_\alpha)_{\alpha \in A}$ of $\Mcal$, for some index set $A$, such that there is an $\Scal$-diffeomorphism $\phi_\alpha:\pi^{-1}(\Ucal_\alpha)\xrightarrow{\cong}\Ucal_\alpha \times \mathcal F_\alpha$ with $\mathcal F_\alpha$ being some $\Scal$-manifold, and
    \item $\pi$ corresponds to the canonical projection $\Ucal_\alpha \times \mathcal{F}_\alpha\to \Ucal_\alpha$.
\end{itemize}
We call such $\Scal$-diffeomorphisms $\phi_\alpha$ satisfying the above condition \textit{$\Scal$-local trivializations} or simply \textit{local trivializations}. We call the collection $((\Ucal_\alpha, \phi_\alpha))_{\alpha\in A}$, as above, an \textit{$\Scal$-fiber bundle atlas}. 

Let $\mathcal B\xrightarrow{\pi} \Mcal$ be an $\Scal$-fiber bundle. We say that it has \textit{typical fiber $\mathcal{F}$}, where $\mathcal F$ is some $\Scal$-manifold, if there is an $\Scal$-fiber bundle atlas $((\Ucal_\alpha,\phi_\alpha))_{\alpha\in A}$, as above, such that $\phi_\alpha:\pi^{-1}(\Ucal_\alpha)\xrightarrow{\cong} \Ucal_\alpha\times \mathcal F$ for every $\alpha\in A$.
\end{definition}

\begin{remarks}
Let $\mathcal B\xrightarrow{\pi}\Mcal$ be an $\Scal$-fiber bundle. Then clearly any fiber $\mathcal B_p \coloneqq \pi^{-1}(p)$, $p\in \Mcal$, is an embedded $\Scal$-submanifold of $\mathcal B$. 

Similarly, for any $\Scal$-smooth section $s$ of $\mathcal B\xrightarrow{\pi}\Mcal$ its image $s(\Mcal)$ is an embedded $\Scal$-submanifold. This follows as in a local trivialization of $\mathcal B$ the image of the section $s$ corresponds to the graph of an $\Scal$-smooth map of $\Scal$-manifolds which is an embedded $\Scal$-submanifold as discussed in Example \ref{Ex:GraphSubmfd}. 
\end{remarks}

\begin{definition}[{\cite[Part I, Definition 4.3.1]{hamilton}}]
Let $\mathcal{V}$ and $\Mcal$ be $\Scal$-manifolds and let $\pi:\mathcal V\to \Mcal$ be an $\Scal$-fiber bundle. As in the finite-dimensional case we call $\mathcal{V}\xrightarrow{\pi}\Mcal$ an \textit{$\Scal$-vector bundle} if
\begin{itemize}
    \item $\pi^{-1}(x)$ has the structure of a vector space for every $x\in \Mcal$,
    \item if there is an $\Scal$-fiber bundle atlas $((\Ucal_\alpha,\phi_\alpha))_{\alpha\in A}$, for some index set $A$, such that $\phi_\alpha:\pi^{-1}(\Ucal_\alpha)\xrightarrow{\cong}\Ucal_\alpha \times V_\alpha$ with $V_\alpha$ being some $\Scal$-vector space, and
    \item if the vector space structure on $\pi^{-1}(x)$ agrees with that of $V_\alpha$ for every $x\in \Ucal_\alpha$, that is, if $\left.\phi_\alpha\right|_{\pi^{-1}(x)}:\pi^{-1}(x)\to V_\alpha$ is a vector space isomorphism.
\end{itemize}
We call the collection $((\Ucal_\alpha, \phi_\alpha))_{\alpha\in A}$, as above, an \textit{$\Scal$-vector bundle atlas}.
\end{definition}

\begin{remarks}
If $\mathcal V\to \Mcal$ is an $\Scal$-vector bundle with an $\Scal$-vector bundle atlas $((\Ucal_\alpha,\phi_\alpha:\pi^{-1}(\Ucal_\alpha)\to V_\alpha))_{\alpha\in A}$ then if $\Ucal_\alpha\cap \Ucal_\beta\neq \emptyset$ the transition functions 
$\phi_\alpha\circ\phi_\beta^{-1}:(\Ucal_\alpha\cap \Ucal_\beta)\times V_\beta\to (\Ucal_\alpha\cap \Ucal_\beta)\times V_\alpha$ are $\Scal$-smooth families of $\Scal$-isomorphisms for all $\alpha,\beta\in A$.
\end{remarks}

\begin{definition}[{\cite[Definition II.3.3]{InfDimLieneeb}}]
Let $\Mcal$ be an $\Scal$-manifold then we can define its \textit{tangent bundle} $\mathcal{TM}$ as in the finite-dimensional case.
Let $\mathcal{A}$ be the smooth structure of $\Mcal$ in the category $\Scal$, then we can define the tangent bundle as
\begin{align*}
    \mathcal{TM}\coloneqq \left(\bigcup_{(\Ucal_\psi,\psi)\in \mathcal A} \psi(\Ucal_\psi)\times F_\psi\right) \bigg/\sim
\end{align*}
where $\psi:\Ucal_\psi\to F_\psi$ are charts with values in an $\Scal$-vector spaces $F_\psi$. 
Moreover, the equivalence relation $\sim$ is defined such that if $(\Ucal_\psi,\psi),(\Ucal_\phi,\phi)\in \mathcal A$ and $\Ucal_\psi \cap \Ucal_\phi\neq \emptyset$, then $(x,v)\in \psi(\Ucal_\psi\cap \Ucal_\phi)\times F_\psi$ is equivalent to $(y,w)\in \phi(\Ucal_\psi\cap\Ucal_\phi)\times F_\phi$ if $T(\phi\circ\psi^{-1})(x,v)=(y,w)$. The tangent space $\mathcal{T}_p\Mcal$ at the point $p\in \Mcal$ is the vector space with the underlying set $\{[\psi(p),v]|~(\Ucal_\psi,\psi)\in \mathcal A,\ p\in \Ucal_\psi,\ v\in F_\psi\}$ and the vector space structure $[\psi(p),v]+[\psi(p),w]\coloneqq [\psi(p),v+w]$ which hence agrees with that of $F_\psi$ for $(\Ucal_\psi,\psi)\in \mathcal A,\ p\in \Ucal_\psi$.
\end{definition}

\begin{remark}
Since we can locally identify $\mathcal{TM}$ with $\psi(\Ucal_\psi)\times F_\psi$,
$(\Ucal_\psi, \psi:\Ucal_\psi\to F_\psi)\in \mathcal A$, and the transition functions are precisely the tangent maps of the chart transition functions of $\Mcal$, and hence $\Scal$-diffeomorphisms, it is clear that this naturally gives $\mathcal{TM}$ the structure of an $\Scal$-manifold. We also get a natural projection $\mathcal{TM}\xrightarrow{\pi}\Mcal$ since the local trivial projections $\psi(\Ucal_\psi)\times F_\psi\to \psi(\Ucal_\psi)$,
$(\Ucal_\psi, \psi:\Ucal_\psi\to F_\psi)\in \mathcal A$, define a well defined global map which is hence also an $\Scal$-smooth surjection. So the tangent bundle of an $\Scal$-manifold is an $\Scal$-vector bundle.
\end{remark}

\begin{remark}
\label{Rem:TangentMapManifolds}
Let $f:\Mcal\to \Ncal$ be a $C^1$-map of $\Scal$-manifolds. We define the map $\mathcal Tf:\mathcal{TM}\to \mathcal{TN}$, as follows. Let $\phi$ and $\psi$ be charts on $\Mcal$ and $\Ncal$, respectively. Then in the induced local trivializations of $\TMcal$ and $\mathcal{TN}$ the map $\mathcal{T}f$ corresponds simply to $T(\phi\circ f\circ \psi^{-1})$. The map $\mathcal Tf$ is clearly well defined and continuous. If $f$ is $\Scal$-smooth then so is $\mathcal{T}f$.
\end{remark}

\begin{definition}
Let $\Mcal$ be an $\Scal$-manifold then we call an $\Scal$-smooth section of $\mathcal{TM}$ an $\Scal$-smooth vector field. We denote the vector space of $\Scal$-smooth vector fields by $\mathfrak{X}_\Scal(\Mcal)$. In each category $\Scal$, we denote the smooth vector fields simply by $\mathfrak{X}(\Mcal)$. For example in the zero-tame category $\mathfrak{X}_\Scal(\Mcal)$ denotes the zero-tame vector fields and $\mathfrak{X}(\Mcal)$ denotes the smooth vector fields.
\end{definition}

\begin{remark}
\label{Rem:MapsOfLocConvManifolds}
If $\mathcal{M}$ and $\mathcal{N}$ are locally convex manifolds it is indeed possible to give $C^\infty(\mathcal{M},\mathcal{N)}$ a sensible topology, the so called $C^\infty$-compact open topology. It is the initial topology with respect to the tangent maps $\mathcal{T}^k:C^\infty(\Mcal,\Ncal)\to C(\mathcal{T}^k\Mcal,\mathcal{T}^k\Ncal)$, $k\in \N$, where $C(\mathcal{T}^k\Mcal,\mathcal{T}^k\Ncal)$ is equipped with the compact open topology. 
If $(F,(\nnorm{-}{\alpha})_{\alpha\in A})$ is a locally convex vector space, then one can show that $C^\infty(\mathcal{M},F)$ equipped with the $C^\infty$-compact open topology is a locally convex vector space and its topology is induced by the family of seminorms of the form $\nnorm{f}{\alpha,k,K}\coloneqq \sup_{x\in K}\nnorm{\mathcal{T}^kf}{\alpha}$ for $k\in \N$, $K\subseteq \mathcal{T}^k \Mcal$ compact, $\alpha\in A^k$, and $f\in  C^\infty(\mathcal{M},F)$. Here we write $\nnorm{x}{\alpha}=\sum_{i=1}^k\nnorm{x_i}{\alpha_i}$ for
$\alpha=(\alpha_1,...,\alpha_k)\in A^k$ and $x=(x_1,...,x_k)\in F^k=\mathcal{T}^kF$. For more details, see \cite[Definition 2.1, Remark 2.2, and p.~208, B.6 and Lemma B.7]{schmeding_2022}.
In fact, one can additionally see that if $F$ is a Fréchet space and $\Mcal$ is modeled on metrizable locally convex vector spaces, e.g.~Fréchet spaces, then $C^\infty(\Mcal,F)$, like $F$, is still sequentially complete. One uses completenes and metrizability of $F^n$, $n\in \N$, to show that for any Cauchy sequence $(f_i)_{i\in \N}$ in $C^\infty(\Mcal,F)$, the sequence $(\mathcal{T}^nf_i)_{i\in \N}$ converges uniformly on compact sets as it is a uniform Cauchy sequence on each compact set. 
We can now consider the sequence $(\left.\mathcal T^n(f_i\right|_\Ucal))_{i\in \N}$, $n\in \N$, where $\Ucal\subseteq \Mcal$ is open in $\Mcal$ and diffeomorphic to an open subset of a metrizable locally convex vector space $G$. Since $\Ucal\times G^n\cong \mathcal{T}^n\Ucal$, $n\in \N$, is metrizable it is a $k$-space (\cite[§7, p.~231, Theorem 13]{Kelley}) and since $(\left.\mathcal T^n(f_i\right|_\Ucal))_{i\in \N}$ converges uniformly on compact sets its limit is continuous (\cite[§7, p.~231, Theorem 12]{Kelley}). So $\lim_{i\to \infty}\mathcal{T}^n(\left.f_i\right|_\Ucal)$ is continuous and finally $\lim_{i\to \infty}\mathcal{T}^nf_i$ is continuous for every $n\in \N$. 
Then, applying locally a similar argument to that presented by Hamilton in \cite[p.~38 (lemma)]{hamilton1977deformation}, we can iteratively show that $\lim_{i\to \infty}\mathcal T^n f_i=\mathcal{T}^n\lim_{i\to \infty}f_i$ for every $n\in \N$, hence showing that $\lim_{i\to \infty}f_i$ is smooth and the space is sequentially complete. However, in general, $C^\infty(\Mcal,F)$ is no longer metrizable.

If $\mathcal{V}\xrightarrow{\pi} \mathcal{M}$ is a smooth locally convex vector bundle then its sections can be given the structure of a locally convex vector space in a sensible way. Namely, let 
$((\Ucal_i,\phi_i:\Ucal_i\to F_i))_{i\in I}$, for some index set $I$, be a smooth locally convex vector bundle atlas. Then we can equip the smooth sections $\Gamma(\mathcal V)$ with the initial topology with respect to the maps $\Gamma(\mathcal{V})\to C^\infty(\Ucal_i,F_i)$, $i\in I$, that map a section to its local representative. Here $C^\infty(\Ucal_i,F_i)$ is equipped with the $C^\infty$-compact open topology making it into a locally convex space. We identify $\Gamma(\mathcal V)$ with a subspace of the locally convex space $\Pi_{i\in I}C^\infty(\Ucal_i,F_i)$. Hence, it becomes itself a locally convex space. So, in particular, for a locally convex manifold $\Mcal$ the vector fields $\mathfrak{X}(\mathcal{M})$ inherit in a natural way the structure of a locally convex vector space. For more details, see \cite[p.~216, C.7]{schmeding_2022}.
\end{remark}

\begin{remark}
\label{Rem:X(f)}
Let $\Mcal$ be a locally convex manifold and $f\in C^\infty(\Mcal,V)$ for $V$ being some locally convex vector space, and let $X$ be a smooth vector field on $\Mcal$. Then $Tf\circ X:\Mcal\to \mathcal{T}V\cong V\times V$ is smooth and hence by projecting to the second factor $V\times V\xrightarrow{\pi_2}V$ the map $(X(f))(p)\coloneqq \pi_2(Tf(X(p)))$ is a well defined smooth function in $C^\infty(\Mcal,V)$. Furthermore, if $V=\R$ it is a derivation as for $f,g\in C^\infty(\Mcal,\R)$ we have $X(fg)=\pi_2(T(fg)(X))=gX(f)+fX(g)$. The product rule easily follows from Theorem \ref{Thm:PartialDer} and the smoothness of the diagonal map $\Delta:\Mcal\to \Mcal\times\Mcal$. Note that in local coordinates $\pi_2\circ T(f)$ simply corresponds to $Df$. 
\end{remark}

Neeb formulates the following lemma as an exercise, see \cite[Exercise II.17]{InfDimLieneeb}. In the following we provide a proof. 
\begin{lemma}
\label{Lemma:X.fInj}
Let $\Mcal$ be an $\Scal$-manifold and $\Ucal\subseteq \Mcal$ be an open chart domain $\Scal$-diffeomorphic to an open subset $U\subseteq F$ of some $\Scal$-vector space $F$. Then the map 
\begin{align*}
    E:\mathfrak{X}(\Ucal)&\longrightarrow \mathrm{Der}(C^\infty(\Ucal,\R))\\
    X&\longmapsto (f\mapsto X(f))
\end{align*}
with $X(f)$ as in Remark \ref{Rem:X(f)}, and $\mathrm{Der}(C^\infty(\Ucal,\R))$ denoting the derivations on the space $C^\infty(\Ucal,\R)$, is well defined and injective.
\end{lemma}
\begin{proof}
The map $E$ is well defined as we have seen in Remark \ref{Rem:X(f)}.

Note that we may describe a vector field over $\Ucal$ as a smooth map $U\to F$.
Injectivity now follows, since for any continuous linear functional $l:F\to \R$ we have that $\left.l\right|_U\in C^\infty(U,\R)$. Moreover, let $X,X'\in C^\infty(U,F)$ such that $E(X)=E(X')$ then in particular for every $x\in U$ and every continuous linear functional $l:F\to \R$ we have
$l(X(x))=(X(\left.l\right|_U))(x)=(X'(\left.l\right|_U))(x)=l(X'(x))$. By the Hahn-Banach theorem (cf.~Corollary \ref{Cor:l(v)=l(w)Thenv=w}) $X(x)=X'(x)$ for every $x\in U$. Hence, $X=X'$, which proves the injectivity of $E$.
\end{proof}

\begin{proposition}[{\cite[Lemma II.3.6 and Proposition II.3.7]{InfDimLieneeb}}]
Let $\Mcal$ be an $\Scal$-manifold, and $X,Y\in \mathfrak{X}(\Mcal)$. Then there is a unique vector field denoted by $[X,Y]\in \mathfrak{X}(\Mcal)$ such that for any open chart domain $\Ucal\subseteq \Mcal$ we have
\begin{align*}
    [X,Y](f)=X(Y(f))-Y(X(f))
\end{align*}
for all $f\in C^\infty(\Ucal,\R)$. Furthermore, $(\mathfrak{X}(\Mcal),[-,-])$ is a Lie algebra and $(\mathfrak{X}_\Scal(\Mcal),[-,-])$ is a Lie subalgebra of $(\mathfrak{X}(\Mcal),[-,-])$.
\end{proposition}
\begin{proof}
Neeb has shown in \cite[Lemma II.3.6]{InfDimLieneeb} that $[X,Y]$ exists and is unique. Moreover, he shows in \cite[Proposition II.3.7]{InfDimLieneeb} that $[-,-]$ is indeed a Lie bracket. This is an immediate consequence of the fact that, by definition, the map $E$ from Lemma \ref{Lemma:X.fInj} is an injective map that maps local representatives of $[X,Y]$ to the commutator of the induced derivations. Hence, it only remains to prove that $[-,-]$ is closed under $\mathfrak{X}_\Scal(\Mcal)$. This follows trivially from Neeb's proof. In fact, Neeb shows in \cite[Lemma II.3.6]{InfDimLieneeb} that if $\Ucal$ is some chart domain on $\Mcal$ which is $\Scal$-diffeomorphic to some open subset $U$ of some $\Scal$-vector space $F$ then expressing $X$ and $Y$ locally as $\Scal$-smooth maps $\Tilde{X},\Tilde{Y}:U\to F$ the local representative of $[X,Y]$ viewed as an $\Scal$-smooth map $\widetilde{[X,Y]}:U\to F$ is given by $\widetilde{[X,Y]}(p)=D\Tilde{Y}(p)\Tilde{X}(p)-D\Tilde{X}(p)\Tilde{Y}(p)$, $p\in U$, which is $\Scal$-smooth.
\end{proof}

\begin{remark}
\label{Rem:LieBraVFsLocConv}
We call $[X,Y]$ as above the Lie bracket of the vector fields $X$ and $Y$.
\end{remark}

\begin{definition}
Let $\Mcal$ and $\Ncal$ be $\Scal$-manifolds, and $\phi:\Mcal\to \Ncal$ a smooth map. Then, as in the finite-dimensional case, two vector fields $X\in\mathfrak{X}(\Mcal)$ and $Y\in\mathfrak{X}(\Ncal)$ are called \textit{$\phi$-related} if $(\mathcal{T}\phi) \circ X=Y\circ \phi$.
\end{definition}

\begin{proposition}[{\cite[Lemma II.3.8]{InfDimLieneeb}}]
\label{Prop:LieBraPhiRel}
Let $\Mcal$ and $\Ncal$ be $\Scal$-manifolds, and 
$\phi:\Mcal\to \Ncal$ a smooth map. Let $X,X'\in \mathfrak{X}(\Mcal)$ and $Y,Y'\in \mathfrak{X}(\Ncal)$ and let $X$ and $Y$, and $X'$ and $Y'$ be $\phi$ related respectively. Then, $[X,X']$ and $[Y,Y']$ are $\phi$-related as well.
\end{proposition}

\subsection{Manifolds of smooth maps}

Smooth maps of finite-dimensional manifolds and sections of bundle spaces are ubiquitous in differential geometry. It turns out that these spaces, at least in the case of a compact domain, can be equipped with the structure of a zero-tame Fréchet manifold. We introduce a zero-tame manifold structure on the sections of surjctive submersions with compact codomain and apply this result to other spaces of mappings. We also discuss some types of zero-tame maps between such manifolds. This will allow us to understand the diffeomorphism group of a compact manifold, which we will discuss in the next section and which is one of the key components of autoequivalences of exact Courant algebroids (cf.~Theorem \ref{AuteqCharacterizationTheorem}). Moreover, we briefly discuss the relationship between finite-dimensional smooth manifolds and finite dimensional locally convex manifolds.

\begin{theorem}[{\cite[Part II, Theorem 2.3.1]{hamilton} and \cite[Proposition 3.2.4]{RWang}}]
\label{FiberBundleSection}
Let $M$ be a compact manifold, and $p:B\to M$ be a surjective submersion, e.g.~a fiber bundle. Then, the smooth sections $\Gamma(B)$ of $p$ form a zero-tame manifold.
\end{theorem}
\begin{proof}
We will give an outline of the proof and refer to \cite{RWang} for details.

Given a section $\sigma\in \Gamma(B)$ Wang showed in \cite[Lemma 3.2.1]{RWang} that there is an open neighbourhood $U\subseteq B$ and an open neighbourhood $V$ of the zero section in $\left.T^{\tn{vert}}B\right|_{\sigma(M)}$, by which we denote the vertical tangent bundle of $B$ along $\sigma$, such that there is a fiber-preserving diffeomorphism $\phi_\sigma:U\to V$. This induces a bijection on sections $\Phi_{\sigma}:\{\tau\in \Gamma(B)|~\tau(M)\subseteq U\}\to \Gamma(V\subseteq \left.T^{\tn{vert}}B\right|_{\sigma(M)})$. These induced maps will be the charts of $\Gamma(B)$. We note that the transition functions are then nonlinear vector bundle operators (cf.~Definition \ref{Def:VBop}) and, hence, zero-tame (cf.~Lemma \ref{VBopSmoothTame}).

For more details and a proof that the resulting topology is actually Hausdorff we refer to \cite[Proposition 3.2.4]{RWang}. 
\end{proof}

If $M$ and $N$ are manifolds and $M$ is compact then by simply choosing $B=M\times N$ and $p$ to be the projection to the first factor we can deduce the following corollary from the previous Theorem.
\begin{corollary}[{\cite[Part II, Corollary 2.3.2]{hamilton}}]
\label{Cor:SmoothMapsMfd}
Let $M$ and $N$ be manifolds and let $M$ be compact then the space $C^\infty(M,N)$ is a zero-tame Fréchet manifold.
\end{corollary}

\begin{proposition}[{\cite[Part II, Theorem 2.3.3]{hamilton}}]
\label{Prop:CompSmoothTame}
Let $M$, $M'$ and $N$ be manifolds, and let $M$ and $M'$ be compact. Then, the map 
\begin{align*}
    C: C^\infty(M',N)\times C^\infty(M,M')&\longrightarrow C^\infty(M,N)\\
    (f,g)&\longmapsto f\circ g
\end{align*}
is a smooth tame map.
\end{proposition}

\begin{proposition}
\label{Prop:FiberBundleOperators}
Let $M$ be a compact manifold, and $A\xrightarrow{\pi} M$ and $B\to M$ be surjective submersions. Let $U\subseteq A$ be open and $p:U\to B$ be a smooth fiber-preserving map. Then $\Gamma(U\subseteq A)\coloneqq \{s\in \Gamma(A)|~s(M)\subset U\}$ is an open subset of $\Gamma(A)$ and the map $P:\Gamma(U\subseteq A)\to \Gamma(B)$ given by $P(\sigma)\coloneqq p\circ \sigma$, $\sigma\in \Gamma(U\subseteq A)$, is zero-tame.
\end{proposition}
\begin{proof}
As in the case of sections of vector bundles $\Gamma(U\subseteq A)$ might be empty in which case it is trivially open.

So, we will assume that $\Gamma(U\subseteq A)\neq \emptyset$. Let $s\in \Gamma(U\subseteq A)$ then we know there is an open neighbourhood $V\subseteq A$ of $s(M)$ with a fiber-preserving diffeomorphism to an open subset $W\subseteq E$ for some vector bundle $E$ inducing a chart on $\Gamma(A)$ (cf.~Theorem \ref{FiberBundleSection}). We may assume that $V\subseteq U$. Then the subset $\Gamma(W\subseteq E)\cong \Gamma(V\subseteq A)\subseteq \Gamma(U\subseteq A)$ is open in $\Gamma(A)$ (cf.~Lemma \ref{Lemma:Gamma(UsubE)Open}) and so $\Gamma(U\subseteq A)$ is open.
Furthermore, in any choice of charts on $\Gamma(U\subseteq A)$ and $\Gamma(B)$ induced by such fiber-preserving diffeomorphisms to open subsets of vector bundles the representative of the map $P$ is induced by a smooth fiber-preserving map of open subsets of vector bundles and thus the representative is a vector bundle operator and hence zero-tame (cf.~Lemma \ref{VBopSmoothTame}). So, $P$ is zero-tame.
\end{proof}

\begin{remark}
    Note that in the previous proposition we have shown that if $U\to M$ is again a surjective submersion then $\Gamma(U)\cong \Gamma(U\subset A)$ as zero-tame manifolds, i.e. the inherent zero-tame structure of $\Gamma(U)$ and the inherited zero-tame structure of $\Gamma(U\subseteq A)$ agree.
    We will still mostly use the notation $\Gamma(U\subseteq A)$ for easier readability.
\end{remark}

\begin{convention}
\label{Conv:FiberBundleOperators}
Let $A\xrightarrow{\pi} M$ and $B\to M$ be surjective submersions and let $M$ be a compact manifold. Let $U\subseteq A$ be open and $p:U\to B$ be a smooth fiber-preserving map. Then we refer to the zero-tame map $P:\Gamma(U\subseteq A)\to \Gamma(B)$ given by $P(\sigma)\coloneqq p\circ \sigma$, $\sigma\in \Gamma(U\subseteq A)$ as a \textit{bundle operator}. 
\end{convention}

\begin{proposition}[{\cite[Lemma 3.2.9]{RWang}}]
\label{Prop:SmMapsIntoProd=ProdOfSmMaps}
Let $M$ be a compact manifold, $B_1\to M$ and $B_2\to M$ be surjective submersions and $B_1\times_MB_2$ denote their fiber product. Then, the obvious bijection between and $\Gamma(B_1\times_MB_2)$ and $\Gamma(B_1)\times \Gamma(B_2)$ is a zero-tame diffeomorphism. 
\end{proposition}

\begin{corollary}
Let $M$, $N$ and $N'$ be manifolds and $M$ compact then the obvious bijection of $C^\infty(M,N\times N')$ and $C^\infty(M,N)\times C^\infty(M,N')$ is a zero-tame diffeomorphism.
\end{corollary}

\begin{lemma}
\label{Lemma:GateauxImpliesFrechetOnFinDim}
Any continuously G\^ateaux differentiable map $f:(U\subseteq \R^n)\to \R^m$, with $n,m\in \N$ and $U\subseteq \R^n$ open, is continuously Fréchet differentiable. In particular any G\^ateaux differentiable [smooth] map between finite-dimensional manifolds is Fréchet differentiable [smooth].
\end{lemma}
\begin{proof}
We choose a basis $(e_i)_{i=1,\ldots,n}$ of $\R^n$. Then $D_if(x)e_i=\lim_{t\to 0}\frac{f(x+t e_i)-f(x)}{t}=\del_if (x)$, $x\in U$. Since this limit exists and defines a continuous map by assumption, we find that $f$ is Fréchet differentiable as all of its partial derivatives exist and are continuous.
\end{proof}

\begin{proposition}
\label{Prop:FinDimMfdsAsFrechetMfds}
Any second countable locally convex manifold which can be covered by charts with values in $\R^n$ for some $n\in \N$ is a smooth finite-dimensional manifold. Moreover, any smooth finite-dimensional manifold is a zero-tame Fréchet manifold. Maps between smooth finite-dimensional manifolds that are continuously Gâteaux differentiable are Fréchet differentiable. Additionally, any smooth map between smooth finite-dimensional manifolds is zero-tame.
\end{proposition}
\begin{proof}
By Lemma \ref{Lemma:GateauxImpliesFrechetOnFinDim} the coordinate transition functions of a locally convex manifold that can be covered by charts with values in $\R^n$, $n\in \N$, are smooth in the conventional sense. Since each chart in the smooth structure then also has to take values in $\R^n$ (cf.~Lemma \ref{Lemma:FinDimHausdTVS}) this extends to the smooth (locally convex) structure. Since by Lemma \ref{Lemma:GateauxImpliesFrechetOnFinDim} continuous Gâteaux differentiability and continuous Fréchet differentiability agree on finite-dimensional manifolds, if we additionally assume second-countability, we get that the smooth finite dimensional structure and the smooth locally convex structure on the manifold induced by the given atlas agree.

Zero-tameness of smooth finite-dimensional manifolds and smooth maps thereof follows as a special case of Example \ref{Example:ManifoldsModelledOnBanachSpaces}. 
\end{proof}

\begin{remark}
We note that any finite-dimensional manifold is zero-tame with respect to the usual Euclidean norm on $\R^n$, for $n\in \N$. However, we can equip $\R^n$ with a grading that is not zero-tamely equivalent to the trivial grading via the Euclidean norm. And so on a finite-dimensional manifold we may have different non-equivalent zero-tame structures.
\end{remark}

\begin{remark}
\label{Rem:M=Cinfty(.,M)}
Any finite-dimensional manifold $M$ is diffeomorphic to the zero-tame Fréchet manifold $C^\infty(*,M)$ where $*$ denotes the manifold consisting of only one point. In fact, let $p\in M$ then $p$ corresponds to the section $s_p$ of $*\times M\to *$ sending $\{*\}\mapsto (\{*\},p)$. Choosing a chart $(U,\varphi)$ around $p$ gives a diffeomorphism of $(*\times U)$ and some open subset $(*\times W)\subseteq (*\times \R^{\dim(M)})$. The induced map on sections under the identification of $s_p$ with $p$ is simply the map itself and is hence a chart on $C^\infty(*,M)$. So, indeed the manifold structure of $C^\infty(*,M)$ as a tame manifold is precisely the same as that of $M$ as a finite-dimensional smooth manifold.
\end{remark}

\subsection{Infinite-dimensional Lie Groups and their Lie algebras}

In this section, we define tame Fréchet Lie groups and describe how we can, as in the finite-dimensional case, obtain a Lie algebra structure on the tangent space at the identity element from the Lie algebra of (left-invariant) vector fields over the tame Fréchet Lie group. Lastly, we present the notion of a regular Lie group. Despite our focus being on tame Fréchet Lie groups, this section also covers locally convex, and zero-tame Fréchet Lie groups. To accommodate these different types of manifolds we continue to use the language of $\Scal$-manifolds and $\Scal$-smooth maps introduced in Convention \ref{Conv:ScalManifolds}. The following section is mostly based on \cite[Section III.1]{InfDimLieneeb}.

\begin{definition}[{\cite[Definition III.1.1]{InfDimLieneeb}, \cite[p.~148]{hamilton}}]
Let $(\Gcal,\mu,V)$ be a group with a group operation $\mu:\Gcal\times \Gcal\to \Gcal$ and inversion map $V:\Gcal\to \Gcal$. If $\Gcal$ is additionally an $\Scal$-manifold, and if $\mu$ and $V$ are $\Scal$-smooth maps, then we call $\Gcal$ an $\Scal$-Lie group. 
\end{definition}

\begin{convention}
Let $(\Gcal,\mu,V)$ be an $\Scal$-Lie group, and let $g\in \Gcal$ then we write
$L_g\coloneqq \mu(g,-):\Gcal\to \Gcal$ and $R_g\coloneqq \mu(-,g):\Gcal\to \Gcal$ for the left and right translations respectively. These are, of course, $\Scal$-diffeomorphisms for each $g\in \Gcal$.

Moreover, we will frequently just write $\Gcal$ instead of $(\Gcal,\mu,V)$ and denote $gh\coloneqq\mu(g,h)$ and $g^{-1}\coloneqq V(g)$ for $g,h\in \Gcal$.
\end{convention}

\begin{example}
\label{Ex:C(M,G)}
Let $G$ be a finite-dimensional Lie group and $M$ be a compact manifold then the tame Fréchet manifold $C^\infty(M,G)$ (cf.~Corollary \ref{Cor:SmoothMapsMfd}) is a zero-tame Fréchet Lie group via pointwise multiplication and pointwise inversion. Note that these maps are bundle operators and hence zero-tame (cf.~Convention \ref{Conv:FiberBundleOperators} and Proposition \ref{Prop:FiberBundleOperators}). Here, we view multiplication as being induced by the smooth map
$M\times G\times G\ni (p,g,h)\mapsto (p,gh)\in M\times G$, where we use that there is a zero-tame diffeomorphism 
$C^\infty(M,G\times G)\cong C^\infty(M,G)\times C^\infty(M,G)$ (cf.~Proposition \ref{Prop:SmMapsIntoProd=ProdOfSmMaps}).
\end{example}

\begin{definition}[{\cite[Definition III.1.5]{InfDimLieneeb}}]
Let $\Gcal$ be an $\Scal$-Lie group with identity element $e\in \Gcal$. Then a smooth vector field $X\in \mathfrak{X}(\Gcal)$ is called left-invariant if $X(g)=\mathcal{T}L_g(X_e)$ for all $g\in\Gcal$. Equivalently, $X$ is left-invariant if $X\circ L_g=\mathcal{T}L_g\circ X$.
\end{definition}

\begin{remark}
Naturally, any left-invariant vector field $X$ on an $\Scal$-Lie group $\Gcal$ is in $\mathfrak{X}_\Scal(\Gcal)$ as with respect to local trivializations around $e$ and $g$ we may express it, suppressing the trivializations, as $X(g)=(g,D_2\mu(g,e)X_e)=D\mu(g,e)(0,X_e)$ which is is $\Scal$-smooth in $g$ since $\mu$ is $\Scal$-smooth.
\end{remark}

\begin{corollary}[of Proposition \ref{Prop:LieBraPhiRel}]
Let $\Gcal$ be an $\Scal$-Lie group. Then, the Lie bracket $[X,Y]$ (cf.~Remark \ref{Rem:LieBraVFsLocConv}) of two left-invariant vector fields $X$ and $Y$ is again left-invariant. In particular, the space of left-invariant vector fields forms a Lie subalgebra of $\mathfrak{X}_\Scal(\Gcal)$.    
\end{corollary}
\begin{proof}
By definition $X$ and $Y$ are $L_g$ related to themselves, respectively, for every $g\in \Gcal$. Hence, by Proposition \ref{Prop:LieBraPhiRel} the vector field $[X,Y]$ is $L_g$ related to itself for every $g\in \Gcal$, i.e.~left-invariant.
\end{proof}

\begin{remark}
\label{Rem:TeGLieAlg}
Let $\Gcal$ be an $\Scal$-Lie group with identity element $e\in \Gcal$, and let $\mathfrak{X}_L(\Gcal)$ denote the left-invariant vector fields on $\Gcal$. Exactly as in the finite-dimensional case, it is clear that the map $\mathfrak{X}_L(\Gcal)\ni X\mapsto X_e\in \mathcal T_e\Gcal$ is a vector space isomorphism. Note that for each $v\in \mathcal{T}_e\Gcal$ we have a left-invariant vector field $X^v(g)\coloneqq \mathcal{T}L_g(v)$, $g\in \Gcal$, so the induced Lie bracket on $\mathcal{T}_e\Gcal$ is $[u,v]=[X^u,X^v](e)$ for $u,v \in \mathcal T_e\Gcal$. 

This has the advantage that $\mathcal T_e\Gcal$ has the structure of an $\Scal$-vector space. Indeed, we will see that $[-,-]:\mathcal T_e\Gcal\times \mathcal T_e\Gcal\to \mathcal T_e\Gcal$ is an $\Scal$-smooth bilinear map in Theorem \ref{Thm:ad=[-,-]}.
\end{remark}

\begin{definition}
Let $\Gcal$ be an $\Scal$-manifold with identity element $e\in \Gcal$. Then we denote by $\mathrm{Lie}(\Gcal)$ the Lie algebra $(\mathcal T_e\Gcal,[-,-])$ defined in the previous remark and call it the \textit{Lie algebra of $\Gcal$}.
\end{definition}

\begin{convention}
Let $\Gcal$ and $\Hcal$ be $\Scal$-Lie groups, and let $\phi:\Gcal\to \Hcal$ be a morphism of $\Scal$-Lie groups, i.e.~an $\Scal$-smooth group homomorphism. Then, we denote by
$\mathrm{Lie}(\phi)\coloneqq \left.\mathcal{T}\phi\right|_{\mathrm{Lie}(\Gcal)}:\mathrm{Lie}(\Gcal)\to \mathrm{Lie}(\Hcal)$. It is clearly a linear and $\Scal$-smooth map.
\end{convention}

\begin{proposition}[{\cite[Proposition III.1.8]{InfDimLieneeb}}]
Let $\Gcal$ and $\Hcal$ be $\Scal$-Lie groups, and let $\phi:\Gcal\to \Hcal$ be a morphism of $\Scal$-Lie groups. Then $\mathrm{Lie}(\phi):\mathrm{Lie}(\Gcal)\to \mathrm{Lie}(\Hcal)$ is a homomorphism of Lie algebras. Hence, in particular, $\mathrm{Lie}(-)$ is a functor from $\Scal$-Lie groups and $\Scal$-Lie group homomorphisms to $\Scal$-Lie algebras and $\Scal$-Lie algebra homomorphisms.    
\end{proposition}

\begin{definition}[{\cite[Definition III.1.15]{InfDimLieneeb}}]
\label{Def:AdAndad}
Let $\Gcal$ be an $\Scal$-Lie group with identity element $e\in \Gcal$. Then we denote by $c_g\coloneqq L_g\circ R_{g^{-1}}$ the conjugation by $g\in \Gcal$. We then denote by $\mathrm{Ad}_g$ the linear map $\mathrm{Lie}(c_g):\mathrm{Lie}(\Gcal) \to \mathrm{Lie}(\Gcal)$. We note that the map $\Gcal\times \mathrm{Lie}(\Gcal)\ni (g,v)\mapsto \mathrm{Ad}_g(v)\in \mathrm{Lie}(\Gcal)$ is an $\Scal$-smooth map which we denote by 
$\mathrm{Ad}^{\wedge}$. Then we can consider its tangent map $$\mathcal{T}\mathrm{Ad}^\wedge:\mathcal{T}(\Gcal\times \mathrm{Lie}(\Gcal))\cong \mathcal{T}\Gcal\times (\mathrm{Lie}(\Gcal)\times \mathrm{Lie}(\Gcal))\to \mathcal{T}(\mathrm{Lie}(\Gcal))\cong \mathrm{Lie}(\Gcal)\times \mathrm{Lie}(\Gcal).$$ We also denote by $Z:\mathrm{Lie}(\Gcal)\to \mathrm{Lie}(\Gcal)\times \mathrm{Lie}(\Gcal)$ the map $Z(v)=(v,0)$ for $v\in \mathrm{Lie}(\Gcal)$, by $i:\mathrm{Lie}(\Gcal)\to \mathcal T \Gcal$ the natural inclusion, and by $\mathrm{Lie}(\Gcal)\times \mathrm{Lie}(\Gcal)\xrightarrow{\pi_2}\mathrm{Lie}(\Gcal)$ the projection to the second factor. All of these maps are clearly $\Scal$-smooth.
We then define the map $\mathrm{ad}\coloneqq \pi_2\circ \mathcal T \mathrm{Ad}^\wedge \circ (i\times Z): \mathrm{Lie}(\Gcal)\times\mathrm{Lie}(\Gcal)\to \mathrm{Lie}(\Gcal)$. Clearly, the map $\mathrm{ad}$ is $\Scal$-smooth and bilinear.
\end{definition}

\begin{remark}
\label{Rem:LocalFormulaForad}
We keep the notation from the previous definition.
In local coordinates centered at $e$ with value in some $\Scal$-vector space $F$ we can express $\mathrm{ad}\{u,v\}$ by $D\widetilde{\mathrm{Ad}^\wedge}(0)\{v,u\}$ for $u,v\in F$, where $\widetilde{\mathrm{Ad}^\wedge}$ denotes the local coordinate expression of $\mathrm{Ad}^\wedge$, and we note that $\mathrm{Ad}^\wedge$ is an $\Scal$-smooth family of linear maps and we use the notation for the differential introduced in Convention \ref{Conv:DiffOfFamOfLins}.
The switch of the order of $u$ and $v$ comes from the isomorphism $\mathcal T(\Gcal\times \mathrm{Lie}(\Gcal))\cong \mathcal T \Gcal \times (\mathrm{Lie}(\Gcal)\times \mathrm{Lie}(\Gcal))$. In fact, in the local chart we have that $(0,u)\in \{0\}\times F\cong \mathcal{T}_e\Gcal$ and $(v,0)\in \{v\}\times \mathrm{Lie}(\Gcal)\cong \mathcal{T}_v\mathrm{Lie}(\Gcal)$ corresponds to the tangent vector $((0,v),(u,0))\in \{(0,v)\}\times (F\times \mathrm{Lie}(\Gcal))\cong \mathcal T_{(e,v)}(\Gcal\times \mathrm{Lie}(\Gcal))$ on the product and hence $\pi_2(T\widetilde{\mathrm{Ad}^\wedge}((0,v),(u,0)))=D\widetilde{\mathrm{Ad}^\wedge}(0,v)(u,0)=D_1\widetilde{\mathrm{Ad}^\wedge}(0,v)u=D\widetilde{\mathrm{Ad}^\wedge}(0)\{v,u\}$.     
\end{remark}

\begin{theorem}[{\cite[Proposition III.1.16]{InfDimLieneeb}}]
\label{Thm:ad=[-,-]}
Let $\Gcal$ be an $\Scal$-Lie group then $\mathrm{ad}\{u,v\}=[u,v]$ for $u,v\in \mathcal{T}_e\Gcal$.
\end{theorem}

To familiarize ourselves with this concept in the infinite-dimensional case let us take a look at the easiest nontrivial example.
\begin{example}
Let $G$ be a finite-dimensional Lie group with Lie algebra $\mathrm{Lie}(G)=\gfrak$, and let $M$ be a compact manifold. Then the Lie algebra of the zero-tame Lie group $C^\infty(M,G)$ (cf. Example \ref{Ex:C(M,G)}) is $C^\infty(M,\gfrak)$ with the bracket $[\gamma,\sigma](p)=[\gamma(p),\sigma(p)]_{\gfrak}\in \gfrak$ for every $p\in M$.
The proof is as follows.
Let $\mathrm{exp}:\gfrak\to G$ be the exponential map of the Lie group $G$, $e\in G$ the identity element, let $U\subseteq \gfrak$ be an open neighbourhood of $0\in \gfrak$, and let $ V\subseteq G$ be an open neighbourhood of $e\in G$ such that $U\stackrel{\mathrm{exp}}{\cong}V$. Then, the map $\Id\times (\left.\mathrm{exp}\right|_U):M\times U\to M\times V$ is a fiber-preserving diffeomorphism and hence induces a chart on the level of smooth sections, i.e.~it induces a tame diffeomorphism of the open subsets $C^\infty(M,U)\subseteq C^\infty(M,\gfrak)$ and $C^\infty(M,V)\subseteq C^\infty(M,G)$ (for openness see Proposition \ref{Prop:FiberBundleOperators}) and so
$T_EC^\infty(M,G)\cong C^\infty(M,\gfrak)$. Here we denote by $E$ the constant map $E(x)= e$ for all $x\in M$, which is the identity element of $C^\infty(M,G)$.
Now, the conjugation is simply $c:C^\infty(M,G)\times C^\infty(M,G)\to C^\infty(M,G)$ by
$c(f,g)(p)=c_f(g)(p)=f(p)g(p)(f(p))^{-1}$ for $f,g\in C^\infty(M,G)$ and $p\in M$. Because $c_f(g)(p)=c_{f(p)}(g(p))$, and using Remark \ref{Rem:Ev Phi=phi_p Ev} on the evaluation map $\Ev_p:C^\infty(M,\gfrak)\to \gfrak$ with $\Ev_p(f)\coloneqq f(p)$ for $f\in C^\infty(M,\gfrak)$ and $p\in M$, we find that the calculation layed out in Definition \ref{Def:AdAndad}  breaks down to a pointwise calculation in the chart induced by $\left.\mathrm{exp}\right|_U:U\to V$ and hence $[\gamma,\sigma](p)=[\gamma(p),\sigma(p)]_{\gfrak}$ for every $p\in M$.
Note that $(\left.\mathrm{exp}\right|_U)^{-1}:V\to U$ has the identity as its differential at $e\in G$.
\end{example}

\begin{definition}[{\cite[Part I, Definition 4.6.5]{hamilton}}]
\label{Def:PrincBundle}
Let $\Gcal$ be an $\Scal$-Lie group, $\Mcal$ an $\Scal$-manifold and $\Pcal\xrightarrow{p} \Mcal$ be an $\Scal$-fiber bundle with typical fiber $\Gcal$.
We call $\Pcal\xrightarrow{p} \Mcal$ an $\Scal$-principal bundle if there is an $\Scal$-smooth $\Gcal$-action $\alpha:\Pcal\times\Gcal\to \Pcal$ on $\Pcal$ such that 
\begin{enumerate}
    \item the action of $\Gcal$ preserves the fibers of $\Pcal\xrightarrow{p}\Mcal$ and acts simply transitively on them and
    \item \label{item:LocalActionOfGonP} for every $x\in \Mcal$ there is an open neighbourhood $\Ucal\subseteq \Mcal$ of $x$ and a local trivialization $\phi:p^{-1}(\Ucal)\to\Ucal\times \Gcal$ such that $\alpha$ corresponds to the right action $(\Ucal\times\Gcal)\times \Gcal\ni((y,g),h)\mapsto (y,gh)\in \Ucal\times \Gcal$.
\end{enumerate}
\end{definition}

\begin{remark}
\label{Rem:PrincBundleViaTransFunct}
Let $\Gcal$ be an $\Scal$-Lie group, $\Mcal$ an $\Scal$-manifold, and let $\Pcal\xrightarrow{p} \Mcal$ be an $\Scal$-fiber bundle with typical fiber $\Gcal$. Let $((\Ucal_\alpha, \phi_\alpha))_{\alpha\in A}$, for some index set $A$, be a fiber bundle atlas such that $\phi_\alpha\circ \phi_\beta^{-1}(x,g)=(x,f(x)g)$ for $\alpha,\beta\in A$, $x\in \Ucal_\alpha \cap \Ucal_\beta$, $g\in \Gcal$, and $f\in C^\infty(\Ucal_\alpha\cap\Ucal_\beta, \Gcal)$. Then, defining locally the action of $\Gcal$ on $\Pcal$ by 
$(\Ucal_\alpha\times \Gcal)\times \Gcal\ni ((x,g),h)\mapsto (x,gh)\in  \Ucal_\alpha\times \Gcal$ we easily see that it is well defined and that $\Pcal$ is indeed an $\Scal$-principal bundle. 
\end{remark}

The next definition was originally formulated in \cite[Section 38]{ConvSetting}. We follow the language used by Schmeding in \cite[Definition 3.30]{schmeding_2022}
\begin{definition}
\label{Def:RegularLieGr}
Let $\Gcal$ be a locally convex Lie group with Lie algebra $\mathrm{Lie}(\Gcal)$ and identity element $e\in \Gcal$.
We say that such a Lie group is \textit{semiregular} if for every $\eta\in C^\infty([0,1],\mathrm{Lie}(\Gcal))$ the differential equation
\begin{align*}
    \gamma_\eta'(t)=\mathcal{T}R_{\gamma_\eta(t)}(\eta(t))           
\end{align*}
with the identification $\mathrm{Lie}(\Gcal)=\mathcal{T}_e\Gcal$ has a unique solution $\gamma_\eta:[0,1]\to \Gcal$ with $\gamma_\eta(0)=e$. Here, $R$ denotes the right translation.
We call a semiregular Lie group \textit{regular} if the evaluation map 
\begin{align*}
    \Ev_1:C^\infty([0,1],\mathrm{Lie}(\Gcal))&\to \Gcal\\
    \eta&\mapsto \gamma_\eta(1)
\end{align*}
is smooth.
\end{definition}

\begin{remark}
\begin{enumerate}
    \item Regular Lie groups fulfill Lie's second theorem, see \cite[Remark 3.32 (b)]{schmeding_2022}
    \item Usually in literature the definition of (semi)regularity of a Lie group involves the left multiplication instead of the right multiplication. These notions are, however, equivalent, see \cite[Remark 3.32 (c)]{schmeding_2022}
\end{enumerate}
\end{remark}

\begin{example}
\label{Ex:DiffRegLieGr}
It turns out that the diffeomorphism group of a compact manifold, which we will focus on in the next section is a regular Lie group, see \cite[3.36 Example]{schmeding_2022}.
\end{example}

\subsection{The Diffeomorphism Group}

By Theorem \ref{AuteqCharacterizationTheorem} we know that the autoequivalences of an exact Courant algebroid consist of closed $2$-forms and diffeomorphisms fulfilling a compatibility condition. In this section we discuss the diffeomorphism group of a compact manifold and show that it is a tame Fréchet Lie group.

\begin{proposition}[{\cite[Chapter III, Proposition 1.10]{StableMappings}}]
\label{Prop:DiffOpen}
Let $M$ be a compact manifold. Then, the diffeomorphism group $\Diff$ is an open subset of the tame Fréchet manifold $C^\infty(M,M)$ and is, hence, itself a tame Fréchet manifold. \end{proposition}

\begin{remark}
Let $h$ be a Riemannian metric on $M$, let $e^h:TM\to M$ denote its geodesic exponential map, and let $\mathrm{Exp}^h:TM\to M\times M$ denote the map that sends $T_pM\ni v\mapsto (p,e^h(v))\in M\times M$ for every $p\in M$. It is easy to see that $\mathrm{Exp}^h$ has full rank along the zero section since $\left.e^h\right|_{T_pM}$ has full rank at $0\in T_pM$.
So, for each $p\in M$ there is a neighbourhood $V_p$ of $0\in T_pM\subseteq TM$ such that $\left.\mathrm{Exp}^h\right|_{V_p}$ is a diffeomorphism onto its image. We write $V\coloneqq \cup_{p\in M}V_p$. Since $\mathrm{Exp}^h$
is fiber-preserving we see, that the map $\left.\mathrm{Exp}^h\right|_V$ is injective and hence a diffeomorphism onto its image. Now, we define the open neighbourhood $\Vcal\coloneqq \Gamma(V\subseteq TM)$ of $0\in \VFs$ and denote the map induced by $\left.\mathrm{Exp}^h\right|_V:V\to M\times M$ by
$\exh:\Vcal\to C^\infty(M,M)$. Since $\exh$ is induced by a fiber-preserving diffeomorphism $\eh{\Vcal}\subseteq C^\infty(M,M)$ is open and the map $\exh:\Vcal\to \exh(\Vcal)$ is a smooth tame diffeomorphism (cf.~Proposition \ref{Prop:FiberBundleOperators}). Since additionally $\Diff\subseteq C^\infty(M,M)$ is open (cf.~Proposition \ref{Prop:DiffOpen}) there is a neighbourhood $\Ucal\subseteq \Vcal$ of $0\in \VFs$ such that $\exh(\Ucal)\subseteq \Diff$ is open and the map $\exh:\Ucal \to \Diff$ is a smooth tame diffeomorphism onto its image. Note that by Proposition \ref{Prop:CompSmoothTame} the composition $R_f(g)=g\circ f$, $g\in C^\infty(M,M)$ is a smooth tame diffeomorphism from $C^\infty(M,M)$ onto itself for every $f\in \Diff$. And hence $R_f\circ\exh:\Ucal\to R_f(\eh{\Ucal})\subseteq \Diff$ is a smooth tame diffeomorphism.
\end{remark}

\begin{convention}
\label{Conv:ChartsOnDiff}
Using the notation of the previous remark we denote by $(\Ucal,\exh_f)$, the local parametrization of $\Diff$ around $f\in \Diff$ defined by the diffeomorphism $\exh_f\coloneqq R_f\circ\exh:\Ucal\to R_f(\eh{\Ucal})\subseteq \Diff$. Note that its inverse is a chart around $f\in \Diff$. For convenience, we denote $(\Ucal,\exh_{\Id})$ by $(\Ucal,\exh)$.
\end{convention}

\begin{proposition}[{\cite[Part II, Theorem 2.3.5]{hamilton}}]
\label{Proposition:DiffLieGrAndLieAlg} 
Let $M$ be a compact manifold. Then, $\Diff$ is a tame Fréchet Lie group and its Lie algebra is $(\VFs,-[-,-])$, where $[-,-]$ denotes the usual Lie bracket of vector fields.
\end{proposition}
\begin{proof}
By openness in $C^\infty(M,M)$ we already know that $\Diff$ is a smooth tame manifold (cf.~Proposition \ref{Prop:DiffOpen}). Hamilton shows in \cite[Part II, Theorem 2.3.5]{hamilton} that it is indeed a smooth tame Lie group. 

To determine the Lie algebra of $\Diff$ we follow Theorem \ref{Thm:ad=[-,-]}. Let $X,Y\in \VFs$, and $(\Ucal,\exh)$ be a local parametrization of $\Diff$ around $\Id$ as in Convention \ref{Conv:ChartsOnDiff}. As $(\Ucal,\exh)$ is a local parametrization around $\Id$ it is clear that $\mathcal{T}_{\Id}(\Diff)\cong \VFs$. We will prove later that using this identification we find 
$\mathrm{Ad}_f(Y)=D((\exh)^{-1}\circ c_f\circ \exh)(0)Y=f_*Y$ (cf.~Lemma \ref{Lemma:PullbackVfields}) and that $D_Z((\eh{Z})_*Y)(0)X=-\LieD_XY=-[X,Y]$, $Z\in \Ucal$ (cf.~Remark \ref{Rem:PushForwDifferential}).  
\end{proof}

\subsection{Families of smooth maps}

In this section we compare two notions of families of smooth maps. We prove an exponential law for such families. In particular we show for a surjective submersion $B\to M$ with compact codomain $M$ that the evaluation map $\Gamma(B)\times M\to B$ is smooth tame. This will allow us to compute derivatives of maps with values in spaces of the type $\Gamma(B)$ by considering the derivative in a fixed fiber.

\begin{proposition}
\label{Prop:EvVBSmooth}
Let $V\to M$ be a smooth vector bundle over a compact manifold $M$. Then the map $\mathrm{Ev}: \Gamma(V)\times M\to V$ mapping $(s,p)\mapsto s(p)$ for $p\in M$ and $s\in \Gamma(V)$ is smooth tame. In particular, for every $k\in \N$ we can find tame estimates for $\mathcal T^k\mathrm{Ev}$ of base $0$ and degree $k$ around every point in $\mathcal{T}^k(\Gamma(V)\times M)$.
\end{proposition}
\begin{proof}
We note that since $V$ is a finite-dimensional manifold or more generally a manifold modeled on Banach spaces we automatically obtain by Corollary \ref{Cor:SmToBanSmTame} that if $\Ev$ is smooth, then it is also smooth tame.
We note that the map $\mathrm{Ev}(s,p)=s(p)$ with 
$p\in M$ and $s\in \Gamma(V)$ is linear in $s$. So we already know that 
$D_1\Ev(s,p)h=\Ev(h,p)=h(p)$ for $h\in\Gamma(V)$ and $D_1^2\Ev(s,p)=0$. Similarly for a fixed $s\in \Gamma(V)$ the differential $D_2\Ev(s,p)v=ds_p(v)$, $v\in T_pM$, is the usual derivative in finite dimension. It easily follows that if it exists, then $D^k\Ev$ is of the form
\begin{align*}
    D^k\Ev(s,p)\{(h_1,v_1),\ldots,(h_k,v_k)\}=&D^k_2 \Ev(s,p)\{v_1,\ldots,v_k\}\\
    &+\sum_{i=1}^kD^{k-1}_2 \Ev(h_i,p)\{v_1,\ldots,v_{i-1},v_{i+1},\ldots,v_k\}
\end{align*}
for $p\in M$, $s,h_1,\ldots,h_k\in \Gamma(V)$ and $v_1,\ldots,v_k\in T_pM$.
So, by Theorem \ref{Thm:PartialDer} it suffices to show that the map $D_2^k\Ev(s,p)\{v_1,\ldots,v_k\}$ is continuous and tame for every $k\in \N$.
To see this let $\Tilde{U}\cong \R^n$ be a chart domain around $p\in M$ and $v_1,\ldots,v_k\in \R^n$ with $n=\mathrm{dim}(M)$. We additionally assume that there is also a trivialization of $V$ over $\Tilde{U}$. We now take $U\subset \Tilde{U}$ to be an open subset such that $\Bar{U}\subset \Tilde{U}$. Then in the chart and trivialization over $\Tilde{U}$ using the Einstein summation convention we can express 
\begin{align}
\label{eq:D_1^k(Ev)}
    D_2^k\widetilde{\Ev}(s,x)\{v_1,\ldots,v_k\}=(D_{i_1}\ldots D_{i_k} s)(x)v_1^{i_1}\ldots v_n^{i_k}
\end{align}
for $x\in \R^n\cong \Tilde{U}$.
After possibly estimating some factors, stemming from transition functions and coordinate transformations, on the relatively compact set $U$ we find
that there is a constant $C>0$ such that for all $x\in U$ and $v_1,\ldots,v_k \in \R^n$ 
\begin{align}
\label{eq:EstForD^k_1Ev}
    \abs{D_2^k\widetilde{\Ev}(s,x)\{v_1,\ldots,v_k\}}\leq  C\nnorm{s}{k}'\abs{v_1}\ldots\abs{v_k}
\end{align}
is satisfied.
If we additionally restrict $\abs{v_i}\leq c$, $i=1,\ldots,k$, for $c>0$ some constant then we surely find that there is a constant $C>0$ such that for all $x\in U$
\begin{align*}
    \abs{D_2^k\widetilde{\Ev}(s,x)\{v_1,\ldots,v_k\}}\leq C\left(1+\nnorm{s}{k}'+\sum_{i=1}^k\abs{v_i}\right).
\end{align*}
This is a tame estimate of degree $k$ in $s$, degree $0$ in $v_1,\ldots,v_k$ and base $0$. 
By definition $D_2^k\widetilde{\Ev}(s,x)\{v_1,\ldots,v_k\}=(D_{i_1}\ldots D_{i_k} s)(x)v_1^{i_1}\ldots v_n^{i_k}$ is continuous in $x\in U$. Moreover, by equation \eqref{eq:EstForD^k_1Ev} we find that the other requirements of Lemma \ref{Lemma:ContFamOfMultLin} are satisfied and so $D^k_2\Ev$ is in fact continuous.  
\end{proof}

\begin{remark}
\label{Rem:DifferentialEvAtPoint}
Let $M$ be a compact manifold, $V\to M$ be a vector bundle, $U\subseteq V$ be open, $W\to M$ be a vector bundle, and $P:\Gamma(U\subseteq V)\to \Gamma(W)$ a nonlinear vector bundle operator induced by a fiber-preserving smooth map $p:U\to V$. Then, by continuity of $\Ev$ the equation
\begin{align*}
    \Ev(DP(\sigma)\gamma,x)&=\Ev\left(\lim_{t\to 0}\frac{P(\sigma+t\gamma)-P(\sigma)}{t},x\right)\\
    &=\lim_{t\to 0}\Ev\left(\frac{P(\sigma+t\gamma)-P(\sigma)}{t},x\right)=\d_vp(\sigma(x))(\gamma(x))
\end{align*}
holds for every $x\in M$, $\sigma\in \Gamma(U\subseteq V)$ and $\gamma\in \Gamma(V)$ since the map $\Ev(-,x)$ is a continuous linear map. Here, by $\d_v$ we mean the differential along the fibers of $V$.
\end{remark}

\begin{corollary}
\label{Cor:EvFBSmooth}
Let $M$ be a compact manifold and let $B\to M$ be a surjective submersion. Then the evaluation map $\mathrm{Ev}:\Gamma(B)\times M\to B$ mapping $(\sigma,p)\to \sigma(p)$, for $\sigma\in \Gamma(B)$ and $p\in M$, is smooth tame. In particular $\mathcal T^k\mathrm{Ev}$ has local tame estimates of base $0$ and degree $k$ around each point in $\mathcal T^k(\Gamma(B)\times M)$.
\end{corollary}
\begin{proof}
This is a direct consequence of Proposition \ref{Prop:EvVBSmooth}. We choose some $\sigma\in\Gamma(B)$. By Theorem \ref{FiberBundleSection} we know that there is an open neighbourhood $U\subseteq B$ of $\sigma$ and a fiberwise diffeomorphism $\phi:U\to W$ to some open neighbourhood $W$ of the zero section in some smooth vector bundle $V\to M$. Then we denote the induced map of sections by $\Phi$. Then the map $\phi\circ \mathrm{Ev}\circ \Phi^{-1}:\Gamma(W\subseteq V)\to W$ is smooth tame by Proposition \ref{Prop:EvVBSmooth}. The degrees of tameness also directly follow from that proposition. We note that $\phi$ is a zero-tame map as it is a smooth map of smooth finite-dimensional manifolds (cf.~Proposition \ref{Prop:FinDimMfdsAsFrechetMfds}). 
\end{proof}

\begin{remark}
\label{Rem:Ev Phi=phi_p Ev}
Let $V\to M$ be a vector bundle, $M$ a compact manifold, and 
$f:(a,b)\to \Gamma(V)$, $a,b\in \R$ be a $C^1$ curve. Then by continuity and linearity of
$\Ev_p=\Ev(-,p):V\to V_p$, $p\in M$, we find $(\Ev_p(f))'=\Ev_p(f')$. Let $B\to M$ be a surjective submersion, $\sigma\in \Gamma(B)$, $U\subseteq B$ an open neighbourhood of $\sigma$, $W\subseteq V$ be an open subset of a vector bundle $V\to M$ and $\phi:U\to V$ a fiber-preserving diffeomorphism inducing a chart $\Phi:\Gamma(U\subseteq B)\to \Gamma(W\subseteq V)$ of $\Gamma(B)$ then $\Ev_p\circ \Phi=\left.\phi\right|_{B_p}\circ \Ev_p$ for $p\in M$, where we understand $\Gamma(B)\xrightarrow{\Ev_p}B_p\xrightarrow{\left.\phi\right|_{B_p}}V_p$. Hence, if $f:(a,b)\to \Gamma(B)$, $a,b\in \R$, is a $C^1$ curve then we have
$$\Ev_p((\Phi\circ f)')(0)=(\Ev_p\circ \Phi \circ f)'(0)=(\left.\phi\right|_{B_p}\circ \Ev_p\circ f)'(0)=\d(\left.\phi\right|_{B_p})_{f(0)(p)}((\Ev_p\circ f)'(0)).$$ 
\end{remark}

\begin{theorem}
\label{Theorem:FamiliesOfFBSections}
Let $N$ and $M$ be smooth manifolds, and let $M$ be compact. Let $B\to M$ be a surjective submersion. Let $f:N\times M\to B$ be a smooth map such that $f(n,-)\in \Gamma(B)$ for each $n\in N$. Then the induced map $F:N\to \Gamma(B)$ with $F(n)=f(n,-)$, $n\in N$, is zero-tame.
\end{theorem}
\begin{proof}
The condition $f(n,-)\in \Gamma(B)$ is equivalent to the condition that $f$ is a smooth fiber-preserving map from the trivial fiber bundle $M\times N\to M$ to the surjective submersion $B\to M$. So, it induces a bundle operator $\Tilde{F}:C^\infty(M,N)\to \Gamma(B)$ which is hence zero-tame (cf.~Proposition \ref{Prop:FiberBundleOperators}).

Let $\phi:U\to \R^d$, $d=\dim(N)$, be a chart for $U\subseteq N$ open. Then $C^\infty(M,U)\subseteq C^\infty(M,N)$ is open (cf.~Proposition \ref{Prop:FiberBundleOperators}). Let $\underline{\R}^d$ denote the trivial vector bundle $M\times \R^d \to M$ then as $\phi$ induces a bundle operator we get a zero-tame diffeomorphism $\Phi:C^\infty(M,U)\cong \Gamma(\underline{\R}^d)$, which is a local chart on $C^\infty(M,N)$. 
We can now look at the subspace of of constant sections of $\underline{\R}^d$ and denote it by $\Gamma_c(\underline{\R}^d) \subset \Gamma(\underline{\R}^d)$. This is a finite-dimensional subspace continuously isomorphic to $\R^d$ by the obvious bijection that maps a constant section with value $c\in \R^d$ to $c$ (cf.~Lemma \ref{Lemma:FinDimHausdTVS}). Moreover, since $\Gamma_c(\underline{\R}^d)\subseteq \Gamma(\underline{\R}^d)$ as a complete subspace of a metrizable space is a closed subspace it is a graded Fréchet space and for $s_c\in \Gamma_c(\underline{\R}^d)$ with constant value $c\in \R^d$ we have $\nnorm{s_c}{l}=\nnorm{s_c}{0}=|c|$, $l\in \N$ so the previously mentioned isomorphism $\Gamma_c(\underline{\R}^d)\to \R^d$ is zero-tame. Here, $\nnorm{-}{l}$, $l\in \N$, is the norm on $\Gamma(\underline{\R}^d)$ defined as in equation \eqref{VBsectGlobNorm} with respect to the trivial covariant derivative, and Euclidean norm. Note that in general not every (continuous) linear map of (finite-dimensional) graded Fréchet spaces is also zero-tame. So, noting that $\phi$ as a smooth map of smooth finite-dimensional manifolds (with the canonical zero-tame structure induced by the Euclidean norm) is zero-tame (cf.~Proposition \ref{Prop:FinDimMfdsAsFrechetMfds}), it is clear that the injective map $$U\stackrel{\phi}{\cong} \R^d \cong \Gamma_c(\underline{\R}^d)\hookrightarrow \Gamma(\underline{\R}^d)\stackrel{\Phi^{-1}}{\cong} C^\infty(M,U)\hookrightarrow C^\infty(M,N)$$ is zero-tame and is precisely the map that maps some $n\in N$ to the constant map $M\to N$ with value $n$. Hence, we find that the induced map $\left.F\right|_U:U\hookrightarrow C^\infty(M,N)\xrightarrow{\Tilde{F}} \Gamma(B)$ is zero-tame and thus $F:N\to \Gamma(B)$ is a zero-tame map. 
\end{proof}

\begin{corollary}
\label{FamiliesOfDiffeos}
Let $M$, $N$ and $Y$ be smooth manifolds, let $M$ be compact, and let $f:N\times M\to Y$ smooth. Then the induced map $N\to C^\infty(M,Y)$ is zero-tame. Moreover, if $Y=M$ and $f$ is such that $f(n,-)\in \Diff$, for all $ n\in N$, then the induced map $N\to \Diff$ is zero-tame.
\end{corollary}
\begin{proof}
We first observe that we can extend a smooth map $f:N\times M\to Y$ to 
$\Hat{f}:N\times M\to M\times Y$ by sending $\Hat{f}(n,m)=(m,f(n,m))$. By Theorem \ref{Theorem:FamiliesOfFBSections} we know that the induced map $N\to \Gamma(M\times Y\to M)\cong C^\infty(M,Y)$ is smooth tame. Since 
$\Diff \subset C^\infty(M,M)$ is open (cf.~Proposition \ref{Prop:DiffOpen}) the result for families of diffeomorphisms follows in the same way. 
\end{proof}

We now summarize the main results of this section.
\begin{corollary}
Let $M$, $N$ and $Y$ be finite-dimensional smooth manifolds, and let $M$ be compact. Let $B\to M$ be a surjective submersion. Then, we have found the following bijective correspondences:
\begin{enumerate}
    \item $$\{f\in C^\infty(N\times M,B)|~f(n,-)\in \Gamma(B)\ \forall n\in N\}\overset{1:1}{\longleftrightarrow} \{f:N\to \Gamma(B) \ \tn{zero-tame}\}$$
    \item $$C^\infty(N\times M,Y)\overset{1:1}{\longleftrightarrow} \{N\to C^\infty(M,Y) \ \tn{zero-tame}\}$$
    \item $$\{f\in C^\infty(N\times M,M)|~f(n,-)\in \Diff\ \forall n\in N\}\overset{1:1}{\longleftrightarrow} \{N\to \Diff \ \tn{zero-tame}\}.$$
\end{enumerate}
\end{corollary}
\begin{proof}
We will prove this only for families of sections of $B$ since the other two cases follow as direct consequences.

The correspondence from left to right immediately follows from Theorem \ref{Theorem:FamiliesOfFBSections}.
The other direction follows since $\Ev:M\times \Gamma(B)\to B$ is smooth. So, if $f:N\to \Gamma(B)$ is smooth, then $\Ev\circ (f\times \Id):N\times M\to B$ is a smooth map of smooth finite-dimensional manifolds by Proposition \ref{Prop:FinDimMfdsAsFrechetMfds}.  
\end{proof}

\begin{remark}
We note that there is a broader theory of so-called canonical manifolds which is discussed in \cite[2.3 Manifolds of Mappings]{schmeding_2022}.
\end{remark}

\begin{remark}
\label{Rem:C^nFamilies}
Note that more generally for any $n\in \N$ all $C^n$-maps $N\to \Gamma(B)$, $N\to C^\infty(M,Y)$, and $N\to \Diff$ results again in $C^n$-maps $N\times M\to B$, $N\times M\to Y$, and $N\times M\to M$, respectively, by smoothness of the evaluation map $\Ev$ and Proposition \ref{Prop:FinDimMfdsAsFrechetMfds}.
\end{remark}

\begin{remark}
We note that every smooth map $N\to \Gamma(B)$, $N\to C^\infty(M,Y)$, or $N\to \Diff$ is zero-tame. 
To see this, let $F:N\to \Gamma(B)$ be smooth. Then we know that $f\coloneqq\Ev\circ (F\times \Id):N\times M\to B$ is a smooth map of finite-dimensional manifolds in the conventional sense (cf.~Proposition \ref{Prop:FinDimMfdsAsFrechetMfds}). Since $F:N\to \Gamma(B)$ is the map induced by the smooth map $f:N\times M\to B$ it is zero-tame.
\end{remark}

\begin{remark}
\label{Remark:ConvCalcForMapOfSect}
Let $M$ be a compact manifold, let $A\to M$ and $B\to M$ be surjective submersions, and let $\Ucal\subseteq \Gamma(A)$ be open. As laid out in Remark \ref{Rem:ConvenientCalculus} we find that a map $F:\Ucal\to \Gamma(B)$ is smooth if and only if $F\circ c$ is smooth for every smooth curve $c\in C^\infty(\R,\Ucal)$. Hence, to show smoothness of $F$ it suffices to check that the map $F\circ c$ can be induced by a smooth map $f_c:\R\times M\to B$ with $f(t,p)\in B_p$ for all $(t,p)\in \R \times M$ for every $c\in C^\infty(\R,\Gamma(B))$. However, we note that we are not aware of any way to infer smooth tameness of $F$ by some similar method. 
\end{remark}

Using these results we can for instance prove the following.
\begin{corollary}
\label{Cor:FlowSmooth}
    Let $M$ be a compact manifold. Then the map  $\Phi:\VFs\times \R\to \Diff$ mapping a vector field $X$ and parameter $t$ to the flow of the vector field at time $t$ is smooth.
\end{corollary}
\begin{proof}
    Since $M$ is compact the flow is indeed globally defined.
    We check that the time-one flow $\Phi_1$ as a map $\VFs\to \Diff$ is smooth. The result then follows from $\Phi_t(X)=\Phi_1(tX)$. Now let $Y\in C^\infty(\R,\VFs)$. Then, let $\phi^Y_1:\R\times M\to M$ be the map $(x,\lambda)\mapsto (\Phi_1(Y(\lambda)))(x)$. Note that $\lambda$ is treated as a parameter independent of time and the arguments still apply if we replaced the parameter space $\R$ by an arbitrary (finite-dimensional) manifold. Then by standard ODE theory $\phi^Y_1$ is smooth. But then $\Phi_1\circ Y$ is precisely the map induced by $\phi^Y_1$ and so $\Phi_1$ is smooth by the previous remark. 
\end{proof}

\subsection{The Pullback of tensor fields along diffeomorphisms as a smooth tame map}

As we have seen in Theorem \ref{AuteqCharacterizationTheorem} the characterization of the group of autoequivalences of a Courant algebroid involves the pullback of differential forms along diffeomorphisms. We will show later in Section \ref{SubSect:LieGrStrOfAuteq} the charts of the group of autoequivalences will also involve the pullback. 

In this section we show that on a compact manifold the pullback of tensor fields along diffeomorphisms is a smooth tame map. We deduce that the pushforward of tensor fields along diffeomorphisms must be smooth tame and compute the differential of both the pullback and pushforward.

\begin{convention}
\label{Conv:tensfields}
As there are different conventions in literature, we want to clarify that in the following we will use the notation $\tensf{n}{m}\coloneqq \Gamma((TM)^{\tens n}\tens(T^*M)^{\tens m})$, $n,m\in \N$, to denote the smooth $(n,m)$-tensor fields on a (compact) manifold $M$.
\end{convention}

\begin{theorem}
\label{PullbackSmoothTame}
Let $M$ be a compact smooth finite-dimensional manifold.
The pullback of tensor fields along diffeomorphisms as a map
\begin{align*}
    P:\mathrm{Diff}(M)\times \tensf{n}{m}&\longrightarrow \tensf{n}{m}\\
    (f,\tau)&\longmapsto P(f,\tau)\coloneqq f^*\tau
\end{align*}
of tame Fréchet manifolds is a smooth tame map for every $n,m\in \N$. 
\end{theorem}

We will prove this theorem in several steps. Firstly, one might be tempted to think that $P$ is simply a nonlinear partial differential operator (cf.~Definition \ref{Def:NonlinPartDiffOp}). However, take for instance $\om\in \tensf{0}{1}=\Forms{1}$ and $f\in \Diff$ by definition we have
\begin{align*}
    (f^*\om)_p(X)=\om_{f(p)}(\d f_p(X))
\end{align*}
for $p\in M$, and $X\in \VFs$.
This expression depends on $f$ and its first derivative in the point $p$. However, it depends on $\om$ at the point $f(p)$ and thus cannot be interpreted as a nonlinear differential operator jointly in $(f,\om)$.

\begin{lemma}
\label{PullbackFixed_f}
Let $f\in \Diff$ then the map
\begin{align*}
    P_f:\tensf{0}{m}&\longrightarrow \tensf{0}{m}\\
    \tau &\longmapsto P_f(\tau)\coloneqq f^*\tau
\end{align*}
is a tame linear map for any $m\in \N$.
\end{lemma}
\begin{proof}
Let there be some coordinates $(U,\varphi=(x^1,\ldots,x^d))$ on an open subset of $U\subseteq M$ with $d=\dim(M)$. Then, $(f(U),\varphi \circ f^{-1})$ is a chart on $M$. Let $$\left.\tau\right|_{f(U)}=\tau_{i_1,\ldots,i_m}\d(x^{i_1}\circ f^{-1})\tens \ldots\tens \d(x^{i_m}\circ f^{-1})$$
using the Einstein summation convention where $\tau_{i_1,\ldots,i_m}\in C^\infty(f(U))$, and the indices have range in $i_1,\ldots,i_m\in \{1,\ldots,d\}$.  Then for $p\in U$ we have $$(f^*\tau)_p=\tau_{i_1,\ldots,i_m}(f(p))\d(x^{i_1}\circ f^{-1}\circ f)_p\tens \ldots\tens \d(x^{i_m}\circ f^{-1}\circ f)_p=\tau_{i_1,\ldots,i_m}(f(p))\d x^{i_1}_p\tens\ldots\tens \d x^{i_m}_p.$$
If $(U_i,\varphi_i)_{i=1,\ldots,k}$, $k\in \N$, is an open cover on $M$ by charts, such that each $(U_i,\varphi_i)$ can be extended to a larger chart $(V_i,\Tilde{\varphi_i})$ such that $\Bar{U}_i\subset V_i$, then the same holds true for $(f(U_i),\varphi\circ f^{-1})_{i=1,\ldots,k}$. Let the former induce the norms $(\nnorm{-}{i}')_{i\in \N}$ and the latter induce the norms $(\nnorm{-}{f,j}')_{j\in \N}$ as described in equation \eqref{VBsectLocNorm}. Then, it is clear that after potentially locally estimating some partial derivatives of $f$ (which come up due to the chain rule) over the relatively compact chart domains there is a constant $C_i>0$ for every $i\in \N$ such that 
\begin{align*}
    \nnorm{f^*\tau}{i}'\leq C_i\nnorm{\tau}{f,i}'.
\end{align*}
So, since $\nnorm{-}{k}'$ does not depend on the choices of charts or partition of unity up to zero-tame equivalence of gradings this shows tame linearity.
\end{proof}

\begin{lemma}
\label{PullbackFixed_w}
Let $\tau\in \tensf{0}{m}$. Then the map 
\begin{align*}
    P_\tau:\Diff&\longrightarrow \tensf{0}{m}\\
    f&\longmapsto P_\tau(f)\coloneqq f^*\tau
\end{align*}
is a nonlinear partial differential operator of degree $1$ and hence smooth tame. 
\end{lemma}
\begin{proof}
Let $(\Ucal,\exh)$ be as in Convention \ref{Conv:ChartsOnDiff}. We note that since $M$ is compact it is geodesically complete and the map $\exh:\VFs \to C^\infty(M,M)$ makes sense.
Let $X$ be a vector field and $j^1X$ its $1$-jet. Then it is clear that the map that sends 
$(j^1X)_p\mapsto (\eh{X}^*\tau)_p=(\d\eh{X}_p)^*(\tau_{\eh{X}(p)})$ defines a smooth fiber-preserving map from the first jet bundle $J^1(TM)$ to the tensor bundle $(T^*M)^{\tens m}$. One can see this easily by writing the expression down in local coordinates and applying the chain rule to $\d\eh{X}_p$. Thus, it induces a nonlinear vector bundle operator $F:\Gamma(J^1(TM))\to \mathcal{T}^0_m$, see Definition \ref{Def:VBop}. Hence, the map $(F\circ \left.j^1\right|_\Ucal):\Ucal\to \tensf{0}{m}$ is a first-order nonlinear partial differential operator (cf.~Definition \ref{Def:NonlinPartDiffOp}) and thus smooth tame (cf.~Proposition \ref{Prop:NonlinPartDiffOpsAreSmoothTame}).
So we see that $P_\tau$ is smooth tame around $\Id$. For an arbitrary $g\in \Diff$ we observe that $P_\tau(f\circ g)=P_g(P_\tau(f))$ for all $f\in \Diff$ where $P_g$ is defined as in the previous lemma. Since $P_g$ is a tame linear map, and $\Diff \ni f\mapsto f\circ g\in \Diff$ is a smooth tame diffeomorphism we see that $P_\tau$ is smooth tame on a neighbourhood of $g$. Since $g\in \Diff$ was arbitrary $P_\tau:\Diff\to \tensf{0}{m}$ is smooth tame.
\end{proof}

\begin{lemma}
\label{IntProdSmTame}
Any map of sections of vector bundles induced by a fiberwise contraction, e.g. tensor contractions, is zero-tame. In particular, the musical isomorphisms associated with a (semi-)Riemannian metric or an (almost) symplectic form are zero-tame.
\end{lemma}
\begin{proof}
We can interpret these operations as nonlinear vector bundle operators (cf.~Definition \ref{Def:VBop}) by Proposition \ref{Prop:EplusFsect}.
\end{proof}

We will now show that the pullback of vector fields and $1$-forms is smooth tame.

\begin{lemma}
\label{Lemma:FlowInitialVelocity}
Let $(\Ucal,\exh)$ be the standard local parametrization of $\Diff$ around $\Id$ (cf.~Convention \ref{Conv:ChartsOnDiff}). Let $X\in \VFs$ then the flow $\Phi^X_t$ of $X$ understood as a path in $\Diff$ is a smooth curve, and it has initial velocity $((\exh)^{-1}\circ \Phi^X)'(t=0)=X$.
\end{lemma}
\begin{proof}
By Corollary \ref{Cor:FlowSmooth} we know that the flow $\Phi^X:\R\to \Diff$ of a vector field $X$ is a smooth one-parameter group of diffeomorphisms. Note that $X$ is complete as $M$ is compact.

Let $(\Ucal,\exh)$ be a local parametrization of $\Diff$ around $\Id$ as in Convention \ref{Conv:ChartsOnDiff}. Then we express the flow of $X\in \VFs$ in this parametrization by $(\exh)^{-1}\circ\Phi^X_t$ which is well defined at least for small $t$. As layed out in Remark \ref{Rem:Ev Phi=phi_p Ev} we find
\begin{align*}
    \left(\left.\frac{\d}{\d t}\right|_{t=0}(\exh)^{-1}\circ\Phi^X_t\right)_p&=(\d(\left.\mathrm{Exp}\right|_{T_pM})^{-1})_p\left(\left.\frac{\d}{\d t}\right|_{t=0}(\Phi_t^X(p))\right)\\
    &=(\d(\left.\mathrm{Exp}\right|_{T_pM})_0)^{-1}(X_p)=X_p.
\end{align*}
Here we have used the usual identification of the diffeomorphism $\Phi^X_t$ with the section $m\in M\mapsto (m,\Phi_t^X(m))\in M\times M$ of the projection to the first factor $M\times M\xrightarrow{\pi_1}M$.    
\end{proof}

\begin{lemma}
\label{Lemma:PullbackVfields}
The map 
\begin{align*}
    P^1:\Diff \times \VFs& \longrightarrow \VFs\\
    (f,X)&\longmapsto f^*X
\end{align*}
is smooth tame.
\end{lemma}
\begin{proof}
Let $c:\Diff\times \Diff\to \Diff$ denote the conjugation map $c_f(g)=f\circ g\circ f^{-1}$ with $f,g\in \Diff$. It is smooth tame since $\Diff$ is a tame Fréchet Lie group (cf. Proposition \ref{Proposition:DiffLieGrAndLieAlg}).
Let $X\in \VFs$ and $(\Ucal,\exh)$ be a local parametrization of $\Diff$ around $\Id$ as in Convention \ref{Conv:ChartsOnDiff}. Then, by Lemma \ref{Lemma:FlowInitialVelocity} we see that
\begin{align*}
    (D((\exh)^{-1}\circ c_f\circ \exh)(0)X)_p&=((\exh)^{-1}\circ c_f\circ \Phi_t^X)'(t=0)(p)\\
    &=((\exh)^{-1}(f\circ \Phi_t^X\circ f^{-1}))'(t=0)(p)\\
    &=(\d(\left.\mathrm{Exp}\right|_{T_pM})_0)^{-1}((f\circ \Phi_t^X\circ f^{-1})_p)'(t=0)\\
    &=(\d f\circ X\circ f^{-1})_p=(f_*X)_p.
\end{align*}
for $f\in \Diff$.
Then $P^1(f,X)=D((\exh)^{-1}\circ c_{V(f)}\circ \exh)(0)X$, and so it is smooth tame. Here, $V$ denotes the smooth tame inversion map of $\Diff$.
\end{proof}

\begin{lemma}
\label{PullbackOneForms}
The map
\begin{align*}
    P_1:\Diff\times \Forms{1}&\longrightarrow \Forms{1}\\
    (f,\omega)&\longmapsto f^*\omega
\end{align*}
is smooth tame.
\end{lemma}
\begin{proof}
We can use a Riemannian metric $g$ on $M$ which induces tame linear isomorphisms $g:\VFs\to \Forms{1}$, and $g^{-1}:\Forms{1}\to \VFs$ (cf.~Lemma \ref{IntProdSmTame}). Similarly, we can look at the induced map $G:\Diff\times\VFs\to \Forms{1}$ defined by $G(f,X)=(f^*g)(X,-)$ which is a smooth tame family of isomorphisms by Lemma \ref{PullbackFixed_w} and Lemma \ref{IntProdSmTame}.

Then we find that $\Diff\times \Forms{1}\ni(f,\omega)\mapsto G(f,P^1(f,g^{-1}\om))\in \Forms{1}$ is smooth tame (cf.~Lemma \ref{Lemma:PullbackVfields}). We claim that the resulting differential form is precisely $f^*\omega$. To show this, let us evaluate it on some vector field $X\in \VFs$. We find that
\begin{align*}
    G(f,P^1(f,g^{-1}\omega))(X)(p)&=(f^*g)_p((\d f_p)^{-1}((g^{-1}w)_{f(p)}),X_p)\\
    &=g_{f(p)}(g^{-1}_{f(p)}\omega_{f(p)},\d f_p(X_p))\\
    &=\omega_{f(p)}(\d f_p(X_p))\\
    &=(f^*\omega)(X)(p)
\end{align*}
for $p\in M$.
This concludes the proof.
\end{proof}

The goal is now to use the result for the pullback of vector fields and $1$-forms to show that the pullback of $(n,m)$-tensor fields for any $n,m\in \N$ is smooth tame.

\begin{proof}[Proof of Theorem \ref{PullbackSmoothTame}]
Note that for $n=m=0$ we only have the pullback of functions along diffeomorphisms, which is smooth tame as the composition of smooth maps is smooth tame (cf.~Proposition \ref{Prop:CompSmoothTame}) and we have also already covered the cases of $(1,0)$- and $(0,1)$-tensor fields.

In the following, we will assume $n\geq 1$ or $m\geq 1$. 
Let $p\in M$ then we can find an open chart domain $V_p\subseteq M$ around $p$ such that $\Bar{V}_p$ lies in a bigger chart.
We may also assume that there is an even smaller open neighbourhood $U_p\subseteq V_p\subseteq M$ of $p$ and a function $b_p\in C^\infty(M)$, so that $b_p(x)=1$ for $x\in U_p$ and $b_p(x)=0$ for $x\in M\setminus V_p$. Now, since $M$ is compact there are finitely many $p_i\in M$, for $i$ in some finite index set $I$, such that $(U_{p_i})_{i\in I}$ provides an open cover of $M$. We can also find a partition of unity $(\rho_i)_{i\in I}$ subordinate to $(U_{p_i})_{i\in I}$ such that $\rho_i$ has support in $U_{p_i}$. We now denote $U_i\coloneqq U_{p_i}$, $V_i\coloneqq V_{p_i}$ and $b_i\coloneqq b_{p_i}$ for $i\in I$.

We now consider the following series of maps. First we map $\tau \in \tensf{n}{m}$ to $b_i^{(n+m)} \tau$. Since this has support in $V_i$ which is properly contained in a chart domain we may write $b_i^{(n+m)}\tau$ as
\begin{align*}
    \sum_{\substack{\alpha_1,\ldots,\alpha_n\\\beta_1,\ldots,\beta_m}}\tau^{\alpha_1,\ldots,\alpha_n}_{i,\beta_1,\ldots,\beta_m}(b_i\del_{x_{\alpha_1}})\tens\ldots\tens (b_i\del_{x_{\alpha_n}})\tens (b_i\d x_i^{\beta_1}) \tens \ldots \tens (b_i\d x_i^{\beta_m})
\end{align*}
where $(x_i^1,\ldots,x_i^d)$, $d=\dim(M)$, denote the coordinates on this larger chart domain, and $\tau^{\alpha_1,\ldots,\alpha_n}_{i,\beta_1,\ldots,\beta_m}$ denote the associated smooth component functions of $\tau$ with $\alpha_i=1,\ldots,d$, $i=1,\ldots,n$, and $\beta_j=1,\ldots,d$, $j=1,\ldots,m$.
For each component function, with the indices as above, we can induce a well defined map
\begin{align*}
    \pi^{(\alpha_1,\ldots\alpha_n)}_{i,(\beta_1,\ldots,\beta_m)}:\tensf{n}{m}&\longrightarrow \VFs^n\times \Forms{1}^m \\
    \tau &\longmapsto(\tau^{\alpha_1,\ldots,\alpha_n}_{i,\beta_1,\ldots,\beta_m}\rho_ib_i \del_{x_i^{\alpha_1}},b_i\del_{x_i^{\alpha_2}},\ldots,b_i\del_{x_i^{\alpha_2}},b_i \d x_i^{\beta_1},\ldots,b_i \d x_i^{\beta_m}).
\end{align*}
Here we do not sum over indices that appear multiple times. Note that since $b_i$ are bump functions these are actually global vector fields and $1$-forms with support in $\Bar{V}_i$. In the case that $n=0$ the component functions are simply multiplied with the first covariant factor.
These maps are clearly well defined. They are smooth zero-tame as they are nonlinear vector bundle operators (cf.~Definition \ref{Def:VBop}, and Lemma \ref{VBopSmoothTame}) where we use the zero-tame isomorphism $\Gamma((TM)^{\oplus n}\oplus (T^*M)^{\oplus m})\cong \VFs^n\times \Forms{1}^m$ (cf.~Proposition \ref{Prop:EplusFsect}).

Now we have the map 
\begin{align*}
    P_m^n:\Diff\times \VFs^n\times \Forms{1}^m &\longrightarrow \VFs^n\times \Forms{1}^m\\
    (f,X_1,\ldots,X_n,\om_1,\ldots,\om_m)&\longmapsto (f^*X_1,\ldots,f^*X_n,f^*\om_1,..,f^*\om_m)
\end{align*}
which is smooth tame by Lemma \ref{Lemma:PullbackVfields} and Lemma \ref{PullbackOneForms}. 

Now we define the map
\begin{align*}
    \tens_m^n:\VFs^n\times \Forms{1}^m&\longrightarrow \tensf{n}{m}\\
    (X_1,\ldots,X_n,\om_1,\ldots,\om_m)&\longmapsto X_1\tens\ldots\tens X_n \tens \om_1\tens\ldots\tens \om_m.
\end{align*}
We again use the zero-tame isomorphism
$\Gamma((TM)^{\oplus n}\oplus (T^*M)^{\oplus m})\cong \VFs^n\times \Forms{1}^m$. Under this identification $\tens^n_m$ is simply a nonlinear vector bundle operator (cf.~Definition \ref{Def:VBop}) as the map 
\begin{align*}
    (TM)^{\oplus n}\oplus(T^*M)^{\oplus m}&\longrightarrow (TM)^{\tens n}\tens (T^*M)^{\tens m}\\
    X_1\oplus\ldots\oplus X_n\oplus \om_1\oplus\ldots\oplus \om_m&\longmapsto X_1\tens\ldots\tens X_n \tens \om_1\tens\ldots\tens \om_m
\end{align*}
is a smooth (nonlinear) fiber-preserving map. Hence, $\tens^n_m$ is zero-tame (cf. Lemma \ref{VBopSmoothTame}).

Now we observe that the map 
\begin{align*}
    \Tilde{P}\coloneqq \sum_{i\in I}\sum_{\substack{\alpha_1,\ldots,\alpha_n\\\beta_1,\ldots,\beta_m}} \tens^n_m\circ P_m^n\circ \left(\mathrm{Id}_{\Diff} \times \pi^{(\alpha_1,\ldots\alpha_n)}_{i,(\beta_1,\ldots,\beta_m)}\right):\Diff\times \tensf{n}{m}\to \tensf{n}{m}
\end{align*}
is smooth tame as a sum over compositions of smooth tame maps.
We claim that the map $\Tilde{P}$ is already $P$. For that purpose let $f\in \Diff$ and $\tau\in \tensf{n}{m}$. Then 
\begin{align*}
    &\Tilde{P}(f,\tau)\\
    &= \sum_{i\in I}\sum_{\substack{\alpha_1,\ldots,\alpha_n\\\beta_1,\ldots,\beta_m}} \tens^n_m\big{(} f^*(\tau^{\alpha_1,\ldots,\alpha_n}_{i,\beta_1,\ldots,\beta_m}\rho_ib_i \del_{x_i^{\alpha_1}}),f^*(b_i\del_{x_i^{\alpha_2}}),\ldots,f^*(b_i\del_{x_i^{\alpha_2}}),f^*(b_i \d x_i^{\beta_1}),\ldots,f^*(b_i \d x_i^{\beta_m})\big{)}\\
    &=\sum_{i\in I}f^*(\rho_ib_i^n \tau)=\sum_{i\in I}f^*(\rho_i\tau)=f^*\tau=P(f,\tau).
\end{align*}
Here we use that $\rho_ib_i=\rho_i$ since $\rho_i$ has support in $U_i$ but $\left.b_i\right|_{U_i}=1$. A similar proof that $\Tilde{P}$ and $P$ agree also holds in the case that $n=0$.
This shows that $P$ is smooth tame and concludes the proof.

\end{proof}

\begin{corollary}
\label{Cor:PullBackForms}
The map 
\begin{align*}
    P:\Diff \times \Forms{n}&\longrightarrow \Forms{n}\\
    (f,\om)&\longmapsto f^*\om
\end{align*}
is smooth tame for every $n\in \N$.
\end{corollary}
\begin{proof}
This follows directly from Theorem \ref{PullbackSmoothTame} and Corollary \ref{Cor:FsubbundleE->FtameDirectSummand} where we identify $\bigwedge^n(T^*M)$ with the subbundle of $(T^*M)^{\tens n}$ of antisymmetric tensors.
\end{proof}

\begin{convention}
\label{Conv:StandChartOnDifftimesTensOrForms}
Let $(\Ucal,\exh_f)$ be the standard local parametrization of $\Diff$ centered at $f\in \Diff$ (cf.~Convention \ref{Conv:ChartsOnDiff}). We then denote the resulting standard local parametrization of $\Diff \times \tensf{n}{m}$, $n,m \in \N$, centered at $(f,0)$ by
$\exh_{f,\mathcal T}\coloneqq \exh_f\times \mathrm{Id}_{\tensf{n}{m}}$ and similarly we denote the standard local parametrization of $\Diff \times \Forms{k}$, $k\in \N$, centered at $(f,0)$ by $\exh_{f,\Omega}\coloneqq \exh_f \times \mathrm{Id}_{\Forms{k}}$. Note that $\exh_{f,\mathcal T}$, and $\exh_{f,\Omega}$ are tame diffeomorphisms onto their image and so the inverses of these maps are local charts of the respective manifolds. For $f=\Id$ we will again denote the respective local parametrizations by $\exh_{\mathcal{T}}\coloneqq \exh_{\Id,\mathcal{T}}$ and $\exh_{\Omega}\coloneqq \exh_{\Id,\Omega}$. The respective degrees will be clear from context. 
\end{convention}

\begin{lemma}
\label{Lemma:PullbackDifferential}
Let $P:\Diff\times \tensf{n}{m}\to \tensf{n}{m}$, $n,m\in \N$, be the pullback as in the previous theorem. Then in the chart $\exh_{f,\mathcal T}$ (cf.~Convention \ref{Conv:StandChartOnDifftimesTensOrForms}) centered at $(f,0)\in \Diff\times \tensf{n}{m}$ we find $D(P\circ \exh_{f,\mathcal T})(0,\tau)(X,T)=f^*\LieD_X \tau+f^*T$ for
$\tau,T\in \tensf{n}{m}$, and $X\in \VFs$.  
\end{lemma}
\begin{proof}
Let $\Phi^X:\R\to \Diff$ be the flow of the vector field $X\in \VFs$. We have seen in Lemma \ref{Lemma:FlowInitialVelocity} that $((\exh)^{-1}\circ \Phi^X)'(t=0)=X$ and clearly $\Phi^X_0=\Id$. So, by the chain rule, and continuity and linearity of the evaluation at the point $p\in M$ (cf.~Lemma \ref{Prop:EvVBSmooth}) we find that
\begin{align}
\begin{split}
\label{LieDviaExph}
    \left(D_1(P\circ \exh_\mathcal{T})(0,\tau)X\right)_p&=\left(\left.\frac{\d}{\d t}\right|_{t=0}P(\eh{tX},\tau)\right)_p  =\left(\left.\frac{\d}{\d t}\right|_{t=0}P(\Phi^X_t,\tau)\right)_p\\
    &=\left.\frac{\d}{\d t}\right|_{t=0}(P(\Phi^X_t,\tau))_p=(\LieD_X\tau)_p
\end{split}
\end{align}
for every $\tau \in \tensf{n}{m}$.

We can now use this to compute the differential of $P$. For this purpose let $f\in \Diff$, $X\in \VFs$ and $\tau,T \in \tensf{n}{m}$. Noting that $(f,\tau)$ corresponds to $(0,\tau)\in \VFs\times \tensf{n}{m}$ in the chart $\exh_{f,\mathcal T}$, we obtain in this chart that
\begin{align*}
    D(P\circ \exh_{f,\mathcal{T}})(0,\tau)(X,T)&=\left.\frac{\d}{\d t}\right|_{t=0}P(\eh{tX}\circ f,\tau)+\left.\frac{\d}{\d t}\right|_{t=0}P(f,\tau+tT)\\
    &=\left.\frac{\d}{\d t}\right|_{t=0}f^*(\eh{tX}^*\tau)+f^*T=f^*\LieD_X\tau+f^*T.
\end{align*}
Here we used equation \eqref{LieDviaExph} and the fact that $P(f,-):\tensf{n}{m}\to \tensf{n}{m}$ is a (tame) continuous linear map.
\end{proof}

\begin{remark}
\label{Rem:PushForwDifferential}
Let $V$ denote the inversion map of $\Diff$. The pushforward of a tensor field is the map $P\circ (V\times \mathrm{Id}_{\tensf{n}{m}}):\Diff\times \tensf{n}{m}\to \tensf{n}{m}$, $n,m\in \N$, which is smooth tame as a composition of smooth tame maps. 
Completely analogously to Lemma \ref{Lemma:PullbackDifferential} we can now take the flow $\Phi_t^X$ of a vector field $X\in \VFs$ and calculate for $\tau \in \tensf{n}{m}$ the differential
\begin{align*}
    \left.\frac{\d}{\d t}\right|_{t=0}P(V(\eh{tX}),\tau)=\left.\frac{\d}{\d t}\right|_{t=0}P(V(\Phi_t^X),\tau)=\left.\frac{\d}{\d t}\right|_{t=0}P(\Phi_{-t}^X,\tau)=-\LieD_X\tau.
\end{align*}
As before this shows that in the standard chart $\exh_{f,\mathcal T}$ (cf.~Convention \ref{Conv:StandChartOnDifftimesTensOrForms}) around $(f,0)\in \Diff\times \tensf{n}{m}$ we find
\begin{align*}
    D(P\circ (V\times \mathrm{Id}_{\Forms{n}})\circ \exh_{f,\mathcal T})(0,\tau)(X,T)&=\left(\left.\frac{\d}{\d t}\right|_{t=0}P(\Phi_{-t}^X,f_*\tau)\right)+f_*T\\
    &=-\LieD_X(f_*\tau)+f_*T
\end{align*}
for $X\in \VFs$ and $T\in \tensf{n}{m}$.
\end{remark}

\section{The tame Lie group of autoequivalences of an exact Courant algebroid over a compact base}
\label{Section:TheGroupOfautoequivalencesOfAnExCAAsATameLieGroup}

In this section $M$ always denotes a compact manifold, $H$ denotes a closed differential $3$-form $H\in \ClForms{3}$, and $(\GenTM,H)$ denotes the generalized tangent bundle twisted by $H$ as defined in Definition \ref{Def:GenTanBund} and Convention \ref{Conv:ExCAshorthand}.

In this section, we show that the space of diffeomorphisms $\SDiff\subseteq \Diff$ preserving the \Severa\ class is an open subgroup of $\Diff$.
We show that $\Auteq$ is a tame Fréchet manifold and that its manifold structure is given by 
$\Auteq \cong \SDiff \times \ClForms{2}$.
We show that, equipped with this tame Fréchet manifold structure, $\Auteq$ becomes a tame Fréchet Lie group. Moreover, we show that $\Auteq$ and $\mathrm{Auteq}(M,H')$ are isomorphic as tame Fréchet Lie groups if the \Severa\ classes agree, i.e.~$[H]=[H']\in H_{\mathrm{dR}}^3(M)$. We conclude that autoequivalences of an exact Courant algebroid over a compact manifold have a unique tame Fréchet Lie group structure up to isomorphism of tame Fréchet Lie groups.
We compute the Lie algebra of the tame Fréchet Lie group $\Auteq$.
Lastly, we equip the space of generalized almost complex structures of $(\GenTM,H)$ with the structure of a tame Fréchet manifold and show that the canonical left action of $\Auteq$ on it is smooth tame.

\subsection{The group of diffeomorphisms preserving the \Severa\ class}

When describing the autoequivalences of an exact Courant algebroid $(\GenTM,H)$, we showed in Theorem \ref{AuteqCharacterizationTheorem} that for $f\in \Diff$ and $B\in \Forms{2}$ the map $f_*e^B$ is an autoequivalence of $(\GenTM,H)$ if and only if $f^*H-H=\d B$. In particular, that means that $f_*e^B$ is an autoequivalence only if the \Severa\ class is preserved by $f$, i.e.~only if
$[H]=[f^*H]\in H_{dR}^3(M)$.
So we first wish to take a closer look at the following group.

\begin{definition}
We denote the diffeomorphisms preserving the \Severa\ class of $H$ by 
\begin{equation*}
    \SDiff\coloneqq \{f\in \mathrm{Diff}(M)|~[f^*H]=[H]\}.
\end{equation*}
We will also call this group the \textit{group of \Severa-diffeomorphisms}.
\end{definition}

\begin{theorem}
\label{Thm:SDiffClopen}
The subset $\SDiff\subseteq \Diff$ is a subgroup and is open and closed as a subset of $\Diff$. In particular, $\SDiff$ is a tame Lie group. 
\end{theorem}
\begin{proof}
It is clear that $\SDiff\subseteq \Diff$ is a subgroup.
The rest of the statement is a direct result of the homotopy invariance of de Rham cohomology (\cite[Proposition 17.10]{Lee2003}). Noting that a path in $\Diff$ corresponds to an isotopy by Remark \ref{Rem:C^nFamilies}, we find that the path-connected component of any $f\in \SDiff$ lies again in $\SDiff$. So, since $\Diff$ is a manifold, $\SDiff\subseteq \Diff$ is open and closed as a subset by the usual arguments. Since $\SDiff$ is open in $\Diff$ and is a subgroup, it is a tame Lie group.
\end{proof}

\subsection{The tame Fréchet manifold of autoequivalences}
In this section, we establish the tame Fréchet manifold structure on $\Auteq$ that will make it a tame Fréchet Lie group. We will show that as a tame Fréchet manifold $\Auteq$ can be identified with $\SDiff\times \ClForms{2}$. Here, as before, $\ClForms{2}$ denotes the closed differential $2$-forms.

First, we want to lay out a strategy on how we can equip the group $\mathrm{Auteq}(M,H)$ with the tame manifold structure of a tame principal bundle (cf.~Definition \ref{Def:PrincBundle}) with structure group $(\Forms{k}_{\mathrm{ex}},+)$ and base $\SDiff$. We show that, as such, it is indeed globally trivializable. Hence, the manifold structure of the group of autoequivalences of an exact Courant algebroid $(\GenTM,H)$ can be shown to be the product $\SDiff\times \ClForms{2}$ in a sensible way.
However, the group structure is more complicated than the product structure.

From Theorem \ref{AuteqCharacterizationTheorem} it is clear that the map 
$$\{(g,C)\in \SDiff\times \Forms{k}|~g^*H-H=\d C\}\ni (f,B)\to f_* e^B\in \Auteq$$ is bijective.
Now we can look at the problem more systematically. If we fix a diffeomorphism $f\in \SDiff$, then we know that there has to be a $B\in \Forms{2}$ such that $f^*H-H=\d B$. So, any other $b\in \Forms{2}$ such that $f^*H-H=\d b$ differs from $B$ by a closed differential $2$-form, i.e.~$B-b\in \Forms{2}_{\mathrm{cl}}$. Hence, for a fixed $f\in \SDiff$ the permissible differential forms $B\in \Forms{2}$ such that $f_*e^B$ is an autoequivalence form an affine space modeled on $\Forms{2}_{\mathrm{cl}}$.

We will now describe a way to define a smooth tame manifold structure and, in fact, the structure of a smooth tame principal bundle on $\Auteq$. First, we will define it abstractly and then make sure that this is actually a sensible way to define the structure. 

\begin{definition}
\label{gauge}
Let $\Vcal \subseteq \SDiff$ be open. We call a smooth tame map 
$\psi:\Vcal\to \Forms{k}$ a \textit{gauge over $\Vcal$} (or simply a \textit{local gauge}) if it fulfills 
\begin{align*}
    f^*H-H=\d\psi(f)
\end{align*}
for every $f\in \Vcal$. If $\Vcal=\SDiff$, then we refer to $\psi$ as a \textit{global gauge}.
\end{definition}

\begin{definition}
\label{Triv}
Let $\Vcal\subseteq \SDiff$ be open. We call a map $\Psi:\Vcal\times \ClForms{2}\to \Auteq$ a \textit{local trivialization of $\Auteq$} if it is of the form $\Psi(f,b)=f_*e^{b+\psi(f)}$, $(f,b)\in \Vcal\times \ClForms{2}$, where $\psi$ is some gauge over $\Vcal$.
\end{definition}

\begin{remark}
Let $\Psi$ be a local trivialization as above. Note that usually we would call $\Psi^{-1}$ a local trivialization. However, in this context it is simply more convenient to denote it this way around.    
\end{remark}

\begin{remark}
\label{Rem:TransFunctAuteq}
Let $\Vcal\subseteq \SDiff$ be open. Now, if we have two such local trivializations $\Psi$, $\Psi'$ defined without loss of generality on $\Vcal\times \ClForms{2}$ induced by gauges $\psi$ and $\psi'$ over $\Vcal$ respectively then the change of trivialization has the following form
\begin{align*}
    \Psi'^{-1}\circ \Psi (f,b)=(f,b+(\psi-\psi')(f)).
\end{align*}
Note that this map is smooth tame since by definition the gauges are smooth tame. Also note that $\d(\psi(f)-\psi'(f))=f^*H-H-(f^*H-H)=0$. Hence, the transition maps are precisely given by addition of a smooth tame family of closed two-forms.
\end{remark}

Assuming that such gauges always exist locally we can induce a topology and smooth tame structure on $\Auteq$. Note that the topology will automatically be Hausdorff.
So the question becomes that of existence of local gauges.
In fact we will now show that there is a global gauge. In other words in the abstract framework which we have set up we may in fact write $\Auteq$ as the product manifold $\SDiff\times \ClForms{2}$. However, note that this is only the manifold structure and explicitly not the group structure.

\begin{lemma}
\label{Lemma:PhiGlobGauge}
The map $\phi$ defined by 
\begin{align*}
    \phi:\SDiff&\longrightarrow \Forms{2}\\
    f&\longmapsto Q(f^*H-H)
\end{align*}
is a global gauge.
Here $Q$ again denotes the operator $\d^*G$ introduced in Lemma \ref{Lemma:dQex=Id}
\end{lemma}
\begin{proof}
Since for every $k\in \N$ the pullback $P:\mathrm{Diff}(M)\times \Forms{k}\to \Forms{k}$ is a smooth tame family of isomorphisms (cf.~Corollary \ref{Cor:PullBackForms}), and the operator $Q:\Forms{k}\to \Forms{k+1}$ is a tame linear map (cf.~Remark \ref{Rem:HodgeOpsSmoothTame}) the map
$\phi:\SDiff\to \Forms{2}$ is smooth tame. Also note that by definition if $f\in \SDiff$, then $f^*H-H$ is exact. So by Lemma \ref{Lemma:dQex=Id} the equation $\d(Q(f^*H-H))=f^*H-H$ holds. Hence $\phi$ defines a global gauge. 
\end{proof}

\begin{convention}
\label{Phi,phi}
In the previous lemma we have constructed the canonical global gauge $\phi$.
We denote the associated global trivialization by $\Phi$. We will use this notation throughout the rest of the thesis.  
\end{convention}

\begin{theorem}
\label{Thm:AuteqTameMfd}
The space $\Auteq$ can be equipped with the structure of a tame Fréchet manifold smooth tamely diffeomorphic to $\SDiff\times \ClForms{2}$.
\end{theorem}
\begin{proof}
We use $\Phi:\SDiff\times \ClForms{2}\to \Auteq$ to induce a topology and smooth structure on $\Auteq$, i.e.~we use the unique topology on $\Auteq$ making $\Phi$ into a homeomorphism, and if $(\Ucal,\varphi)$ is a chart on $\SDiff\times \ClForms{2}$,
$\Ucal \subset \SDiff\times \ClForms{2}$ open, then the induced chart on $\Auteq$ is $(\Phi(\Ucal),\varphi\circ\Phi^{-1})$. By construction this gives $\Auteq$ the structure of a tame Fréchet manifold and makes $\Phi$ into a smooth tame diffeomorphism.
\end{proof}

\begin{remark}
The result from Theorem \ref{Thm:AuteqTameMfd} gives $\Auteq$ the structure of a globally trivial tame principal bundle (cf.~Definition \ref{Def:PrincBundle}) with structure group $(\ClForms{2},+)$ and base $\SDiff$. 
Also note that automatically any local trivialization (cf.~Definition \ref{Triv}) is a smooth tame diffeomorphism onto its image.  
Moreover, by Remark \ref{Rem:PrincBundleViaTransFunct} and Remark \ref{Rem:TransFunctAuteq} we find that the action of $\ClForms{2}$ on $\Auteq$ is the trivial action (cf.~Definition \ref{Def:PrincBundle} \ref{item:LocalActionOfGonP}) in each local trivialization (as in Definition \ref{Triv}).
\end{remark}

\begin{remark}
\label{AuteqChart}
For any $f\in \SDiff$ we can use the local parametrization $(\Ucal_{\Id},\exh_f)$ (cf.~Convention \ref{Conv:ChartsOnDiff}), and any gauge $\psi$ over $R_f(\eh{\Ucal})$ (cf.~Definition \ref{gauge}) with associated trivialization $\Psi$ (cf.~Definition \ref{Triv}) to identify the open subset
$\Ucal\times \ClForms{2}\subseteq \VFs\times \ClForms{2}$ with its image under $\Psi\circ (\exh_f\times \mathrm{Id}_{\ClForms{2}})$ in $\Auteq$. This provides an explicit local parametrization of $\Auteq$, the inverse of which is a chart. We will denote this local parametrization by $\exh_{f,\Psi}$.
\end{remark}

\begin{remark}
The group of autoequivalences $\Auteq$ viewed as the set 
$\{(f,B)\in\Diff\times \Forms{2}|~f^*H-H=\d B \}$ is a tame submanifold of $\Diff\times \Forms{2}$.
In fact, we may extend the global trivialization $\Phi$ to the tame diffeomorphism
$\Phi':\Diff\times \Forms{2}\to \Diff\times \Forms{2}$ such that $\Phi'(f,\om)=(f,\om + Q(f^*H-H))$ for $f\in \Diff$, $\om \in \Forms{2}$. We know that $\ClForms{2}\subseteq \Forms{2}$ is a tame direct summand (cf.~Corollary \ref{Cor:ClFormsTameDirSumm}) and that $\SDiff$ is an open subset of $\Diff$, so $\SDiff\times \ClForms{2}\subseteq \Diff\times \Forms{2}$ is a tame submanifold. Now $\Auteq$ corresponds precisely to
$\Phi'(\SDiff\times \ClForms{2})$, which is, hence, a tame Fréchet submanifold. 
\end{remark}

\subsection{The tame Lie group structure of the group of autoequivalences}
\label{SubSect:LieGrStrOfAuteq}

In this section, we show that the group multiplication and inversion map of $\Auteq$ are smooth tame maps. Moreover, we show that if $H,H'\in \ClForms{3}$ induce the same \Severa\ class $[H]=[H']\in H^3_{\mathrm{dR}}(M)$, then $\Auteq \cong  \mathrm{Auteq}(M,H')$ as tame Fréchet Lie groups and so the tame Fréchet Lie group structure of an exact Courant algebroid is determined uniquely up to isomorphism of tame Fréchet Lie groups.\\

The group multiplication map on $\Auteq$ will be denoted by
$\mu:\Auteq \times \Auteq \to \Auteq$. In the following, let $\Phi$ and $\phi$ be as in Convention \ref{Phi,phi}. Let $(f,b),(g,c)\in \SDiff\times \Forms{2}_{\mathrm{cl}}\stackrel{\Phi}{\cong}\Auteq$, then the multiplication of the associated autoequivalences gives us
\begin{align*}
    f_*e^{b+\phi(f)}g_*e^{c+\phi(g)}&=f_*g_*(g^{-1})_*e^{b+\phi(f)}g_*e^{c+\phi(g)}
    =f_*g_*e^{g^*(b+\phi(f))}e^{c+\phi(g)}\\
    &=(f\circ g)_*e^{(g^*b+c+g^*\phi(f)+\phi(g))}\\
    &=(f\circ g)_*e^{(g^*b+c+B(f,g)+\phi(f\circ g))}.
\end{align*}
Here we define
\begin{align}
\label{B(f,g)}
\begin{split}
    B(f,g)&\coloneqq g^*\phi(f)+\phi(g)-\phi(f\circ g)\\
    &= g^*Q(f^*H-H)+Q(g^*H-H)-Q((f\circ g)^*H-H)\\
    &=g^*Q(f^*H-H)-Qg^*(f^*H-H)=[g^*,Q](f^*H-H).
\end{split}
\end{align}
where $[-,-]$ denotes the commutator.
Hence, 
\begin{align}
\label{Prod}
    \mu^\Phi((f,b),(g,c))\coloneqq(\Phi^{-1}\circ \mu \circ (\Phi\times\Phi))((f,b),(g,c))=(f\circ g, g^*b+c+B(f,g)).
\end{align}
Since the pullback is smooth tame by Theorem \ref{PullbackSmoothTame} and $f\circ g$ is merely the product of elements in the tame Lie group $\SDiff$ (cf.~Theorem \ref{Thm:SDiffClopen}) we get that the product on $\Auteq$ is smooth tame. 

Now we are concerned with the inversion mapping on $\Auteq$ which we will denote by $V:\Auteq\to\Auteq$.
Let $f_*e^B$ be an autoequivalence, i.e.
$f\in \SDiff$ and $B\in \Forms{2}$ such that $f^*H-H=\d B$. Then we can easily find its inverse. It is clear that if $g_*e^C$ is the inverse of $f_*e^B$ then $g=f^{-1}$. We have
\begin{align*}
    f^{-1}_*e^Cf_*e^B=e^{f^*C+B}.
\end{align*}
So in order to obtain the identity it follows that $B=-f^*C$ or equivalently
$C=-f_*B=-(f^{-1})^*B$. 
Let us now compute this with respect to the gauge $\phi$ inducing the global trivialization $\Phi$ (cf.~Convention \ref{Phi,phi}).
If we have $(f,b)\in \SDiff\times \ClForms{2}$ then we find
\begin{align}
\begin{split}
\label{Inv}
    V^\Phi(f,b)&\coloneqq (\Phi^{-1}\circ V\circ \Phi)(f,b)=\Phi^{-1}\left(V\left(f_*e^{b+\phi(f)}\right)\right)\\
    &=\Phi^{-1}\left(f^{-1}_*e^{(-f_*b-f_*\phi(f))}\right)=\Phi^{-1}\left(f^{-1}_*e^{(-f_*b+I(f)+\phi(f^{-1}))}\right)\\
    &=(f^{-1},-f_*b+I(f)).
\end{split}
\end{align}
Here, we define 
\begin{align}
\begin{split}
\label{I(f)}
    I(f)&\coloneqq -f_*\phi(f)-\phi(f^{-1})\\
    &= -f_*(Q(f^*H-H))-Q((f^{-1})^*H-H)=[Q,f_*](f^*H-H).
\end{split}
\end{align}
We have already shown that the pullback and push forward are smooth tame operations in Theorem \ref{PullbackSmoothTame} and Remark \ref{Rem:PushForwDifferential}. Hence, $I:\SDiff\to \ClForms{2}$ is smooth tame.

To conclude, we have found that on $\Auteq$ the product as well as the inversion map are smooth tame in the tame manifold structure we had previously established on it in Theorem \ref{Thm:AuteqTameMfd}. As a direct result, we can formulate the following theorem.

\begin{theorem}
\label{Thm:AuteqTameLieGroup}
Let $M$ be a compact manifold and let $H\in \ClForms{3}$. Then the group $\Auteq$ equipped with the structure of a tame Fréchet manifold as in Theorem \ref{Thm:AuteqTameMfd} is a tame Fréchet Lie group.
\end{theorem}

We already know that if $H,H'\in\ClForms{3}$ such that their de Rham cohomology classes agree, i.e.~$[H]=[H']\in H_{dR}^3(M)$, then $(\mathbb{T}M,H)\cong (\mathbb{T}M,H')$ as Courant algebroids via the homomorphism $e^B:(\mathbb{T}M,H)\to (\mathbb{T}M,H')$ with $B\in \Forms{2}$ and $\d B=H'-H$ (cf.~Proposition \ref{Prop:SeveraClass}). 

Thus, if $f_*e^b\in \Auteq$ for $f\in \SDiff$ and $b\in \Forms{2}$, then 
$f_*e^{b+f^*B-B}=f_*e^be^{f^*B-B}\in \mathrm{Auteq}(M,H')$ since $$f^*H'-H'=f^*H-H+f^*\d B-\d B=\d(b+f^*B-B).$$ 
This induces a map $R_B:\Auteq \to \mathrm{Auteq}(M,H')$ by $R_B(f_*e^b)=f_*e^{b+f^*B-B}$. Clearly, this map is invertible with inverse $(R_B)^{-1}=R_{-B}:\mathrm{Auteq}(M,H')\to \Auteq$. 

We want to check that this is indeed a group homomorphism. 
For this purpose let $f_*e^C,g_*e^D\in\mathrm{Auteq}(M,H)$ with $f,g\in \SDiff$ and $C,D\in \Forms{2}$. Then, we have 
\begin{align*}
    \mu(R_B(f_*e^C),R_B(g_*e^D))&=f_*e^Ce^{f^*B-B}g_*e^De^{g^*B-B}=f_*e^Cg_*(g^{-1})_*e^{f^*B-B}g_*e^De^{g^*B-B}\\
    &=f_*e^Cg_*e^De^{g^*(f^*B)-g^*B+g^*B-B}=f_*e^Cg_*e^De^{(f\circ g)^*B-B}\\
    &=R_B(\mu(f_*e^C,g_*e^D)).
\end{align*}
This shows that $R_B$ is indeed a group isomorphism.

Let again $\Phi$ (cf.~Convention \ref{Phi,phi}) denote the usual trivialization of $\Auteq$, and let $\Phi'$ denote the analogous trivialization of $\mathrm{Auteq}(M,H')$. Then, we have for $(f,b)\in \SDiff\times \ClForms{2}$ that 
\begin{align*}
    ((\Phi')^{-1}\circ R_B \circ \Phi)(f,b)&=(\Phi')^{-1}(f_*\mathrm{exp}(b+f^*B-B+Q(f^*H-H)))\\
    &=(\Phi')^{-1}(f_*\mathrm{exp}(b+f^*B-B-Q(f^*\d B-\d B)+Q(f^*H'-H')))\\
    &=(f,b+(\mathrm{Id}-Q\d)(f^*B-B)).
\end{align*}
Hence $R_B$ is smooth tame as the pullback is a smooth tame operation (cf.~Lemma \ref{PullbackSmoothTame}) and $Q$ is tame linear (cf.~Remark \ref{Rem:HodgeOpsSmoothTame}). Since its inverse is of the form $R_{-B}$ it is also smooth tame and hence, $R_B$ is an isomorphism of tame Fréchet Lie Groups.
As a direct consequence we may formulate the following theorem. 

\begin{theorem}
\label{Thm:IsoClassesOfAuteq}
Let $E$ be an exact Courant algebroid over a compact manifold $M$. Then by Proposition \ref{Prop:SeveraClass} there is some $H\in \ClForms{3}$ such that $E\cong (\GenTM,H)$ as Courant algebroids. We can identify the group of autoequivalences of $E$ with the tame Fréchet Lie group $\Auteq$. The identification is unique up to the isomorphism of tame Fréchet Lie groups.
\end{theorem}

Lastly, we summarize these statements by proving Theorem \ref{MainThm:AuteqMfdAndLieGr}.

\begin{proof}[Proof of Theorem \ref{MainThm:AuteqMfdAndLieGr}]
    The result directly follows from Theorem \ref{Thm:AuteqTameMfd}, Theorem \ref{Thm:AuteqTameLieGroup}, and Theorem \ref{Thm:IsoClassesOfAuteq}.
\end{proof}

\subsection{The Lie algebra of the group of autoequivalences}

The aim in this section is to compute the Lie algebra $\auteq\coloneqq \mathrm{Lie}(\Auteq)$.
To do so, we follow Theorem \ref{Thm:ad=[-,-]} using the formula for the adjoint action in a local chart discussed in Remark \ref{Rem:LocalFormulaForad}.
This way we will find the Lie bracket on $T_{\mathrm{Id}}\Auteq$ turning it into the Lie algebra $\auteq$. Rubio and Tipler have already determined the Lie algebra of $\Auteq$ viewed as an ILH Lie group, see \cite{RubioTipler}. We show that their result agrees with ours up to isomorphism of tame Fréchet Lie algebras.\\

We are now ready to prove Theorem \ref{MainThm:LieAlg}. For the readers convenience we restate it here and make the local chart in which we aim to compute the Lie algebra precise.

\begin{theorem*}[Explicit form of Theorem \ref{MainThm:LieAlg}]
The Lie algebra of $\Auteq$ in the local parametrization $\exh_{\Id,\Phi}$ (cf.~Remark \ref{AuteqChart} and Convention \ref{Phi,phi}) centered at $\mathrm{Id}_{\GenTM}\in \Auteq$ is
$(\VFs\times \ClForms{2},[-,-])$ with the Lie bracket
\begin{align*}
    [(X,B),(Y,C)]=(\LieD_YX,\LieD_YB-\LieD_XC-(\LieD_XQ\LieD_Y-\LieD_YQ\LieD_X)H+Q(\LieD_{\LieD_XY}H))
\end{align*}
for $X,Y\in \VFs$ and $B,C\in \ClForms{2}$. 
\end{theorem*}
\begin{proof}
Since $\SDiff$ contains the (path-)connected component of $\Id$ in $\Diff$ (cf.~Theorem \ref{Thm:SDiffClopen}) we know that the Lie algebra of $\SDiff$ is $\VFs$ with the Lie bracket being $(X,Y)\mapsto \LieD_YX$ for $X,Y\in \VFs$ (cf.~Proposition \ref{Proposition:DiffLieGrAndLieAlg}). 

We also know that the underlying vector space of $\auteq$ is $T_{\mathrm{Id}_{\GenTM}}\Auteq\cong \VFs\times \ClForms{2}$ using the parametrization $\exh_{\Id,\Phi}$ (cf.~Remark \ref{AuteqChart} and Convention \ref{Phi,phi}) of $\Auteq$ around $\mathrm{Id}_{\GenTM}$. 

To compute the Lie bracket we make use of Theorem \ref{Thm:ad=[-,-]} and Remark \ref{Rem:LocalFormulaForad}, and calculate the adjoint representation $\mathrm{ad}^\Phi$ of $\auteq$ (cf.~Definition \ref{Def:AdAndad}) under the identification $\auteq \cong \VFs \times \ClForms{2}$ of tame Fréchet spaces via the local parametrization $\exh_{(\Id,\Phi)}$. We will first calculate the adjoint action $\mathrm{Ad}^\Phi$ of $\Auteq$ on its Lie algebra using appropriate identifications. 

To do so, we first express the left- and right translation via the canonical identification in terms of $(f,b),(g,c)\in \SDiff\times \ClForms{2}\stackrel{\Phi}{\cong}\Auteq$.
So, in the following we write 
\begin{align*}
    L^\Phi_{(f,b)}(g,c)&\coloneqq \Phi^{-1}(L_{\Phi(f,b)}\Phi(g,c))\eqqcolon R^\Phi_{(g,c)}(f,b) \text{ , and}\\
    V^\Phi(f,b)&=\Phi^{-1}(V(\Phi(f,b)))
\end{align*}
for $f,g\in \SDiff$ and $b,c\in \ClForms{2}$.
We obtain
\begin{align*}
    L^\Phi_{(f,b)}(g,c)&=(f\circ g,g^*b+c+B(f,g))\\
    R^\Phi_{(f,b)}(g,c)&=(g\circ f,f^*c+b+B(g,f)).
\end{align*}
This follows immediately from equation \eqref{Prod} with $B(f,g)$ being defined as in equation \eqref{B(f,g)}.
Using equation \eqref{Prod} and equation \eqref{Inv} we find
\begin{align*}
    C^\Phi_{(f,b)}(g,c)&\coloneqq (R^\Phi_{V^\Phi(f,b)}\circ L^\Phi_{(f,b)})(g,c)\\
    &=R^\Phi_{(f^{-1},-f_*b+I(f))}(f\circ g, g^*b+c+B(f,g))\\
    &=(c_f(g),f_*(g^*b+c+B(f,g))-f_*b+I(f)+B(f\circ g,f^{-1}))\\
    &=(c_f(g),f_*(g^*b-b+c+B(f,g))+I(f)+B(f\circ g,f^{-1}))
\end{align*}
with $c_f(g)\coloneqq f\circ g\circ f^{-1}$ and $I(f)$ as in equation \eqref{I(f)}.
We remind ourselves that $B(f,g)$ was defined in equation \eqref{B(f,g)} as
\begin{align*}
    B(f,g)=[g^*,Q]_c(f^*H-H)
\end{align*}
with $[-,-]_c$ denoting the commutator and hence,
\begin{align*}
    B(f\circ g,f^{-1})=[f_*,Q]_c((f\circ g)^*H-H)=f_*Q(g^*(f^*H)-H)-Q(c_f(g)^*H-f_*H).
\end{align*}
We note that previously we have already worked within the global gauge $\phi$ (cf. Convention \ref{Phi,phi}). Now, to work in the local parametrization $\exh_{(\Id,\Phi)}$ of $\Auteq$ centered at $\mathrm{Id}_{\GenTM}$, we merely need to take the local parametrization $\exh\times \mathrm{Id}_{\ClForms{2}}$ on $\SDiff\times \ClForms{2}$ centered at $(\Id,0)$ (cf. Convention \ref{Conv:ChartsOnDiff}). This is implicitly done in the following calculations. So, the resulting Lie bracket is, indeed, expressed with respect to the local parametrization $\exh_{(\Id,\Phi)}$. 

We now want to compute the adjoint representation of $\Auteq$ on its Lie algebra using the established identifications. To this end, let $(Y,C)\in\VFs\times \ClForms{2}$ then using Lemma \ref{Lemma:PullbackVfields}, Lemma \ref{Lemma:PullbackDifferential} and Remark \ref{Rem:PushForwDifferential} we compute
\begin{align*}
    \mathrm{Ad}^\Phi_{(f,b)}(Y,C) \coloneqq& \left.\frac{\d}{\d t}\right|_{t=0}C_{(f,b)}(\eh{tY},tC)\\
    =&(f_*Y,f_*[\LieD_Yb+C+\LieD_Y(Q(f^*H-H))-Q\LieD_Y(f^*H-H)]\\
    &+f_*Q(\LieD_Yf^*H)-Q(f_*\LieD_Yf^*H)).
\end{align*}
Following Theorem \ref{Thm:ad=[-,-]} and Remark \ref{Rem:LocalFormulaForad}, we use again Lemma \ref{Lemma:PullbackDifferential} and Remark \ref{Rem:PushForwDifferential} to compute the Lie bracket of two elements
$(X,C),(Y,D)\in\VFs\times \ClForms{2}\cong \auteq$ as follows:
\begin{align}
\label{eq:LieBracketauteq}
\begin{split}
    [(X,B),(Y,C)]=&\mathrm{ad}((X,B),(Y,C))\coloneqq \left.\frac{\d}{\d t}\right|_{t=0}\mathrm{Ad}^\Phi_{(\eh{tX},tB)}(Y,C)\\
    =& (\LieD_YX,\LieD_YB-\LieD_XC+\LieD_Y(Q(\LieD_XH))-Q(\LieD_Y\LieD_XH)\\
    &-\LieD_XQ(\LieD_YH)+Q(\LieD_Y\LieD_XH)+Q(\LieD_X\LieD_YH)-Q(\LieD_Y\LieD_XH))\\
    =&(\LieD_YX,\LieD_YB-\LieD_XC-(\LieD_XQ\LieD_Y-\LieD_YQ\LieD_X)H+Q(\LieD_{\LieD_XY}H)).
\end{split}
\end{align}
This finishes the proof.
\end{proof}

\begin{remark}
In \cite[Theorem 3.8]{RubioTipler} Rubio and Tipler have shown that the Lie algebra $\auteq$ can be identified with the Lie subalgebra 
\begin{align*}
    \gfrak'&\coloneqq\{(X,B)\in \VFs\times \Forms{2}|~\d(\iota_XH-B)=0\}\\
    &=\{(X,B)\in \VFs\times \Forms{2}|~\d B=\LieD_XH\}
\end{align*}
of $\VFs\times \Forms{2}$ with the bracket on $\VFs\times \Forms{2}$ being $$[(X,C),(Y,D)]'=(\LieD_XY,\LieD_XD-\LieD_YC).$$
Note that $[-,-]'$ differs from the natural bracket, which can be computed similarly to equation \eqref{eq:LieBracketauteq} with $H=0$ and considering non-closed differential forms, exactly by a sign.
Note as well, that as $\gfrak'=\ker(\VFs\times \Forms{2}\ni (X,B)\mapsto \d(\iota_XH-B)\in \VFs\times \Forms{2})$ it is a closed subspace of $\VFs\times \Forms{2}$ and, hence, a graded Fréchet space.

There is a tame linear isomorphism of Lie algebras 
\begin{align*}
    F:\gfrak'&\longrightarrow (\VFs\times \ClForms{2},[-,-])\\
    (X,B)&\longmapsto (-X,-B+Q\LieD_XH)
\end{align*}
with $[-,-]$ as in equation \eqref{eq:LieBracketauteq}.
The map $F$ is well defined since
$\d(-B+Q\LieD_XH)=-\LieD_XH+\LieD_XH=0$. It is clearly tame linear with a tame linear inverse. Also, it is an isomorphism of Lie algebras as the following calculation shows
\begin{align*}
    [F(X,B),F(Y,C)]=&[(-X,-B+Q\LieD_XH),(-Y,-C+Q\LieD_YH)]\\
    =&(\LieD_YX,-\LieD_Y(-B+Q\LieD_XH)+\LieD_X(-C+Q\LieD_YH)\\
    &-(\LieD_XQ\LieD_Y-\LieD_YQ\LieD_X)H+Q\LieD_{[X,Y]_V}H)\\
    =&(\LieD_YX,\LieD_YB-\LieD_XC+Q\LieD_{[X,Y]_V}H)\\
    =&F(\LieD_XY,\LieD_XC-\LieD_YB)\\
    =&F([(X,B),(Y,C)]').
\end{align*}
By Remark \ref{Rem:DirSumm&ProdAreTame} this additionally shows that $\gfrak'$ is, indeed, a tame Fréchet space, tamely isomorphic to $\VFs\times \ClForms{2}$.
\end{remark}

\begin{remark}
We had defined (semi-)regular Lie groups in Definition \ref{Def:RegularLieGr}.
In the following we will use that under certain circumstances the smooth maps from a compact manifold to a locally convex manifold can be equipped with the structure of a locally convex manifold, see \cite[C.3 Construction of the Manifold Structure]{schmeding_2022}. This is the structure we refer to throughout this remark.
Since $\Diff$ is a regular Lie group (cf.~Example \ref{Ex:DiffRegLieGr}), one might expect $\Auteq$ to be a regular Lie group. Indeed, this is true.
We use the diffeomorphism (of tame manifolds) $\Auteq\stackrel{\Phi}{\cong} \SDiff\times \ClForms{2}$ (cf.~Theorem \ref{Thm:AuteqTameMfd}), and the induced isomorphism (of tame vector spaces) $\auteq \cong \VFs\times \ClForms{2}$.
Now, let $\eta\in C^\infty([0,1],\VFs)$ and $\om\in C^\infty([0,1],\ClForms{2})$. By definition $\Auteq$ is semiregular if and only if there are curves $\gamma:[0,1]\to \SDiff$ and $\xi:[0,1]\to \ClForms{2}$ such that the equation
\begin{align}
\label{Eq:DiffEqAuteq}
    (\gamma',\xi')\coloneqq\frac{\d}{\d t}(\gamma,\xi)=DR^\Phi_{(\gamma,\xi)}(\Id,0)(\eta,\omega)
\end{align}
is satisfied.
We find that
\begin{align*}
    DR^\Phi_{(\gamma,\xi)}(\Id,0)(\eta,\om)=(\eta\circ \gamma,\gamma^*\om+[\LieD_\eta,Q](\gamma^* H-H)).  
\end{align*}
By the (semi-)regularity of $\Diff$ and hence the (semi-)regularity of $\SDiff$ (note that $\SDiff$ contains the (path-)connected component in $\Diff$ of $\Id$) we know that such a $\gamma$ exists. Actually, $\gamma$ is the flow $\mathrm{Fl}^\eta$ of the time-dependent vector field $\eta$.
Indeed the induced map $\mathrm{Fl}:C^\infty([0,1],\VFs)\to C^\infty([0,1],\SDiff)$ defined by $\eta\mapsto \mathrm{Fl}^\eta$ is well defined and smooth, see \cite[Example 3.36]{schmeding_2022}.
So, $$\xi(\eta,\om)(t)\coloneqq \int_0^t(\mathrm{Fl^\eta(s)}^*\om(s)+[\LieD_{\eta(s)},Q](\mathrm{Fl}^\eta(s)^* H-H))ds$$ is well defined. By the fundamental theorem of calculus we find that $(\mathrm{Fl}^\eta,\xi(\eta,\om))$ solves equation \eqref{Eq:DiffEqAuteq} so $\Auteq$ is semiregular.

The map 
\begin{align*}
    P:\SDiff\times \VFs \times \ClForms{2}&\longrightarrow \ClForms{2}\\
    (f,X,B)&\longmapsto f^*B+[\LieD_X,Q](f^*H-H)
\end{align*}
is clearly smooth (tame) and by \cite[Corollary 2.19, Proposition C.16, and p.~214, C.2]{schmeding_2022} it lifts to a smooth map
\begin{align*}
    \Tilde{P}:C^\infty([0,1],\SDiff\times \VFs \times \ClForms{2})\longrightarrow C^\infty([0,1],\ClForms{2}).
\end{align*}
By \cite[Corollary 2.19, Proposition C.16, and Definition 2.14]{schmeding_2022} we have that 
\begin{align*}
    &C^\infty([0,1],\SDiff\times \VFs \times \ClForms{2})\\
    &\cong C^\infty([0,1],\SDiff)\times C^\infty([0,1],\VFs) \times C^\infty([0,1],\ClForms{2})
\end{align*}
as locally convex manifolds.
So using this identification the map $$\hat{P}: C^\infty([0,1],\VFs\times\ClForms{2})\longrightarrow  C^\infty([0,1],\ClForms{2})$$ that sends $(\eta,\om)$ to $P(\mathrm{Fl}^\eta,\eta,\om)\in C^\infty([0,1],\Forms{2})$ is smooth. So we have that the map 
\begin{align*}
   C^\infty([0,1],\VFs\times\ClForms{2})\xrightarrow{\hat{P}}  C^\infty([0,1],\ClForms{2})\hookrightarrow C([0,1],\ClForms{2})\xrightarrow{\int_0^1}\ClForms{2}
\end{align*}
maps $(\eta,\om)$ to $\xi(\eta,\om)(1)$ and is smooth. In fact, note that in general for any locally convex space $F$ it is easy to see that $C^\infty([0,1],F)\hookrightarrow C([0,1],F)$ is a continuous linear map, compare \cite[C.15]{schmeding_2022} and note that $F$ is globally trivial. Moreover, we had also shown that since $\ClForms{2}$ is a Fréchet space and hence its own completion that the Riemann integral $\int_0^1:C([0,1],\ClForms{2})\to\ClForms{2}$ is a continuous linear map (cf.~Proposition \ref{Prop:RiemInt}). So we conclude that $\Auteq$ is regular.
\end{remark}

\subsection{The action of the group of autoequivalences on generalized almost complex structures}

An interesting application of the group of autoequivalences of $(\GenTM,H)$ appears in \cite[Theorem 5.4]{gualtieri} where Gualtieri uses the natural action of $\Auteq$ on the space of generalized (almost) complex structures of $(\GenTM,H)$.
In this section we recall the notion of generalized almost complex structures and equip the space of generalized almost complex structures of an exact Courant algebroid over a compact manifold with the structure of a zero-tame manifold. We show that the associated autoequivalences act smooth tamely on the space of generalized almost complex structures.

\begin{definition}
\begin{enumerate}
    \item Let $V$ be a finite-dimensional vector space and $\scal{-}{-}$ a scalar product on $V$. We call an automorphism $J:V\to V$ a \textit{linear generalized almost complex structure} if $J^2=-\mathrm{Id}_V$ and $J$ is orthogonal with respect to $\scal{-}{-}$.
    \item Let $M$ be a manifold, $E\to M$ be a vector bundle, and $\scal{-}{-}$ be a bundle metric on $E$. We say an element $J\in \Gamma(\mathrm{End}(E))$ is a generalized almost complex structure if $J^2=-\mathrm{Id}_E$, and $J$ is orthogonal with respect to $\scal{-}{-}$.
    \item We denote by $\Jcal(E,\scal{-}{-})\subset \mathrm{End}(E)$ the set of fiber-wise linear generalized almost complex structures on $E$ equipped with the scalar product $\scal{-}{-}$. Moreover, we write $\Jcal_M\coloneqq \Jcal(\GenTM,\scal{-}{-}_{\mathrm{can}})$ (cf.~Definition \ref{Def:GenTanBund}).
\end{enumerate}
\end{definition}

The following lemma is partially based on \cite[p.~43]{gualtieri}, \cite[Corollary 3.8.10]{gauge}, and \cite[Chapter II, Section 2.13, p.~65]{montgomery2018topological}.
\begin{lemma}
\label{Lemma:LinGACSSubmanifold}
Let $\scal{-}{-}$ denote the standard scalar product of signature $(n,n)$, $n\in\N$, on $\R^{2n}$. Then the set of linear complex structures $\Jcal_{n}\coloneqq \{A\in O(n,n)|~A^2=-\mathrm{Id}\}$ forms a non-empty embedded submanifold of $\mathrm{End}(\R^{2n})$ if and only if $n=2k$ for $k\in \N\setminus \{0\}$.
\end{lemma}
\begin{proof}
It is trivial that $\Jcal_0= \emptyset$, so we consider $n>0$.
We first note that if $J\in \Jcal_n$, $n>0$, then we can always construct an orthonormal basis of the form $\{e_1,...,e_n,Je_1,...,Je_n\}$. To do this, start with a vector $e_1\in \R^{2n}$ such that $\abs{\scal{e_1}{e_1}}=1$ then $\scal{Je_1}{e_1}=-\scal{e_1}{Je_1}=0$. So we have that $e_1$ and $Je_1$ are linearly independent and $\abs{\scal{Je_1}{Je_1}}=\abs{\scal{e_1}{e_1}}=1$. Repeat the process on $(\mathrm{span}(e_1,Je_1))^\perp$ until a full basis is gathered. Moreover, by orthogonality elements of $\Jcal_n$ preserve positive definite subspaces, and negative definite subspaces with respect to $\scal{-}{-}$ and, hence, $\Jcal_n$ is nonempty if and only if $n=2k>0$ is even.
Note, that if $n=2k>0$, $k\in \N$, then the standard linear generalized complex structure
\begin{align*}
    J_S\coloneqq
    \begin{pmatrix}
        0&0 & 0 &-\mathds{1}_{k\times k}\\
        0&0 & \mathds{1}_{k\times k} &0 \\
        0 &-\mathds{1}_{k\times k} & 0&0 \\
        \mathds{1}_{k\times k} &0 &0 &0 
    \end{pmatrix}
\end{align*}
exists.

So, let us now assume $n=2k$ with $k\in \N\setminus \{0\}$.
Clearly, the action
\begin{align*}
    c:O(n,n)\times \mathrm{End}(\R^{2n})&\longrightarrow \mathrm{End}(\R^{2n})\\
    (A,F)&\longmapsto (A\circ F\circ A^{-1})
\end{align*}
is smooth. The orbit of $J_S$ under the action $c$ is precisely $\Jcal_{n}$. In fact, it is clear that $c(A,J_S)\in \Jcal_n$ for every $A\in O(n,n)$. By the existence of a standard orthonormal basis for every $J\in \Jcal_n$, as mentioned in the beginning of the proof, we get that there is an $A\in O(n,n)$ such that $J=c(A,J_S)$. Moreover, it is easy to see that the stabilizer of $J_S$ under $c$ is precisely $U(k,k)$. Indeed, this is one way to define $U(k,k)$. The induced map 
\begin{equation*}
    i:\frac{O(n,n)}{U(k,k)}\longrightarrow \mathrm{End}(\R^{2n})
\end{equation*}
is a smooth injective immersion, see \cite[Corollary 3.8.10]{gauge}. Moreover, $\Jcal_n\subset \mathrm{End}(\R^{2n})$ is the preimage of $-\mathrm{Id}$ under the smooth map $\mathrm{End}(\R^{2n})\ni F\mapsto F^2\in \mathrm{End}(\R^{2n})$ intersected with the closed submanifold $O(n,n)$ and is, hence, a closed subset. So, $\Jcal_n\subset \mathrm{End}(\R^{2n})$ is locally compact.
Since $O(n,n)$ acts transitively on the locally compact Hausdorff space $\Jcal_n\subset \mathrm{End}(\R^{2n})$ the map $i$ is a topological embedding, see \cite[Chapter II, Section 2.13, p.~65]{montgomery2018topological}. Since the map $i$ is an (injective) immersion and a topological embedding, it is a smooth embedding proving the claim. 
\end{proof}

\begin{lemma}
\label{Lemma:GACS-Subbundle}
Let $M$ be a manifold, $E\to M$ be a vector bundle, and $\scal{-}{-}$ be a bundle metric of signature $(2n,2n)$ with $n\in \N\setminus \{0\}$. So, in particular, the rank of $E$ is $\mathrm{rk}(E)=4n$. Then, the set $\Jcal(E,\scal{-}{-})$ is a fiber bundle with the bundle structure being induced by that of $\mathrm{End}(E)$, that is, there is a vector bundle atlas on $\mathrm{End}(E)$ such that it becomes a fiber bundle atlas of $\Jcal(E,\scal{-}{-})$ if restricted to $\Jcal(E,\scal{-}{-})$.  
\end{lemma}
\begin{proof}
Let $\mathrm{End}(E)\xrightarrow{\pi}M$ be the projection of the vector bundle. Around each $x\in M$ we can find an open subset $U\subseteq M$ such that there is a local orthonormal frame of $E$ with respect to $\scal{-}{-}$ over $U$. Let the induced trivialization be denoted by
$\Psi:\pi^{-1}(U)\mapsto U\times \mathrm{End}(\R^{4n})$. Then $\Psi(\pi^{-1}(U)\cap \Jcal(E,\scal{-}{-}))=U\times \Jcal_{2n}$, where $\Jcal_{2n}$ is as in Lemma \ref{Lemma:LinGACSSubmanifold}, which we had shown to be an embedded submanifold of $\mathrm{End}(\R^{4n})$. So we see that $\Jcal(E,\scal{-}{-})$ is a submanifold of $\mathrm{End}(E)$ and in fact a fiber bundle with typical fiber $\Jcal_{2n}$ with the fiber bundle structure induced by that of $\mathrm{End}(E)$. 
\end{proof}

The following proposition is a specific application of \cite[Lemma 3.2.11]{RWang} where Wang shows that if $M$ is a compact manifold, $p:B\to M$ is a surjective submersion, and $A\subseteq B$ is a subbundle, that is, $A$ is a submanifold of $B$ such that $\left.p\right|_A:A\to M$ is a surjective submersion, then $\Gamma(A)\subseteq \Gamma(B)$ is a zero-tame submanifold.
\begin{proposition}
\label{Proposition:GACSZeroTameSubmfd}
Let $M$ be a compact manifold, $E\to M$ be a vector bundle, and $\scal{-}{-}$ be a bundle metric on $E$ of signature $(2n,2n)$, $n\in \N\setminus\{0\}$. Then the set of generalized almost complex structures $\Gamma(\Jcal(E,\scal{-}{-}))$ is a zero-tame submanifold of $\Gamma(\mathrm{End}(E))$.
\end{proposition}
\begin{proof}
By Theorem \ref{FiberBundleSection} it is clear that $\Gamma(\Jcal(E,\scal{-}{-}))$ is a zero-tame Fréchet manifold.
The claim then immediately follows from Lemma \ref{Lemma:GACS-Subbundle} and \cite[Lemma 3.2.11]{RWang}.
\end{proof}

\begin{lemma}
\label{Lemma:IsoClassesOfGACS}
Let $M$ be a compact manifold, let $E_i\to M$ be vector bundles, and let $\scal{-}{-}_i$ be a nondegenerate bundle metric on $E_i$, $i=1,2$ of signature $(2n,2n)$, $n\in \N\setminus\{0\}$. Let $\psi:E_1\to E_2$ be a vector bundle isomorphism with $\psi^*\scal{-}{-}_2=\scal{-}{-}_1$. Then the induced map 
\begin{align*}
    \Psi:\Gamma(\Jcal(E_1,\scal{-}{-}_1))&\longrightarrow \Gamma(\Jcal(E_2,\scal{-}{-}_2))\\
    J&\longmapsto \psi\circ J\circ \psi^{-1}
\end{align*}
is a zero-tame diffeomorphism. 
\end{lemma}
\begin{proof}
Let $J\in \Gamma(\Jcal(E_1,\scal{-}{-}_1)$. Then we clearly have $\Psi(J)^2=-\mathrm{Id}_{E_2}$. Moreover, we have $\scal{\psi\circ J\circ\psi^{-1}u}{\psi\circ J\circ\psi^{-1}v}_2=\scal{J\circ\psi^{-1}u}{ J\circ\psi^{-1}v}_1=\scal{\psi^{-1}u}{\psi^{-1}v}_1=\scal{u}{v}_2$ and hence $\Psi(J)\in \Gamma(\Jcal(E_2,\scal{-}{-}_2))$.

The map $\Jcal(E_1,\scal{-}{-}_1)\ni J\mapsto \psi_{\pi(J)}\circ J\circ \psi^{-1}_{\pi(J)}\in \Jcal(E_2,\scal{-}{-}_2)$ is a smooth fiber-preserving map. Here, $\pi$ denotes the bundle projection map of $\Jcal(E_1,\scal{-}{-}_1)\xrightarrow{\pi} M$. So, $\Psi$ is a bundle operator (cf.~Convention \ref{Conv:FiberBundleOperators}) and hence zero-tame (cf.~Proposition \ref{Prop:FiberBundleOperators}). It is clear that $\Psi^{-1}$ is induced by conjugation with $\psi^{-1}$ and is, by the same arguments, also zero-tame.
\end{proof}

\begin{lemma}
Let $M$ be a compact manifold. Then the map 
\begin{align*}
    c_\Omega:\Forms{2}\times \Gamma(\mathrm{End}(\GenTM))&\longrightarrow \Gamma(\mathrm{End}(\GenTM))\\
    (B,F)&\longmapsto e^B\circ F\circ e^{-B} 
\end{align*}
and the map
\begin{align*}
    c_D:\Diff\times \Gamma(\mathrm{End}(\GenTM))&\longrightarrow \Gamma(\mathrm{End}(\GenTM))\\
    (f,F)&\longmapsto f_*\circ F\circ f_*^{-1}
\end{align*}
are smooth tame.
\end{lemma}
\begin{proof}
The fact that the map $c_\Omega$ is smooth tame follows directly from it being a nonlinear vector bundle operator up to the zero-tame identification 
$\Forms{2}\times \Gamma(\mathrm{End}(\GenTM))\cong \Gamma((\bigwedge^2T^*M)\oplus\mathrm{End}(\GenTM))$ (cf.~Proposition \ref{Prop:EplusFsect}).

Let us now consider $c_D$. We note that $\mathrm{End}(\GenTM)\cong (TM\oplus T^*M)\tens (TM\oplus T^*M)^*$. Again by Proposition \ref{Prop:EplusFsect}, this induces a zero-tame isomorphism
$\Gamma(\mathrm{End}(\GenTM))\cong \tensf{1}{1}\times \tensf{2}{0}\times \tensf{0}{2}\times \tensf{1}{1}$ with $\tensf{n}{m}$ as in Convention \ref{Conv:tensfields}. So, it suffices to check how $c_D$ acts on these components. Let $J\in \Gamma(\mathrm{End}(E))$ correspond to
\begin{align*}
    \begin{pmatrix}
        J_1 & J_2\\
        J_3 & J_4
\end{pmatrix}
\end{align*}
with $J_1\in \tensf{1}{1}$, $J_2\in \tensf{2}{0}$, $J_3\in \tensf{0}{2}$, and $J_4\in \tensf{1}{1}$ then a short calculation shows that $c_D(f,J)$ corresponds simply to
\begin{align*}
    \begin{pmatrix}
        f_*J_1 & f_*J_2\\
        f_*J_3 & f_*J_4
\end{pmatrix}
\end{align*}
for every $f\in \Diff$. Hence, the map $c_D$ is smooth tame as the push forward of tensor fields along diffeomorphisms is a smooth tame map (cf.~Remark \ref{Rem:PushForwDifferential}).
\end{proof}

\begin{theorem}
\label{Thm:ActionSmTame}
Let $M$ be a compact manifold and $H\in \ClForms{3}$.
Then, the canonical left group action 
\begin{align*}
    c:\Auteq\times \Gamma(\Jcal_M)&\longrightarrow \Gamma(\Jcal_M)\\
    (F,J)&\longmapsto F\circ J\circ F^{-1}
\end{align*}
is a smooth tame map.
\end{theorem}
\begin{proof}
For $F\in \Auteq$ and $J\in \Gamma(\Jcal_M)$ it is clear that $c(F,J)^2=-\mathrm{Id}$. Moreover, since $F$ and $J$ preserve $\scal{-}{-}$ we get that $c(F,J)=F\circ J\circ F^{-1}$ is orthogonal with respect to $\scal{-}{-}$ as well. So the map $c$ is well defined. 
Let $f\in \SDiff$, $B\in \ClForms{2}$, and $\Phi$ as in Convention \ref{Phi,phi}. Then, we have
\begin{align*}
    c^\Phi(f,B,J)&\coloneqq c(\Phi(f,B),J))=f_*\circ e^{B+Q(f^*H-H)}\circ J\circ e^{-B-Q(f^*H-H)}\circ f_*^{-1}\\
    &=c_D(f,c_\Omega(B+\phi(f),i_{\Jcal_M}(J)))
\end{align*}
where $c_D$ and $c_\Omega$ are defined as in the previous lemma and $i_{\Jcal_M}$ is the (zero-)tame inclusion of the (zero-)tame submanifold $\Gamma(\Jcal_M)\subset \Gamma(\mathrm{End}(\GenTM))$ (cf.~Proposition \ref{Proposition:GACSZeroTameSubmfd}).
So as a composition of smooth tame maps we have that 
$c^\Phi:\SDiff\times \ClForms{2}\times \Gamma(\Jcal_M)\to \mathrm{End}(\GenTM)$ is smooth tame. Lastly, since the image of $c^\Phi$ lies in the (zero-)tame submanifold $\Gamma(\Jcal_M)\subset \mathrm{End}(\GenTM)$ we have that $c=c^\Phi\circ (\Phi^{-1}\times \Id_{\Gamma(\Jcal_M)}):\Auteq\times \Gamma(\Jcal_M)\to \Gamma(\Jcal_M)$ is smooth tame.
\end{proof}

We note that the results of this section are summarized in Theorem \ref{MainThm:Jcal_M}.

\begin{proof}[Proof of Theorem \ref{MainThm:Jcal_M}]
    The theorem follows directly from Lemma \ref{Lemma:GACS-Subbundle}, Proposition \ref{Proposition:GACSZeroTameSubmfd}, and Theorem \ref{Thm:ActionSmTame}.
\end{proof}

\begin{remark}
Since $\Auteq$ is a tame Fréchet Lie group (cf.~Theorem \ref{Thm:AuteqTameLieGroup}) the inversion map $V:\Auteq\to \Auteq$ is smooth tame. Hence, the canonical right action $c\circ (V\times \mathrm{Id}_{\Gamma(\Jcal_M)}):\Auteq\times \Gamma(\Jcal_M)\to \Gamma(\Jcal_M)$ is smooth tame. Here, $c$ is as in the previous theorem.
\end{remark}

\begin{remark}
We note that by Lemma \ref{Lemma:IsoClassesOfGACS} and Theorem \ref{Thm:IsoClassesOfAuteq} it suffices to consider exact Courant algebroids of the type $(\GenTM,H)$ for some $H\in \ClForms{3}$ to deduce that the natural action of the autoequivalences on the generalized almost complex structures of any exact Courant algebroid is smooth tame.
\end{remark}

\begin{remark}
One way our results for the tame Lie group structure of the autoequivalences of exact Courant algebroids and their smooth tame action on the associated generalized almost complex structures can be useful is in the context of further studies and possibly generalizations of Gualtieri's deformation theorem for generalized almost complex structures for exact Courant algebroids \cite[Theorem 5.4]{gualtieri}.    
\end{remark}

\begin{appendices}

\section{Equivalent gradings on the Fréchet space of sections of a vector bundle over a compact base.}
\label{Appendix:EquivNormsOnGamma(V)}

In this appendix we show that the two different families of seminorms on the space of sections of a vector bundle over a compact manifold introduced in Theorem \ref{VBsectTameFre} are equivalent gradings. These computations where previously done as part of a preparatory project \cite{hausarbeit} for this thesis.

\begin{definition}
Let $X$ be some topological space. We call a family of norms $(\nnorm{-}{x}:\R^n\to \R)_{x\in X}$ on $\R^n$, $n\in \N$, a continuous family of norms if the map $X\times \R^n\ni (x,v)\mapsto \nnorm{v}{x}\in \R$ is continuous.

We call two families of norms $(\nnorm{-}{x}:\R^n\to \R)_{x\in X}$ and $(\nnorm{-}{x}':\R^n\to \R)_{x\in X}$ uniformly equivalent if there are constants $c>0$ and $C>0$ such that 
\begin{align*}
    c\nnorm{v}{x}'\leq \nnorm{v}{x}\leq C\nnorm{v}{x}'
\end{align*}
holds for all $v\in \R^n$ and for all $x\in X$.
\end{definition}

\begin{lemma}
\label{Lemma:ContFamOfNormUnifEquiv}
Let $K$ be some compact topological space and let $(\nnorm{-}{x})_{x\in K}$ be some continuous family of seminorms on $\R^n$, $n\in \N$. Then there are constants $c>0$ and $C>0$ such that 
\begin{align*}
    c|v|\leq \nnorm{v}{x} \leq C|v|
\end{align*}
for all $v\in \R^n$ and $x\in K$. Moreover, this shows that any two continuous families of norms on $\R^n$ with parameter in a compact topological space are uniformly equivalent.
\end{lemma}
\begin{proof}
We define $c=\inf \{\nnorm{v}{x}|~(x,v)\in K\times S^{n-1}\}$ and $C=\sup \{\nnorm{v}{x}|~(x,v)\in K\times S^{n-1}\}$. Here, $S^{n-1}$ is the unit sphere with respect to the Euclidean norm. Then $c>0$ by the compactness of $K\times S^{n-1}$ and the positive definitness of $\nnorm{-}{x}$ for all $x\in K$. Then by the definition of $c$ and $C$ we find that
\begin{align*}
    c\leq \nnorm{\frac{v}{|v|}}{x}\leq C
\end{align*} 
for every $x\in K$ and $v\in \R^n$.
Hence, $c|v|\leq \nnorm{v}{x}\leq C|v|$ for every $x\in K$ and $v\in \R^n$. 
\end{proof}

\begin{proposition}
\label{Prop:equiv of norms on Gamma(V)}
Let $M$ be a compact manifold, and let $V\to M$ be a vector bundle of rank $\mathrm{rk}(V)=r$. Let $\{\Tilde{U}_k\}_{k=1,...,N}$ be a finite open cover of $M$ by coordinate charts $(\Tilde{U}_k,\varphi_k)$, let $\{U_k\}_{k=1,...,N}$ be a refinement such that $\Bar{U}_k\subset \Tilde{U}_k$, and let $\{\theta_k\}_{k=1,...,N}$ be a partition of unity subordinate to $\{U_k\}_{k=1,...,N}$ on $M$. Over each of the $\Tilde{U}_k$ we assume that a local frame $\{e_{k,i}\}_{i=1,...,\mathrm{rk}(V)}$ of $V$ exists. We define a family of seminorms on $\Gamma(V)$ by 
\begin{equation*}
    \nnorm{\sigma}{n}'=\sum_{k=1}^N\sum_{i=1}^r\sup_{|\alpha|\leq n}\sup_{x\in M}|D^{\alpha}(\theta_k\sigma^i)|
\end{equation*}
for $n\in \N$ where we abbreviate $\theta_k\sigma^i\coloneqq (\theta_k\sigma^{k,i})\circ\varphi_k^{-1}$ where $\sigma^{k,i}$ are the coefficients with respect to $e_{k,i}$.

Furthermore, choose covariant derivatives and positive definite bundle metrics on $V$ and $T^*M$. Let $\nabla^j\sigma\in \Gamma((T^*M)^{\tens j}\tens V)$ denote the induced $j$-th covariant derivative of $\sigma\in \Gamma(V)$, $j\in \N$, and let $|-|$ denote the induced fiberwise norms on $(T^*M)^{\tens k}\tens V$, $k\in \N$, as in Example \ref{Ex:SmoothSectionsFrechetSpace}. Then we define another grading on $\Gamma(V)$ by the seminorms
\begin{equation*}
    \nnorm{\sigma}{n}=\sum_{j=0}^n\sup_{x\in M}|\nabla^j\sigma(x)|
\end{equation*}
for $n\in \N$. We had already discussed that this family of seminorms induces a Fréchet topology on $\Gamma(V)$ (cf.~Example 
\ref{Ex:SmoothSectionsFrechetSpace}) so the term grading is appropriate.

Then, $\nnorm{-}{n}$ and $\nnorm{-}{n}'$ are equivalent as norms for each $n\in \N$ and hence $(\nnorm{-}{n}')_{n\in \N}$ is a grading on $\Gamma(V)$ which is tamely equivalent of degree and base $0$ to $(\nnorm{-}{n})_{n\in \N}$.
\end{proposition}
\begin{proof}
We will denote $e_{k,i}$ by $e_i$ where it will be clear from context on which $U_k$ we are working. For $j\in \N$ we will furthermore regularly use that on $\Tilde{U}_k$ we have
\begin{equation*}
    \nabla^{j}\sigma = \sum_{|\alpha|=j}\sum_{i}D^{\alpha}\sigma^i \d x^{\alpha}\tens e_i + l.o.t.
\end{equation*}
with $\d x^{\alpha}=\d x^{\alpha_1}\tens...\tens \d x^{\alpha_j}$ and "$l.o.t.$" stands for lower order terms with respect to $D^{\alpha}$. Here, $\alpha$ denotes a multi-index $\alpha_1\ldots\alpha_j$ with $\alpha_l=1,\ldots,\dim(M)$, $l=1,\ldots,j$. 
Now we see that
\begin{equation*}
    \nnorm{\sigma}{n}= \sum_{j=0}^n\sup_{x\in M}|\sum_{k=1}^N\nabla^j(\theta_k\sigma)(x)| \leq \sum_{j=0}^n\sum_{k=1}^N\sup_{x\in M}|\nabla^j(\theta_k\sigma)(x)|.
\end{equation*}
We can now use that $U_k$ is relatively compact since $M$ is compact to estimate the coefficients of $|\nabla^j(\theta_k\sigma)|$ on $U_k$ from above. This leads to the following result
\begin{align*}
    \sum_{j=0}^n\sum_{k=1}^N\sup_{x\in M}|\nabla^j(\theta_k\sigma)(x)|&\leq C \sum_{j=0}^n\sum_{k=1}^N\sup_{x\in M}\sum_{|\alpha|\leq j}\sum_{i=1}^r|D^{\alpha}(\theta_k\sigma^i)(x)| \\
    &\leq C\sum_{k=1}^N\sum_{i=1}^r \sup_{|\alpha|\leq n}\sup_{x\in M}|D^{\alpha}(\theta_k\sigma^i)(x)|=C\nnorm{\sigma}{n}'
\end{align*}
where we have used that the open cover of $M$ is finite. Thus, we have $\nnorm{\sigma}{n}\leq C\nnorm{\sigma}{n}'$ for some constant $C>0$. This shows that the identity mapping $\mathrm{Id}:(\Gamma(V),(\nnorm{-}{n})_{n\in \N}')\to (\Gamma(V),(\nnorm{-}{n})_{n\in \N})$ is continuous and thus we already know that both gradings induce the same topology by the open mapping theorem (cf.~Theorem \ref{Thm:OpenMapping}), quietly assuming of course that both gradings induce a complete topology.

We wish, however, to show equivalence on the level of gradings. In order to do this we will prove that $\nnorm{\sigma}{n}'\leq C\nnorm{\sigma}{n}$ using induction. In the case that $n=0$ we have
\begin{equation*}
    \nnorm{\sigma}{0}'=\sum_{k=1}^N\sum_{i=1}^r\sup_{x\in M}|\theta_k\sigma^i|\leq \sum_{k=1}^N\sum_{i=1}^r \sup_{x\in U_k}|\sigma^i|\leq C \sum_{k=1}^N \sup_{x\in U_k}\sum_{i=1}^r|\sigma^i|
\end{equation*}
where we have used that all summands are positive and the open cover is finite to pull the sum into the supremum.
Now on the compact sets $\bar U_k$ we may relate the euclidean norm with the norm induced by our choice of scalar products on $V$ and $T^*M$ by constants $C_k>0$ (cf.~Lemma \ref{Lemma:ContFamOfNormUnifEquiv}) such that for all $x\in U_k$
\begin{equation*}
    \sum_{i=1}^r|\sigma^i(x)|\leq r \sqrt{\sum_{i=1}^r |\sigma^i(x)|^2}\leq C_k|\sigma(x)|.
\end{equation*}
Note that surely $|\sigma^i(x)|\leq \sqrt{\sum_{i=1}^r |\sigma^i(x)|^2}$ for $i=1,\ldots,r$.
Since the open cover is finite, there is some $C>0$ such that
\begin{equation*}
    \sum_{k=1}^N \sup_{x\in U_k}\sum_{i=1}^r |\sigma^i|\leq \sum_{k=1}^N\sup_{x\in U_k}C_k|\sigma|\leq C \sup_{x\in M}|\sigma|= C\nnorm{\sigma}{0}.
\end{equation*}
We continue now with the case $n\geq1$ and once more employ
\begin{equation*}
    \sum_{|\alpha|=j}\sum_{i=1}^rD^{\alpha}\theta_k\sigma^i\d x^{\alpha}\tens e_i = \nabla^j\theta_k\sigma + l.o.t.
\end{equation*}
where the lower order terms are with respect to $D^\alpha$ not $\nabla$. We will refer to the left hand side of this equation as $L_{j,k}$ and equivalently $R_{j,k}$ for the right hand side. As before we choose a constant $C>0$ relating the norm $|-|$ induced by the choices of scalar products on $V$ and $T^*M$ to the euclidean norm with respect to $\d x^{\alpha}\tens e_i$ on $\Bar{U}_k$. In particular, there is a constant $C>0$ such that
\begin{equation*}
    |L_{j,k}|\geq C \sqrt{\sum_{i=1}^r\sum_{|\alpha|=j}|D^{\alpha}\theta_k\sigma^i|^2}.
\end{equation*}
Now if we apply the supremum over $M$ and use that all summands are positive and all sums are finite we get
\begin{align*}
    \sum_{i=1}^r\sum_{|\alpha|=j}\sup_{x\in M}|(D^{\alpha}\theta_k\sigma^i)(x)|\leq& C\sup_{x\in M}  \sum_{i=1}^r\sum_{|\alpha|=j} |(D^{\alpha}\theta_k\sigma^i)(x)|\\
    \leq &C\sup_{x\in M}\sqrt{\sum_{i=1}^r\sum_{|\alpha|=j}|(D^{\alpha}\theta_k\sigma^i)(x)|^2}
\end{align*}
for some constants $C>0$.
Now we can simply observe that
\begin{equation*}
    \nnorm{\sigma}{n}'\leq C \sum_{j=0}^n\sum_{k=1}^N \sup_{x\in M}|L_{j,k}(x)|.
\end{equation*}
For the right hand side we can simply observe $|R_{j,k}|\leq |\nabla^j\theta_k \sigma|+|l.o.t.|$. By estimating coefficients from above and adding (positive) terms if necessary we surely may estimate the lower order terms with respect to $D$ by
\begin{equation*}
    |l.o.t.|\leq C \sum_{i=1}^r\sum_{k=1}^N\sum_{|\alpha|<j}\sup_{x\in M}|(D^{\alpha}\theta_k\sigma^i)(x)|=C\nnorm{\sigma}{j-1}'\leq C\nnorm{\sigma}{j-1}
\end{equation*}
where in the last step we use the induction hypothesis. Now we may also estimate the other term using the product rule 
\begin{equation}
    \label{est_tens1}
    |\nabla^j\theta_k\sigma|\leq C( |\nabla^j \sigma| + |(\nabla \theta_k)\tens (\nabla^{j-1}\sigma)|+|(\nabla^2 \theta_k)\tens (\nabla^{j-2}\sigma)|+...)
\end{equation}
However, by the definition of $|-|$ we have
\begin{equation}
    \label{est_tens2}
    |(\nabla^l \theta_k)\tens (\nabla^{j-l}\sigma)|=|\nabla^l \theta_k||\nabla^{j-l}\sigma|\leq C|\nabla^{j-l}\sigma|\leq C\sup_{x\in M}|\nabla^{j-l}\sigma|
\end{equation}
for $0\leq l\leq j$ where we used that $\theta_k$ has compact support. So we get
\begin{equation*}
    |R_{j,k}|\leq |\nabla^j\sigma|+C\nnorm{\sigma}{j-1}
\end{equation*}
where we have absorbed the additional terms in equation \eqref{est_tens1} into the $\nnorm{\sigma}{j-1}$ term using equation \eqref{est_tens2}.
Thus using that $\nnorm{\sigma}{j-1}\leq \nnorm{\sigma}{n}$ for all $j\leq n+1$ in particular for all $j\leq n$, using that the open cover of $M$ is finite, and simply applying the definition of $\nnorm{-}{n}$ we get that 
\begin{equation*}
    \nnorm{\sigma}{n}'\leq C \sum_{j=0}^n\sum_{k=1}^N \sup_{x\in M}|L_{j,k}(x)|=C \sum_{j=0}^n\sum_{k=1}^N \sup_{x\in M}|R_{j,k}(x)|\leq C\nnorm{\sigma}{n}
\end{equation*}
finishing the proof. 
\end{proof}
\end{appendices}

\singlespacing\small
\bibliographystyle{alphaurl}
\bibliography{Quellen}

\end{document}